\documentclass[11pt]{article}
\usepackage[T1]{fontenc}
\usepackage{mathptmx}
\usepackage{amssymb, amsmath, amsthm}
\usepackage{hyperref}
\usepackage{stmaryrd}
\usepackage{tikz-cd}
\usepackage{graphicx}
\usepackage{comment}
\usepackage[all]{xy}

\DeclareRobustCommand{\coprod}{\mathop{\text{\fakecoprod}}}
\newcommand{\fakecoprod}{%
  \sbox0{$\prod$}%
  \smash{\raisebox{\dimexpr.9625\depth-\dp0}{\scalebox{1}[-1]{$\prod$}}}%
  \vphantom{$\prod$}%
}

\newcommand{\thetaLaurent}{\theta_{\mathrm{L}}}
\newcommand{\mr}[1]{\mathrm{#1}}
\newcommand{\mf}[1]{\mathfrak{#1}}
\newcommand{\mc}[1]{\mathcal{#1}}
\newcommand{\mb}[1]{\mathbb{#1}}

\newcommand{\ms}[1]{\mathsf{#1}}
\newcommand{\abcdmat}{\smatrix{a&b\\c&d}}
\newcommand{\Z}{\mb{Z}}
\newcommand{\newE}{e}
\newcommand{\newf}{\delta}
\newcommand{\Q}{\mb{Q}}

\newcommand{\R}{\mb{R}}
\newcommand{\C}{\mb{C}}
\newcommand{\Chains}{\mathrm{Chains}}
\newcommand{\varK}{\mathsf{K}}
\newcommand{\barK}{\overline{\mathsf{K}}}
\newcommand{\varC}{\mathsf{k}}
\newcommand{\Lie}{\mathrm{Lie}}
\newcommand{\barC}{\overline{\mathsf{k}}}
\newcommand{\Cones}{\mathsf{Ch}}
\newcommand{\Rays}{S^1}
\newcommand{\Symb}{\mathsf{Symb}}
\newcommand{\varleft}{\langle}
\newcommand{\varright}{\rangle}
\newcommand{\zp}{\mb{Z}_p}
\newcommand{\Gm}{\mathbb{G}_m}
\newcommand{\A}{\mathrm{supp}}

\newcommand{\qp}{\mb{Q}_p}

\newcommand{\G}{\mb{G}}
\newcommand{\gm}{\G_m}

\newcommand{\F}{\mb{F}}
\newcommand{\varKs}{\varK_2(N)}

\newcommand{\bGamma}{\tilde{\Gamma}}
\newcommand{\ifs}{\text{if }}
\newcommand{\vgamma}{\vec{\gamma}}

\newcommand{\inv}{\dagger}

\newcommand{\et}{\mr{\acute{e}t}}
\newcommand{\id}{\mr{id}}
\newcommand{\SL}{\mathrm{SL}}
\newcommand{\GL}{\mathrm{GL}}

\newcommand{\ord}{\mathrm{ord}}

\newcommand{\ps}[1]{\llbracket #1 \rrbracket}

\newcommand{\smatrix}[1]{\left(\begin{smallmatrix}#1\end{smallmatrix}\right)}
\newcommand{\Pmatrix}[1]{\begin{pmatrix}#1\end{pmatrix}}

\DeclareMathOperator{\Hom}{Hom} 
\DeclareMathOperator{\End}{End} \DeclareMathOperator{\Gal}{Gal}
 
\DeclareMathOperator{\coker}{coker} 
 
\DeclareMathOperator{\im}{im} 
\DeclareMathOperator{\Ext}{Ext} \DeclareMathOperator{\cha}{char}
 
\DeclareMathOperator{\Spec}{Spec} \DeclareMathOperator{\Tr}{Tr}
\DeclareMathOperator{\rank}{rank}
\DeclareMathOperator{\Norm}{N}

\newtheorem{theorem}{Theorem}[subsection]
\newtheorem{proposition}[theorem]{Proposition}
\newtheorem{lemma}[theorem]{Lemma}
\newtheorem{corollary}[theorem]{Corollary}
\newtheorem*{thm}{Theorem}

\newtheorem{conjecture}[theorem]{Conjecture}

\theoremstyle{definition}

\theoremstyle{remark}
\newtheorem{remark}[theorem]{Remark}
\newtheorem*{ack}{Acknowledgments}
\newtheorem{example}[theorem]{Example}

\newcounter{countii}

\numberwithin{equation}{section}

\begin{document}

\title{Eisenstein cocycles in motivic cohomology}
\author{Romyar Sharifi and Akshay Venkatesh}
\date{}
\maketitle

\begin{abstract}
	Several authors have studied  homomorphisms from first homology groups of  modular curves
	to $K_2(X)$, with $X$ either a cyclotomic ring or a modular curve.   These maps
	send Manin symbols in the homology groups to 
	Steinberg symbols of cyclotomic or Siegel units.
	 We give a new construction of these maps and a direct proof of their Hecke equivariance,  
	 analogous to the construction of Siegel units using the universal elliptic curve. Our main tool is 
	 a $1$-cocycle  from $\GL_2(\Z)$   to the second $K$-group
	 of the function field
	 of a suitable group scheme over $X$, from which the maps of interest arise by specialization.
\end{abstract}

\setcounter{tocdepth}{1}
\tableofcontents

\newpage

\section{Introduction}

For a positive integer $N$, let $Y_1(N)$ and $X_1(N)$ denote the usual open and closed modular curves over $\Q$. 
In this paper, we provide a new perspective on two  homomorphisms from the integral homology of the $\C$-points of $X_1(N)$  
to second $K$-groups of the cyclotomic integer ring $\Z[\mu_N]$ and the modular curve $Y_1(N)$:\footnote{We  invert $5$ not to	construct the map, but to prove equivariance with respect to the $5$th Hecke operator.}   \begin{eqnarray} \label{cyclmap}
	&\Pi_N \colon   \mbox{integral homology of $X_1(N)_{/\C}$} \to  K_2(\Z[\mu_N])  [\tfrac{1}{2}]\\
	 \label{zetamap}
	&z_N \colon  \mbox{integral homology of $X_1(N)_{/\C}$}  \to K_2(X_1(N)) [\tfrac{1}{30N}].
\end{eqnarray} 
The map $\Pi_N$ was defined explicitly on slightly larger groups by Busuioc \cite{busuioc} and the first author \cite{sharifi}.  
The map $z_N$ was given an explicit construction in a recent preprint of Brunault \cite{brunault-K4}, following earlier constructions of
Goncharov \cite{goncharov} and Brunault \cite{brunault} of an analogous map to $z_N \otimes \Q$ for $Y(N)$.
The $p$-adic realization of $z_N$ for $p \mid N$ was constructed by Fukaya and Kato in their study \cite{fk} of a conjecture of the first author \cite{sharifi}.  
Most of these constructions boil down to the remarkable fact that Steinberg symbols of cyclotomic or Siegel units satisfy relations parallel to the very simple relations satisfied by Manin symbols
(although Fukaya and Kato use norm relations among Beilinson-Kato elements and a $p$-adic regulator computation); see \S \ref{background} for more.  

Our construction is different,  and is analogous to the construction of Siegel units on $Y_1(N)$.  
 Let us specialize to the $\Pi_N$-case  for a moment to give the idea of our construction,
 postponing a more careful discussion to \S \ref{our_approach}.
 Siegel units are  pullbacks  by an $N$-torsion section of theta-functions on the universal elliptic curve over $Y_1(N)$;  these theta-functions are uniquely specified by their poles. In our situation, 
the role of the theta function is played by a ``big'' $1$-cocycle $\Theta$ on $\GL_2(\Z)$ that is valued in (a quotient of) $K_2$ of the function field of $\mathbb{G}_m^2$. This $\Theta$ is again characterized by its ``poles'', i.e., its image under residue maps to $K_1$ of function fields of divisors on $\gm^2$.
We then pull its restriction to $\Gamma_0(N)$ back via a torsion point on $\mb{G}_m^2$ to obtain a cocycle
$$
	\Theta_N \colon \Gamma_0(N) \to K_2(\Q(\mu_N)).
$$
which underlies $\Pi_N$ described above. 

    The construction of the map from $\Gamma_1(N)$ to $K_2(Y_1(N))$ is similar but
  the role of $\mathbb{G}_m^2$ is replaced throughout by the square $E^2$ of an elliptic curve, and then $E$ is varied over the moduli
  space of elliptic curves.
Because the ``big'' cocycle $\Theta$ is characterized by its poles, it is easy to analyze. In contrast,
the specialized cocycle $\Theta_N$ cannot be so analyzed (it has residues only at primes above $N$, and these carry very little information).
  
  In particular, we are able to prove the following (see Theorems \ref{varpi} and \ref{zeta_map} for details): 
 
\begin{thm}
	The map $\Pi_N$ is Eisenstein with respect to the prime-to-$N$ Hecke operators.
\end{thm}

\begin{thm}
	The map $z_N$ is equivariant for the prime-to-$N$ Hecke operators.
\end{thm}

These results may be considered in the context of a body of results  
that suggest close relationships between homology of arithmetic groups
and $K$-groups of algebraic varieties: see for instance
\cite{fks-survey,gon-preprint,stevens,venkatesh}. 
Most relevant to our paper 
is the work of the first author suggesting that the map induced by $\Pi_N$ on an Eisenstein quotient of homology is an isomorphism to $K_2(\Z[\mu_N])^+$ away from $2$-parts; 
see Conjecture \ref{eisconj} for details.   

\subsection{Background on the maps} \label{background}
We describe in more detail some of the forms of the maps $\Pi_N$ and $z_N$ that have appeared in the literature. 
The map $\Pi_N$ is most easily defined on a larger homology group
relative to the ``non-infinity cusps'' $C_1^{\circ}(N)$, which are those that do not lie over the infinity cusp of the modular curve $X_0(N)$.  That is, the map $\Pi_N$ is the restriction of a map
$$
	\Pi_N^{\circ} \colon H_1(X_1(N),C_1^{\circ}(N),\Z) \to K_2(\Z[\mu_N,\tfrac{1}{N}]) \otimes_{\Z} \Z[\tfrac{1}{2}]
$$
taking image in the slightly larger second $K$-group of the $N$-integers of $\Q(\mu_N)$.

The integral homology relative to the cusps is generated by certain classes $[u:v]$ of geodesics between cusps known as Manin symbols, where $(u,v)$ is a pair of relatively prime integers modulo $N$.\footnote{This generation is a consequence of the fact that $\Z$ is a Euclidean ring. For purposes of generalization, our more abstract approach to the construction of analogues of $\Pi_N$ should therefore prove useful.} Those Manin symbols for which both $u$ and $v$ are nonzero generate the homology relative to the non-infinty cusps.\footnote{See \cite[3.3.7]{fk}, but note that our convention for Manin symbols is the standard one, which is to say that it differs from that of \cite{sharifi} and \cite{fk} by application of an Atkin-Lehner involution. This accounts for the differences with those papers in our description.}  
The map $\Pi_N^{\circ}$ was defined in \cite{busuioc, sharifi} to send each such Manin symbol to a Steinberg symbol of cyclotomic $N$-units in $\Q(\mu_N)$:
\begin{equation} \label{explicit_map}
	\Pi_N^{\circ}([u:v]) = \{1-\zeta_N^u,1-\zeta_N^v\},
\end{equation}
where $\zeta_N$ is a primitive $N$th root of unity.
The Manin symbols satisfy very simple relations, and to show this map is well-defined
is to verify that the relations hold at the level of Steinberg symbols, which results from the usual symbol formula $\{x,1-x\} = 0$ for $N$-units $x$ and $1-x$.

In \cite{sharifi}, the first author conjectured that the $p$-adic realization of $\Pi_N$ (i.e., its tensor product with $\zp$, which we will denote by the same notation) for $p$ dividing $N$ is Eisenstein in the sense that for primes $\ell \nmid N$, one has
\begin{equation} \label{map_eisenstein}
	\Pi_N(T_{\ell} x) = (\ell + \sigma_{\ell})\Pi_N(x)
\end{equation}
for $x \in H_1(X_1(N),\Z_p)$, where $T_{\ell}$ is the $\ell$th Hecke operator and $\sigma_{\ell} \in \Gal(\Q(\mu_N)/\Q)$ is the arithmetic Frobenius at $\ell$.  
For primes $\ell \mid N$, he also conjectured that $\Pi_N(U_{\ell}^*x) = \Pi_N(x)$, where $U_{\ell}^*$ is the $\ell$th adjoint Hecke operator. 

Fukaya and Kato proved this conjecture in \cite{fk} by exhibiting $\Pi_N$ as a specialization at the infinity cusp of the $p$-adic realization of $z_N$.\footnote{The idea of composing a rational version of $z_N$ with a specialization at $\infty$ is also found in \cite[Section 3]{goncharov}.}  
Roughly speaking, their map $z_N$ is also the restriction of a map on relative homology sending $[u:v]$ to a Steinberg symbol $\{g_{\frac{u}{N}},g_{\frac{v}{N}}\}$ of Siegel units on $Y_1(N)$. 
Via a regulator computation, they show that the $p$-adic realization of $z_N$ is Hecke-equivariant for the operators $T_{\ell}$ for $\ell \nmid N$ and $U_{\ell}^*$ for $\ell \mid N$ 
 and they then use the fact that the specialization at infinity map is Eisenstein.\footnote{Actually, they prove that the specialization at infinity map is Eisenstein for the prime-to-level operators and also for the remaining operators when applied to the Beilinson-Kato elements in question.}   The first author has frequently expressed a tentative expectation that the Eisenstein property should hold without passing to the $p$-adic realization.  

Here, we give a construction of the maps $\Pi_N$ and $z_N$ without recourse to explicit symbols or regulator computations.\footnote{In fact, we do not show that our map $z_N$ satisfies the expected explicit formula. Rather, we show it holds in the quotient by a group that dies in any standard realization and which is an artifact of making the construction independent of an auxiliary integer.}  As mentioned earlier, this also allows us to prove that \eqref{map_eisenstein} holds for all $\ell \nmid N$ without tensoring with $\zp$. Unlike in the work of Fukaya and Kato, we do not use the Hecke equivariance of $z_N$ to study the Eisenstein property of $\Pi_N$. Rather, we consider these maps entirely separately.

\subsection{Our approach} \label{our_approach}

As we have mentioned, our goal in this paper is to provide an alternate construction of the maps $\Pi_N$ and $z_N$ that is analogous to the construction of Siegel units on $Y_1(N)$ via theta functions on the universal elliptic curve $\mc{E}$ over $Y_1(N)$.   We now  describe this approach in more detail. 

Recall from \cite[Proposition 1.3]{kato} that given a positive integer $n$ prime to $6N$, there is a theta-function ${}_n \theta$ in $\Q(\mc{E})^{\times}$ that is a unit outside of the $n$-torsion, and which is uniquely specified by the properties that its divisor is $n^2(0) - \mc{E}[n]$ and that it is invariant under norm maps attached to multiplication by positive integers prime to $n$.  Siegel units are obtained by pulling back the theta function ${}_n \theta$ to $Y_1(N)$ using $N$-torsion sections.  Though these Siegel units depend upon $n$,
they satisfy a distribution relation that permits one to construct a ``$n=1$'' unit,  upon inverting $6N$.

The analogues of theta functions in our work are parabolic $1$-cocycles on $\GL_2(\Z)$, again valued in second $K$-groups, but of the function fields of the squares of the multiplicative group $\gm$ over $\Q$ and the universal elliptic curve $\mc{E}$ over $Y_1(N)$. That is, the first is a $1$-cocycle
\begin{equation} \label{bigcocycGm}
	\Theta \colon \GL_2(\Z) \to K_2(\Q(\gm^2))/\langle \{-z_1,-z_2\} \rangle,
\end{equation}
where $\GL_2(\Z)$ acts on the $K$-group via pullback of its right-multiplication action on $\gm^2$, and
where $z_i$ denotes the $i$th coordinate function on $\gm^2$ (cf. Proposition \ref{existence}).  
The second is a family of $1$-cocycles
\begin{equation} \label{bigcocycE2}
	{}_n \Theta \colon \GL_2(\Z) \to K_2(\Q(\mc{E}^2)) \otimes_{\Z} \Z[\tfrac{1}{30}]
\end{equation}
depending upon a choice of prime $n \nmid N$. Using $N$-torsion sections,
we pull back the restrictions of these ``big'' cocycles on $\Gamma_1(N)$ to obtain $\Pi_N$  and a map ${}_n z_N$ 
depending on $n$ (which we make explicit only at the level of cocycles).  As with Siegel units, upon further inverting $N$, we obtain a map $z_N$ 
that may be understood as the $n=1$ analogue of the maps ${}_n z_N$.

Because of the characterization of our big cocycles in terms of their residues, it is easy to provide explicit formulas for 
and  analyze how Hecke operators act on them. 
In particular, the compatibility of the classes of these $1$-cocycles with the actions of Hecke operators is verified directly using the 
equivariance of residue maps for integral matrices of nonzero determinant. The analogous properties of the specialized cocycles follow from the analogous formulas for
the big cocycles. 

\subsubsection{Construction of ``big'' cocycles}

The big cocycles are constructed using three-term  motivic complexes. These play the roles of the two-term complex given by the divisor map in the construction of theta-functions.  Let us describe this in more detail.  Taking $G$ to be $\gm$ or $\mc{E}$ in the respective cases, the ``motivic complexes'' are homological complexes in degrees $2$, $1$, and $0$ of the form 
\begin{equation} \label{Crucial}    
	K_2( \Q(G^2)) \xrightarrow{\partial_2} \bigoplus_{D} K_1(\Q(D)) \xrightarrow{\partial_1} \bigoplus_{x} K_0(k(x)), 
\end{equation}
where the maps are residue maps, and $D$ and $x$ vary respectively over irreducible divisors and closed points of $G^2$.  The first map is given on symbols by the tame symbol, and the second map sends an element of $K_1(\Q(D)) = \Q(D)^{\times}$ to its divisor.
These complexes carry an action of the monoid $\Delta$ of integral $2$-by-$2$ matrices with nonzero determinant via pullback under
the endomorphism of right multiplication. They then also have trace maps with respect to multiplication by positive integers.

Much as a theta function is uniquely determined by its ``poles'', or more specifically, its norm-invariant divisor, our cocycles are uniquely determined by choices of a trace-fixed $\GL_2(\Z)$-invariant element $Z$ of $\bigoplus_x K_0(k(x))$, which is to say a formal $\Z$-linear 
sum of points on $G^2$. 

More specifically, given a suitable choice of $Z$ as above in the image of $\partial_1$, we choose a lift 
$$
	\eta \in \bigoplus_D K_1(\Q(D))
$$
of $Z$.  For $\gamma \in \GL_2(\Z)$, we show that $\gamma \eta - \eta \in \im \partial_2$, so there is a unique element  
\begin{equation} \label{abstractcocycG2}
	\Theta^Z_{\gamma} \in K_2(\Q(G^2))/\ker \partial_2,
\end{equation}
with residue $\gamma \eta- \eta$,  
and the recipe $\gamma \mapsto \Theta^Z_{\gamma}$ defines a ``big'' cocycle $\Theta^Z$ on $\GL_2(\Z)$.  Its cohomology class depends upon the choice of $Z$ but not the choice of $\eta$ (cf.~Proposition \ref{abstractcocyc}).

In the case that $G = \gm$, the complex \eqref{Crucial} is left exact, and the kernel of $\partial_2$ is identified with $H^2(\gm^2,2)$.  For
$x_0$ the identity in $\gm^2$, we choose $Z$ to be the class $e$ of the identity element of the $\GL_2(\Z)$-fixed subgroup $K_0(k(x_0)) \cong \Z$ of $\bigoplus_x K_0(k(x))$. We choose $\eta$ to be the class of $1-z_1^{-1}$ on the rank $1$ subtorus defined by $z_2 = 1$,
though as mentioned the class of $\Theta = \Theta^e$ is independent of this choice. In fact, since we take $\eta$ to be trace fixed,
the ambiguity inherent in taking the quotient of $K_2(\Q(\gm^2))$ by $\ker \partial_2 = H^2(\gm^2,2)$ can be further reduced to its
trace-invariant part, which is generated by $\{-z_1,-z_2\}$.

In the case $G = \mc{E}$, the homology of the motivic complex \eqref{Crucial} does not vanish anywhere, but if we restrict to its trace-invariant part, then, at least upon inverting $6$, it is right exact and the image of the residue map $\partial_1$ is the kernel of the degree map $\bigoplus_x K_0(k(x)) \to \Z$. Since there is no meromorphic function on an elliptic curve
 whose divisor is supported at the origin, the role of $1 \in K_0(k(x_0))$
 in the above construction must be replaced by a slightly less canonically chosen trace-invariant and $\GL_2(\Z)$-fixed element $e_n$ that is
 known to be in the image of $\partial_1$, its choice depending on an auxiliary prime $n \nmid N$.

 \begin{remark}[Toric geometry perspective]
    We also provide an alternate point of view on the cocycle $\Theta$ in the $\gm$-case that is tied to toric geometry
    and which allows us to reduce the ambiguity in $\Theta$ up to torsion of small order. 
    As observed by Brion \cite{brion},
    the function that sends a rational cone $C \subseteq \R^2$ to the  generating function
    $$
    	\phi(C) =\sum_{(m,n) \in \Z^2 \cap C^{\vee}} z_1^m z_2^n \in \Q(z_1, z_2)
    $$ 
    of the {\em dual cone} is additive with respect to subdivisions of cones. (The right-hand series analytically continues from its region of convergence
    to a rational function.)
    The differential symbol 
    $$
    	\{f, g\} \mapsto \frac{d \log(f) \wedge d \log(g)}{dz_1 \wedge dz_2}
    $$
    gives rise to a map $K_2(\Q(\mathbb{G}_m^2)) \rightarrow \Q(z_1, z_2)$. 
    We explain in Section \ref{K2boundary} how the association $C \mapsto \phi(C)$ lifts to $K_2(\Q(\gm^2))$ along this map. 
    For $\gamma \in \SL_2(\Z)$, the image of the cone spanned by $(1,0)$ and $\gamma(1,0)$ 
    is a lift of $\Theta_{\gamma}$. The resulting map is only a cocycle modulo $\{-z_1, -z_2\}$,  and we explain 
    in \S \ref{lifting} how it can be modified to avoid even this ambiguity.
 \end{remark}

\subsubsection{Specialization}

To obtain our specialized cocycles, we pull back our big cocycles under $N$-torsion sections of $G^2$ of the form $(1,\iota_N)$,
where $\iota_N$ is an $N$-torsion point or section of $G$.  That is, for $\gm$, we take $\iota_N$ to be a primitive $N$th root of unity $\zeta_N$, and for $\mc{E}$, we take $\iota_N \colon Y_1(N) \to \mc{E}$ to be the universal $N$-torsion section.  The values $\Theta_{\gamma}$ for $\gamma \in \GL_2(\Z)$ need not be regular at $(1,\iota_N)$, but they are for $\gamma$ in the congruence subgroup 
$\bGamma_0(N)$ of $\GL_2(\Z)$ consisting of matrices with bottom left entry divisible by $N$.  So, we must first restrict to this group
prior to taking the pullback. 

For instance, in the case that $G = \gm$, upon pulling back via $(1,\zeta_N)$, we obtain a cocycle
\begin{equation} \label{specializedGm}
 	\Theta_N \colon \bGamma_0(N) \to K_2(\Q(\mu_N))/\langle \{-1,-\zeta_N\} \rangle,
\end{equation}
the right-hand side being the quotient of $K_2(\Q(\mu_N))$ by a group of order at most $2$.  
The restriction of $\Theta_N$ to
$\Gamma_1(N)$ is a homomorphism taking image in the corresponding quotient of $K_2(\Z[\mu_N])$. In fact, it is easy to see that 
$\Theta$ is parabolic so that $\Theta_N$ induces a map from the parabolic homology of the latter group to the quotient of 
$K_2$, which in turn yields $\Pi_N$.  

The map ${}_n z_N$ for $Y_1(N)$ is constructed analogously.  By pulling back, we obtain a cocycle
\begin{equation} \label{specializedE2}
	{}_n \Theta_N \colon \bGamma_0(N) \to K_2(Y_1(N)) \otimes_{\Z} \Z[\tfrac{1}{30}].
\end{equation}
Much as with Siegel units \cite{kato}, upon specialization we can define a universal rational cocycle independent of this choice.
That is, the pullbacks of the resulting classes to $K_2(Y_1(N))$ satisfy natural distribution relations in $n$
that permit one, upon inverting $N$, to construct a specialized cocycle $\Theta_N$ that should be thought of as the $n=1$ case of the construction; see Theorem \ref{canoncocyc}.

The Eisenstein property of $\Pi_N$ and Hecke equivariance of $z_N$ follow from analogous properties of the cohomology classes of the big cocycles, as do the explicit formula for $\Pi_N$ that arises from \eqref{map_eisenstein} and its analogue for $z_N$ involving Steinberg symbols of Siegel units.   

We are moreover able to show the expected explicit formula for $\Theta_N$ as a sum of Steinberg symbols of Siegel units (Beilinson-Kato elements) in Proposition \ref{expformmod}  {\em modulo} a subgroup of $K_2(Y_1(N))$ that vanishes under any standard regulator map.   It would be desirable to eliminate this last ambiguity.

\subsubsection{Relationship to other topics in the literature}   \label{emc}  

Our construction does not stand in isolation but is related to a rich body of theory that has been developed in different contexts.  It is particularly notable that in both cases studied here, one can view the class of the big cocycle $\Theta$ as arising from a class in equivariant motivic cohomology. We briefly describe this class in the $\gm$-case in \S \ref{eqmotivic1}.
 
 The equivariant class corresponding to $\Theta$ provides
 a kernel to pass between cohomology of $\Gamma_1(N)$ and various $K$-groups. 
Our situation 
  is formally similar to the theory of reductive dual pairs, where the theta-function provides a kernel
 to pass between automorphic forms on different groups, and in fact our proofs of Hecke equivariance are formally similar
 to the arguments about theta-kernels. 
 The idea of using an equivariant class as kernel 
has been used in other contexts, for example in Soul{\'e}'s work \cite{Soule} on the Chern character in algebraic $K$-theory. 
 
Our paper is also related to a number of recent works constructing classes in different flavors
 of equivariant cohomology   \cite{bhyy, bcg, ks}.  
 The class most relevant to us is the Eisenstein symbol studied in \cite{beilinson, faltings}, 
 but constructed here equivariantly. The possibility of such an equivariant refinement was observed 
 in a different context by Nekov\'a\v{r} and Scholl \cite[\S 13]{NS}.
A closely related story is the theory of polylogarithms \cite{BL1, HK}, or again, more precisely, the 
equivariant version of such a theory, as is discussed in \cite[\S 3.7]{BKL}.

Our goals are, however, rather different to
those of the papers mentioned above: namely, we aim to develop a framework optimized for the analysis of
 \eqref{cyclmap} and  \eqref{zetamap}, with an emphasis on the explicit description of these maps by symbols.  
This framework can certainly be extended to study other interesting examples as well, such as relating the first homology of Bianchi spaces and Steinberg symbols of elliptic units,
or relating the second homology of locally
 symmetric spaces for $\GL_3$ and Steinberg symbols of three Siegel units, as proposed in \cite[\S 4.2]{fks-survey}. 
 When working in sufficient generality, it will likely be fruitful to systematically proceed in an equivariant fashion.

\subsection{An outline}

We briefly summarize the contents of the paper.  We start by recalling and establishing certain constructions of motivic cohomology useful to our study in Section \ref{prelim}.  Most importantly, we employ coniveau spectral sequences to construct Gersten-type complexes in Milnor $K$-theory, paying special attention to the case of the square of a commutative group scheme.

The next three sections treat the case of $\gm^2$. In Section \ref{gmcase}, we construct the big cocycle $\Theta$ of \eqref{bigcocycGm}. We derive an explicit formula for $\Theta$ in Proposition \ref{expform} and study its behavior under Hecke operators in Proposition \ref{ThetaEis}.
We then specialize $\Theta$ at a torsion point to construct the cyclotomic cocycle $\Theta_N$ of \eqref{specializedGm} in Section \ref{cyclococyc}, deriving its explicit formula (Proposition \ref{explicitspecialized}) and its transformation under Hecke operators  (Theorem \ref{ThetaNEis}) from the results on $\Theta$.  We recover the map $\Pi_N$ of \eqref{cyclmap} from $\Theta_N$ and verify its Eisenstein property in Theorem \ref{Eisprop}.
Section \ref{K2boundary} has a rather different flavor: in it, we examine the construction of $\Theta$ through the lens of toric geometry. 
The main tool is Proposition \ref{quasi-iso}, which constructs a map from the chain complex of the circle to the motivic complex. 

In the final two sections of the paper, we turn to the more technically demanding case of $\mc{E}^2$.
In Section \ref{E2}, we construct the
big cocycles ${}_n \Theta$ of \eqref{bigcocycE2} for primes $n \nmid N$, 
derive an explicit formula for them in Theorem \ref{expformE2},
and demonstrate their Hecke equivariance in Theorem \ref{heckeE2}.
In Section \ref{modcocyc}, we specialize these cocycles ${}_n \Theta$ using an $N$-torsion section to obtain the cocycles
${}_n \Theta_N$ of \eqref{specializedE2}.  We construct a ``universal'' cocycle $\Theta_N$ independent of $n$ in Theorem \ref{canoncocyc}, and we derive an explicit formula for it in Proposition \ref{expformmod}.  Finally, in Theorem \ref{zeta_map}, we construct the map $z_N$ of \eqref{zetamap} and establish its Hecke equivariance. 
 
\begin{ack}

The research of R.S. was supported in part by the National Science Foundation under Grant No.\  DMS-1801963.  He thanks T. Fukaya and K. Kato for prior conversations regarding maps on homology, T. Geisser, M. Levine, and M. Spitzweck for answers to questions regarding motivic cohomology, and C. Khare for a conversation on Galois representations. He also thanks T. Smits and F. Vu for a careful reading of a draft of this work, and E. Lecouturier and P. Xu for very helpful comments on the preprint version.

The research of A.V. was supported  in part by the National Science Foundation under Grant No.\ 1931087. He thanks Aravind Asok for patiently answering questions about motivic cohomology. He also gratefully acknowledges conversations
with N. Bergeron, P. Charollois, and L. Garcia.

Finally, we thank the referees for careful readings that resulted in several corrections.

\end{ack}

\section{Preliminaries on motivic cohomology} \label{prelim}

We shall recall basic properties of motivic cohomology groups in \S \ref{motcoh} and  coniveau spectral sequences in \S \ref{coniveau}.
We shall use this coniveau spectral sequences 
to construct Gersten-type complexes 
in Milnor $K$-theory that will be central to our later study, paying special attention to the case of the square of a commutative group scheme.

 In \S \ref{tracemaps}, we recall the trace maps which will allow us to take fixed parts.  In \S \ref{smallercomplex}, we discuss the particular case of the square of a commutative group scheme of interest to us, introducing our complexes that compute motivic cohomology and various quasi-isomorphic subcomplexes of motivic cohomology groups.

\subsection{Motivic cohomology} \label{motcoh}
We shall define motivic cohomology  using Bloch's cycle complexes: see for instance \cite{bloch-alg, bloch-moving, levine-tech, levine-K}.  This has a certain psychological advantage for us in that it allows us to think of our classes as coming from cycles. 
  However, which theory of motivic cohomology is used does not matter in our final results, which concern smooth schemes over perfect fields.

Let $Y$ denote a quasi-projective scheme of finite type over a perfect field $F$. For nonnegative integers $j$ and $k$, let $z^k(Y,j)$ denote the group of codimension $k$ cycles in $Y \times \Delta^j$ (the $F$-fiber product) that meet $Y \times \Phi$ for each face $\Phi$ of the algebraic $j$-simplex $\Delta^j$ over $F$ properly. 
 Via alternating sums of face maps, the $z^k(Y,\cdot)$ form a homological complex with $z^k(Y,j)$ in degree $j$. This is Bloch's cycle complex for $Y$; its homology groups are called higher Chow groups. These complexes admit pullbacks by flat maps and pushforwards by proper maps \cite[Proposition 1.3]{bloch-alg}

For any $i \in \Z$, we set
$$
	H^i(Y,k) = H_{2k-i}(z^k(Y,\cdot)).
$$  
We also set $H^i(Y,k) = 0$ for negative integers $k$.
If $Y$ is smooth, then $H^i(Y,k)$ is naturally isomorphic to the $i$th motivic cohomology group of $Y$ with $\Z(k)$-coefficients in the sense of Voevodsky \cite{voevodsky} (see \cite[Theorem 19.1]{mvw}):\footnote{For general $Y$ and $F$ admitting resolution of singularities, they are isomorphic to motivic Borel-Moore homology groups \cite[Theorem 19.18]{mvw}}
$$H^i(Y, k) \cong H^i(Y, \Z(k)).$$  
 As such, we will refer to the groups $H^i(Y,k)$ themselves as motivic cohomology groups.
 (This is slightly non-standard notation, which hopefully makes some of the typography easier to read.)
 
We briefly summarize a number of standard properties of these groups.  To start with, as a consequence of the strong moving lemma of Bloch \cite[Theorem 0.1]{bloch-moving}, they admit arbitrary pullbacks (see \cite[Theorem 4.1]{bloch-alg}).  
They also satisfy: 
\begin{itemize}
\item if $Y = \coprod_{h=1}^t Y_h$ is a finite disjoint union of $F$-schemes, then $H^i(Y,k) \cong \bigoplus_{h=1}^t H^i(Y_h,k)$;
\item $H^i(Y,k) \cong H^i(Y \times \mb{A}^1,k)$ via pullback by the projection morphism $Y \times_F \mb{A}^1 \to Y$ (see \cite[Theorem 2.1]{bloch-alg});
\item $H^0(Y, 0) \cong \Z$ if $Y$ is connected and $H^i(Y, 0) = 0$ for $i \neq 0$;
 \item if $Y$ is smooth, then $H^1(Y,1)$ is naturally isomorphic to the group of global units on $Y$, and $H^2(Y,1)$ is naturally isomorphic to the Picard group of $Y$, while $H^i(Y,1) = 0$ for $i \notin \{1,2\}$ (see \cite[Corollary 4.2]{mvw});   
 \item if $Y$ is smooth, then $H^i(Y,k) = 0$ for $i > k + \dim Y$ (see \cite[Theorem 3.6]{mvw}); 
 \item if $Y$ is a smooth variety over $F$, then $H^i(Y,k) = 0$ for $i > 2k$ (see \cite[Theorem 19.3]{mvw}); 
 \item if $f \colon X \rightarrow Y$ is a finite locally free morphism of quasi-projective $F$-schemes of finite type (so proper of relative dimension zero), then $f_* f^*$ is multiplication by the degree of $f$ (cf. \cite[\href{https://stacks.math.columbia.edu/tag/02RH}{Lemma 02RH}]{stacks}).  
 \end{itemize}
 
Suppose that $Y$ is equidimensional. Then, for any closed $F$-subscheme $\rho \colon Z \to Y$ of pure codimension $c$ and its complement $\iota \colon U \to Y$, there is an exact Gysin sequence   \begin{equation} \label{GZ} 
	\cdots \rightarrow H^i(Y, k) \xrightarrow{\iota^*} H^i(U, k)  \xrightarrow{\partial} H^{i-2c+1}(Z, k-c) \xrightarrow{\rho_*} H^{i+1}(Y, k) 
	\rightarrow \cdots. 
\end{equation}
We refer to the map $\partial$ as a residue map.  It results from the distinguished triangle determined by the left exact sequence of complexes given by pushforward by $\iota$ and pullback by $\rho$ given by Bloch's moving lemma.
 
Motivic cohomology also has cup products 
$$
	\cup \colon H^i(Y,k) \times H^{i'}(Y,k') \to H^{i+i'}(Y,k+k'),
$$ 
which can be constructed by pulling back an external product via the diagonal \cite[\S 5]{bloch-alg}.  
There is then an isomorphism of graded rings 
$$
	\bigoplus_{i=0}^{\infty} K_i^M(F) \xrightarrow{\sim} \bigoplus_{i=0}^{\infty} H^i(F,i)
$$
induced by the standard identifications of both sides with $\Z$ and $F^{\times}$ in degrees $0$ and $1$ 
(see \cite[Theorem 5.1 and Lemma 5.6]{mvw}).  Recall that the canonical homomorphism $K_i^M(F) \to K_i(F)$ to the $i$th algebraic $K$-group of $F$ is an isomorphism for $i \le 2$, the case of $i=2$ being Matsumoto's theorem.

We will need to compare compositions of pushforwards and pullbacks. 
For instance, we shall often employ the following in the case that the the underlying schemes are spectra of fields
and $i=k$, in which case the assertion is one of Milnor $K$-theory
 (see also \cite[Rule 1c, p.\ 329]{Rost} for a direct formulation of this assertion, noting Theorem 1.4 therein).

\begin{lemma}[Base change] \label{pullbackdiagram} 
	Suppose that
	$$ 
	\begin{tikzcd}
	X' \arrow{r}{\pi_X} \arrow{d}{f'} & X  \arrow{d}{f} \\
	Y' \arrow{r}{\pi_Y} & Y 
	\end{tikzcd}
	$$
	is a Cartesian diagram of smooth quasi-projective schemes of finite type over $F$, with $\pi_Y$ flat and $f$ proper.  Then 
	$\pi_X$ is flat, $f'$ is proper, and
	$$
		(f')_* \pi_X^* = \pi_Y^* f_*
	$$ 
	as morphisms $H^i(X, k) \rightarrow H^i(Y', k)$.
\end{lemma} 
 
\begin{proof}
	The assertions regarding $\pi_X$ and $f'$ are standard.  Since $\pi_X$ and $\pi_Y$ are
	flat, these morphisms are already defined on cycles 
	by taking inverse images and images, so they are defined on the terms of Bloch's cycle complexes, and they are compatible 
	with the boundary maps (cf.~\cite[Proposition 1.3]{bloch-alg}).  The stated equality of 
	compositions then already holds at the level of complexes (cf. \cite[Proposition 1.7]{fulton}).
\end{proof}

\begin{corollary}[Projection formula] \label{projformula}
	Let $f \colon X \to Y$ be a proper 
	morphism of smooth quasi-projective schemes of finite type over $F$, and let $\alpha \in H^i(X,k)$ and $\beta \in H^{i'}(X,k')$.
	Then 
	$$
		f_*(\alpha \cup f^*(\beta)) = f_*(\alpha) \cup \beta \in H^{i+i'}(X,k+k').
	$$
\end{corollary}

\begin{proof}
	We need only apply Lemma \ref{pullbackdiagram} to the cartesian square
	$$
		\begin{tikzcd}[column sep = large]
		X \arrow{r}{(1 \times f) \circ \Delta_X} \arrow{d}{f} & X \times Y \arrow{d}{f \times 1} \\
		Y \arrow{r}{\Delta_Y} & Y \times Y,
		\end{tikzcd}
	$$
	where $\Delta_X$ and $\Delta_Y$ are the diagonal embeddings of $X$ and $Y$, respectively.
\end{proof}

We also have the following compatibility of residues with transfers and inclusions of fields.  
 
\begin{lemma} \label{restra}
	Let $E/F$ be a finite extension of fields.  Then for $v$ a discrete valuation on $F$,  one has 
	$$ 
		\partial_{v} \circ \Norm_{E/F}  =  \sum_{w \mid v} \Norm_{k(w)/k(v)} \circ \partial_w 
	$$
	as morphisms $K_n^{M}(E) \rightarrow K_{n-1}^{M}(k(v))$ on Milnor $K$-theory; the sum on the right is over valuations 
	$w$ on $E$ extending $v$, the symbols $\partial_v$ and $\partial_w$ are the residue maps on Milnor $K$-theory induced by the
	valuations $v$ and $w$, and $\Norm$ denotes transfer in Milnor $K$-theory.  
	
	Similarly, for each $w \mid v$ as above, we have 
	$$
	\partial_w \circ \iota_{E/F} =  e_{k(w)/k(v)} \cdot \iota_{k(w)/k(v)} \circ \partial_v
	$$
	as morphisms $K_n^M(F) \rightarrow K_{n-1}^M(k(w))$,
	where $\iota$ denotes a map on Milnor $K$-theory induced by inclusions of fields, and $e_{k(w)/k(v)}$ is the ramification index. 
\end{lemma}

\begin{proof}
	This is stated (without proof, but with references) in \cite[Theorem 1.4]{Rost}; see in particular Rules 3b and 3cs therein. 
\end{proof} 
  
\subsection{Coniveau spectral sequences} \label{coniveau}

Let us recall the \emph{coniveau spectral sequence} for motivic cohomology.
We refer to \cite{deglise}, which contains many of the 
details required to set this up. 
The primary role of this spectral sequence is that it provides
complexes that compute motivic cohomology in our situations of interest, and these are also manifestly equivariant for the automorphism
group of the ambient variety. 
 
Continuity properties of motivic cohomology \cite[Lemma 3.9]{mvw} imply that
for a finite type smooth connected variety $Y$ over a field $F$ with function field $k(Y)$, we have 
$$ 
 	H^p(k(Y),q) \cong \varinjlim_{U \subset Y} H^p(U, q),
$$ 
where the limit is taken over open subvarieties $U$ of $Y$.\footnote{This isomorphism is not a tautology, as the definition of motivic cohomology involves the choice of base scheme: here, on the left, it is $k(Y)$, whereas on the right, it is $F$.} 
 
For $U$ as above and any irreducible divisor $D$ such that $D \cap U$ is nonempty,  
 there is a residue homomorphism 
 $$
 	H^p(U - (D \cap U), q) \rightarrow H^{p-1}(D \cap U, q-1).
$$ 
Consider the collection of open sets $U$ such that $D \cap U$ is smooth and nonempty.
The collection of sets $U - (D \cap U)$ is cofinal in open sets on $Y$,
and the collection of $D \cap U$ is cofinal in open sets on $D$.  
Therefore, the residue maps for $U$ in the collection induce a residue map  
\begin{equation} \label{RR}
 	H^p(k(Y), q) \rightarrow H^{p-1}(k(D), q-1).
 \end{equation}
 This latter map is determined by the field $k(Y)$ and the valuation $v$ on it which cuts out $D$ in $Y$ (see \cite[Lemma 5.4.5]{deglise}).
 When $p=q$, it is the residue in Milnor $K$-theory (see \cite[Proposition 6.2.3]{deglise}). 

With these preliminaries in hand, we recall the coniveau spectral sequence for $n \ge 0$.

\begin{theorem} \label{coniveauthm}
 	There is a right half-plane spectral sequence with $E_1$-page
	$$ E_1^{p,q} = \bigoplus_{x \in Y_p} H^{q-p}(k(x), n-p) \Rightarrow H^{p+q}(Y, n),$$ 
	where $Y_p$ denotes the set of points of $Y$ of codimension $p$ and the
	differentials are residue maps.
\end{theorem}

The coniveau spectral sequence essentially carries the information of ``all Gysin sequences at once''
and is a limit of spectral sequences attached to these Gysin sequences. We briefly explain its derivation: attached to a 
decreasing system $Z = (Z_p)_{p \in \Z}$ of closed $F$-subschemes of $Y$ with each $Z_p-Z_{p+1}$ smooth,
$Z_p = Y$ for $p \le 0$, each $Z_p$ with $1 \le p \le n$ of pure codimension $p$, and $Z_p = \varnothing$ for $p > n$,
we have Gysin sequences 
\begin{equation} \label{Gysin_coniveau}
	\cdots \to H^i(Z_p,n-p) \to H^i(Z_p-Z_{p+1},n-p-1) \xrightarrow{\partial} H^{i-1}(Z_{p+1},n-p-1) \to \cdots
\end{equation}
for $0 \le p \le n-1$.
Setting $D^{p,q} = H^{q-p}(Z_p,n-p)$ and $E^{p,q} = H^q(Z_p-Z_{p+1},n)$, the exact couple $(D^{p,q},E^{p,q})$ determined by the 
exact sequences of \eqref{Gysin_coniveau} gives rise to a convergent right half-plane spectral sequence $E(Z)$ with $E_1$-page
$$
	E_1^{p,q}(Z)  = H^{q-p}(Z_p-Z_{p+1},n-p) \Rightarrow E^{p+q}(Z) = H^{p+q}(Y,n).
$$
Note that the $q$th row of the $E_1$-page of this spectral sequence $E(Z)$ is a complex the form
$$
	H^q(Y-Z_1,n) \xrightarrow{\partial} H^{q-1}(Z_1-Z_2,n-1) \xrightarrow{\partial} \cdots \xrightarrow{\partial} H^{q-n}(Z_n,0).
$$
Our convention will be that $p$th term in this complex has homological degree $n-p$.

If we have two collections $Z' = (Z'_p)_p$ and $Z = (Z_p)_p$ of closed subschemes as above 
with each $Z'_p$ a closed subscheme of $Z_p$, then  we obtain morphisms $E^1_{p,q}(Z) \to E^1_{p,q}(Z')$
via composition $j^* \iota_*$ of pushforward and pullback along
$$Z_p' - Z'_{p+1} \xrightarrow{\iota} Z_p - Z'_{p+1}  \xleftarrow{j} Z_p -Z_{p+1}$$
(with $\iota$ a closed immersion and $j$ an open immersion). 
In particular, we can take direct limits of the spectral sequences over directed sets of such collections.  If we use the collection of all $Z$, then we obtain the coniveau spectral sequence.

The row for $q = n$ in the $E_1$-page of the coniveau sequence is a homological complex $\varK$ given in degrees $n$ through $0$ by 
\begin{equation} \label{generalcplx}
 	\varK = \varK^{(n)}(Y) \colon \quad K_n^M k(Y) \to \bigoplus_{x \in Y_1} K_{n-1}^M k(x) \to \cdots \to \bigoplus_{x \in Y_n} K_0^M k(x).
\end{equation} 
It follows from Lemmas \ref{pullbackdiagram} and \ref{restra} that pushforwards by proper maps and pullbacks by flat maps induce morphisms between these sequences via transfer maps and the maps induced by inclusions of fields, respectively, on Milnor $K$-theory. 
In this paper, we employ this complex for $n = 2$.  So, let us describe this case in more detail.
 
\begin{example} \label{n=2}
	Suppose that $n=2$.  Then the $E_1$-terms of the coniveau sequence in the range $0 \leq p \leq 2$ 
	and $0 \leq q \leq 2$ look like this: 
	\begin{equation*} \label{BGQ}
  	\begin{tikzcd}
	(q=2) & \underbrace{ H^2(k(Y), 2)}_{K_2 k(Y)}  \ar[r] & \bigoplus_D \underbrace{H^1(k(D), 1)}_{K_1 k(D)} \ar[r] &  
	\bigoplus_x \underbrace{ H^0(k(x), 0) }_{K_0 k(x)} \\ 
  	(q=1) & H^1(k(Y), 2)  \ar[r] &  0& 0  \\ 
 	(q=0) & H^0(k(Y), 2)  \ar[r] & 0 & 0 \\
		 & (p=0) & (p=1) & (p=2).
 	\end{tikzcd}
	\end{equation*}
	where the direct sums are over divisors $D$ and codimension $2$ points $x$.

	Except possibly those with $p = 0$ and $q < 0$, all other terms vanish, recalling that the motivic cohomology 
	$H^i(F,k)$ of a field $F$ vanishes when $i > k$. 
	In particular, the spectral sequence degenerates, and the row
	\begin{equation} \label{varCcomplex}
		\varK = \varK^{(2)}(Y) \colon \quad
		K_2 k(Y) \xrightarrow{\partial_2} \bigoplus_D K_1 k(D)  \xrightarrow{\partial_1} \bigoplus_{x} \Z 
	\end{equation}
	is a complex in homological degrees $2$, $1$, and $0$
	computing the cohomology groups $H^2(Y, 2)$, $H^3(Y, 2)$, and $H^4(Y, 2)$, respectively.
	
	As noted after \eqref{RR}, the $D$-component of the map $\partial_2$ is given by the tame symbol in $K$-theory
	\begin{equation} \label{tamesymbol}
		\{f,g\} \mapsto (-1)^{v(f) v(g)} g^{v(f)} f^{-v(g)}
	\end{equation}
	for the valuation $v$ attached to $D$.
	The map $\partial_1$ takes the divisor of $f \in k(D)^{\times}$ (i.e., yielding the order of vanishing at $f$ in each $K_0k(x) \cong \Z$
	for $x \in D$), which we interpret in the sense of intersection theory if $D$ is not smooth. 
\end{example}

\begin{remark}  \label{definedonU}
	Suppose that $n \le 2$.
	For any (connected) open subscheme $U$ of $Y$, the maps
	\begin{equation*}
		H^n(U,n) \rightarrow H^n(k(Y), n) \cong K_nk(Y) 
	\end{equation*}
	are injective, as follows for $n = 2$ from the form of the coniveau spectral sequence for $U$ in Example \ref{n=2}, 
	noting that $k(U) = k(Y)$ (and for $n \le 1$ more easily).
	Accordingly, we will say that a class in $K_n k(Y)$ is defined on $U$ if it lies in the image
	of the morphism $H^n(U,n) \rightarrow H^n(k(Y),n)$.
	Given a class in $K_nk(Y)$ defined on $U$ and a closed point $x  \in U$,
	it is then meaningful to specialize $\kappa$ to $x$ via pullback, producing a class in $K_n k(x)$.  
\end{remark}

\subsection{Trace maps} \label{tracemaps} 

Let $G$ be a smooth, connected commutative group scheme over our base smooth variety $Y$ over $F$.
Let $U$ be a nonempty open $F$-subscheme of a closed $F$-subscheme of $G$ of pure codimension.  Multiplication by any positive integer $m$ defines a morphism $m \colon m^{-1}U \to U$.  
 Pushforward by the finite map given by multiplication by $m$ on $G$ induces a map
$$
	H^i(m^{-1}U,k) \to H^i(U,k).
$$
If $m^{-1}U$ is a subscheme of $U$, then precomposing the pushforward by $m$ with pullback under inclusion
gives a morphism
$$
	[m]_* \colon H^i(U,k) \to H^i(U,k)
$$
denoted by the same symbol, which we refer to as a \emph{trace map} for $m$.  (The reader might compare with
 \cite[Definition 2.1.1]{kr}.)
 
  In the remainder of this paper, we will frequently be interested in  the ``fixed parts'' of motivic cohomology groups,
 comprising all elements fixed by all trace maps $[p]_*$ for $p$ prime not equal to the characteristic of $F$. 
 This frequently isolates a subspace of elements of geometric significance.
 We will also consider various slightly weaker notions which we will refer to
 as ``generalized fixed parts'' to distinguish  from the strict version.

 \begin{example} 
Let $z$ denote the coordinate function on $G = \gm$ over $F$.
Choose $m$ not divisible by the characteristic of $F$, and suppose that $U$ open in $G$ satisfies $m^{-1}U \subset U$.
The map $[m]_*$ on $H^1(U,1) \subset F(z)^{\times}$ 
is characterized by the property that 
for $f \in H^1(U,1)$
and $\alpha \in U(F) \subseteq F^{\times}$ 
\begin{equation} \label{normm}
	([m]_*f)(\alpha) =   \prod_{\beta^m = \alpha} f(\beta),
\end{equation}
with the product taken over $m$th roots of $\alpha$ inside an algebraic closure of $F$.
The pullback $[m]^*$ is given more simply by
$$
	([m]^*f)(\alpha) = f(\alpha^m).
$$

In particular, for $U = \gm-\{1\}$, the norm map
$[m]_*$ fixes $1-z$ in $H^1(\gm-\{1\},1)$, as follows from the calculation
\begin{equation} \label{1zeqn}
 	\prod_{i=0}^{m-1} (1-\zeta_m^i z^{1/m}) = 1-z,
\end{equation}
where $\zeta_m$ denotes a primitive $m$th root of unity.  (In fact, $1-z$ is $[m]_*$-fixed even for $m$ divisible by $\cha F$.)
 \end{example}

\begin{example} \label{product example}
Take two smooth connected commutative group schemes $G_1$ and $G_2$ over $F$, 
 and set $G =G_1 \times G_2$. For $\nu_j \in H^{i_j}(G_j,k_j)$ with $i_j \in \Z$ and $k_j \ge 0$, 
define the exterior product $\nu_1 \boxtimes \nu_2 \in H^{i_1+i_2}(G, k_1+k_2)$
as the cup product $\pi_1^* \nu_1 \cup \pi_2^* \nu_2$, with $\pi_j \colon G \to G_j$ the
projection maps.  We then have
\begin{equation} \label{factor} [m]_* (\nu_1 \boxtimes \nu_2) = [m]_* \nu_1 \boxtimes [m]_* \nu_2.\end{equation}
(To verify this from basic properties, factor the multiplication-by-$m$ map $[m]$ 
 as a product of corresponding maps $[m]_1$ and $[m]_2$ in the first and second coordinates.
Then \eqref{factor} follows from the equality
$$[m]_{1*} (\pi_1^* \nu_1 \cup \pi_2^* \nu_2) = [m]_{1*} (\pi_1^* \nu_1 \cup  [m]_1^* \pi_2^* \nu_2)
= [m]_{1*} \pi_1^* \nu_1 \cup \pi_2^* \nu_2 = \pi_1^* [m]_{*} \nu_1 \cup \pi_2^* \nu_2,$$
where the middle equality is the projection formula of Lemma \ref{projformula},
together the analogous assertion with the roles of first and second variables switched.)
\end{example}

The maps $[m]_*$ commute with each other, with pullback to open subschemes, with pushforward by inclusion of closed subschemes, and with residue maps in Gysin sequences (see \cite[\S 2.1]{kr}). 
They also induce a self-map of the $E_1$-page of the coniveau spectral sequence of Theorem \ref{coniveauthm}.

\begin{remark} \label{tracemapsrmk}
 	For $x \in G_p$ (i.e., a codimension $p$ point) and $y \in G_p$ with $my = x$, the trace map
	$$
		[m]_* \colon \bigoplus_{y \in G_p} K_{n-p}^M k(y) \to \bigoplus_{x \in G_p} K_{n-p}^M k(x) 
	$$
	is the sum of norm maps associated to the induced inclusions $k(x) \hookrightarrow k(y)$ for $x, y \in G_p$
	with $my = x$.  By \cite[Lemma 14]{hesselholt}, this is compatible with residues, so differentials on the $E_1$-page
	of the coniveau sequence. In particular, we have trace maps on our complexes $\varK^{(n)}(G)$ of \eqref{generalcplx}.
\end{remark}

\subsection{Powers of commutative group schemes} \label{smallercomplex}

If we start with a smooth, equidimensional quasi-projective scheme $\mc{Y}$ of finite type over $Y$, then
instead of  
taking a limit of motivic cohomology groups over all open subvarieties of $\mc{Y}$, it is natural
to use only those subvarieties which are themselves defined over $Y$.  As in Section \ref{coniveau}, we have a coniveau-type spectral sequence for this limit. We are actually interested in only very special cases with finer structure. Correspondingly, we consider here complements of much smaller collections of closed subsets defined over $Y$ and limits thereof.

Now let us fix $n \ge 1$ and let $\mc{Y} = G^n$ be the $n$th power of a smooth, connected commutative group scheme $G$ of relative dimension $1$ over $Y$, such as $\mb{G}_{m/Y}$ or a smooth family of elliptic curves over $Y$.  We use throughout the convention that the monoid 
$$
	\Delta = M_n(\Z) \cap \GL_n(\Q)
$$ 
of integral matrices of nonzero determinant acts by right multiplication on $G^n$.  E.g., if $n = 2$ and $\smatrix{a&b\\c&d} \in \Delta$, then for any $g_1, g_2 \in G$, we have
\begin{equation} \label{rightaction}
	(g_1,g_2) \cdot \smatrix{a&b\\c&d} = (g_1^ag_2^c,g_1^bg_2^d).
\end{equation}
This being a right action, the monoid $\Delta$ then acts on the \emph{left} on the motivic cohomology groups $H^i(G^n,k)$ by pullback.

Even better, $\Delta$ acts on the left on the complex $\varK = \varK^{(n)}(\mc{Y})$ of \eqref{generalcplx}, also by pullback. That is, if $x \in \mc{Y}_q$ is a codimension $q$ point of $\mc{Y} = G^r$, $\delta \in \Delta$, and $y \in \mc{Y}_q$ is such that $y \cdot \gamma = x$, then pullback yields a map $\gamma^* \colon k(x) \to k(y)$ of residue fields,  and this induces $\gamma^* \colon K_{n-q}^M k(x) \to K_{n-q}^M k(y)$.  The pullback map on $\varK_{n-q}$ is the sum of these maps, and the residue maps are clearly equivariant for this action.
  
Now let us focus on the case $n = 2$ of interest to us.  We consider divisors of the form
\begin{equation} \label{rankone}
	S_{\alpha} = S_{i,j} = \ker\Bigl(G^2 \xrightarrow{z_1^jz_2^j} G \Bigr)
\end{equation}
for nonzero $\alpha = (i,j) \in \Z^2$.  Then $S_{i,j}$ is connected if and only if $i$ and $j$ are relatively prime. 

Take a finite indexing set $I \subset \Z^2 - \{(0,0)\}$ with at least two elements and
containing at most one representative of each element of $\mb{P}^1(\Q)$.  We set
\begin{eqnarray*}
	S_I = \bigcup_{\alpha \in I} S_{\alpha} &\mr{and}& U_I = G^2 - S_I,
\end{eqnarray*}
where we regard $S_I$ as a closed subscheme of $G^2$ 
with its reduced scheme structure 
and $U_I$ as an open subscheme of $\gm^2$. 
We then consider the union of pairwise intersections
$$
	T_I = \bigcup_{\substack{\alpha, \beta \in I \\ \alpha \neq \beta}} (S_{\alpha} \cap S_{\beta}),
$$
which is a finite subgroup scheme of $G^2$ by our choice of $I$.
Let 
$$
	S_I^{\circ} = S_I - T_I
$$ 
so that $S_I^{\circ}$ is the disjoint union of the smooth subschemes 
$S_{\alpha}^{\circ} = S_{\alpha} \cap S_I^{\circ}$
for $\alpha \in I$.

This fits into the setting above for $n = 2$ with $Z_1 = S_I$ and $Z_2 = T_I$, so $Z_0-Z_1 = U_I$ and 
$Z_1-Z_2 = S_I^{\circ}$.  We obtain a spectral sequence having the following terms
in degrees $(p,q)$ with $0 \le p \le 2$ and $0 \le q \le 2$:
\begin{equation*} 
\begin{tikzcd}
	(q=2) & H^2(U_I,2)  \ar[r]  &H^1(S_I^{\circ},1) \ar[r] &  
	H^0(T_I,0) \\ 
	(q=1) & H^1(U_I,2)  \ \ar[r] &  0& 0  \\ 
	(q=0) & H^0(U_I,2)  \ar[r] & 0 & 0\\
		 & (p=0) & (p=1) & (p=2).
\end{tikzcd} 
\end{equation*}
	
This spectral sequence maps to the coniveau sequence detailed in Example \ref{n=2} 
(with $Y$ replaced by $G^2$).  It follows from Remark \ref{definedonU} that
each of the complexes
$$
	\varK_I \colon \quad  H^2(U_I,2) \to H^1(S_I^{\circ},1) \to H^0(T_I,0)
$$
injects quasi-isomorphically into the \emph{big complex}
$$
	\varK \colon \quad K_2(k(G^2)) \to \bigoplus_D K_1k(D) \to \bigoplus_x K_0k(x).
$$

If we order our indexing sets by $I \le I'$ if $U_{I'} \subseteq U_I$, 
then the limit complex $\varinjlim_I \varK_I$ is also
a quasi-isomorphic subcomplex of $\varK$.  The constructions in this paper can all be carried out using this
complex, which is just large enough to allow for the definition of Hecke actions on $\GL_2(\Z)$-cocycles, 
in that it is preserved under the pullback action of the monoid $\Delta = M_2(\Z) \cap \GL_2(\Q)$.

\section{The square of the multiplicative group} \label{gmcase}

In this section, we shall define a cocycle 
$$ \Theta: \GL_2(\Z) \longrightarrow K_2(\Q(\mathbb{G}_m^2))/\mbox{everywhere regular classes},$$
where ``everywhere regular'' means the image of $H^2(\mathbb{G}_m^2, 2)$. 
We will primarily work over the base field $\Q$, but on occasion we will need to work over a finite base field. We follow the notation  for motivic cohomology of Section \ref{motcoh}.

  In \S \ref{motivcohgmr}, we begin by computing the motivic cohomology of $\gm^r$ for $r \ge 1$.  In \S \ref{symbolsmult}, we introduce explicit symbols in the terms of our motivic complex to be used in the construction.  The parabolic cocycle $\Theta$ is constructed in \S \ref{cocycle}, and its explicit formula and its parabolicity are verified using its characterizing property.  In \S \ref{actions}, we then exhibit an Eisenstein property of the class of $\Theta$ for Hecke operators of all prime levels.

\subsection{Motivic cohomology of $\gm^r$} \label{motivcohgmr}

Let $z$ denote the coordinate function on the multiplicative group $\gm$ over a field $F$, 
 normalized so that the value of $z$ at the identity element is $1$.  The motivic cohomology of $\gm$ involves classes directly constructed from $-z$, together with classes pulled back from the motivic cohomology of $\Spec F$ itself.  We extend this description to powers of $\gm$.

First, we construct suspension isomorphisms in motivic cohomology.  Recall the definition of exterior product from Example \ref{product example}.

\begin{proposition} \label{suspension}
	Let $Y$ denote an equidimensional quasi-projective scheme of finite type over $F$. 
	There is a natural isomorphism
	$$
		H^i(Y, k) \oplus H^{i-1}(Y, k-1) \xrightarrow{\sim} H^i(\gm \times Y,k), 
	$$
	where the map on the first summand is pullback under projection to the first factor 
	and the map on the second summand is left exterior product with $-z$, considered as a class in $H^1(\gm, 1)$. 
\end{proposition}

The reason for choosing $-z$, as opposed to $z$, will be made clear in Lemma \ref{fixedpartgm}.  
This result is well-known and corresponds to the ``fundamental theorem'' of algebraic $K$-theory (proved for $K_0$ and $K_1$ by Bass
and in general by Quillen).   
\begin{proof}
	Consider the canonical embedding $\iota \colon \gm \times Y \hookrightarrow \mb{A}^1 \times Y$ given by the usual embedding 
	in the first coordinate and the identity in the second.  The Gysin sequence has the form
	$$
		\cdots \to H^i(\mb{A}^1 \times Y,k) \to H^i(\gm \times Y,k) \xrightarrow{\partial} H^{i-1}(Y,k-1) \to \cdots.
	$$
	As noted in \S \ref{motcoh}, the pullback $H^i(Y,k) \to H^i(\mb{A}^1 \times Y,k)$ by the projection map 
	is an isomorphism.  Thus, it suffices to show that $\partial$ is split by a map on the righthand summand in the theorem, 
	and this follows from $\partial(-z \boxtimes x) = \partial(-z) \boxtimes x = x$ for $x \in H^{i-1}(Y,k-1)$.
 \end{proof}

\begin{corollary} \label{cohomgmr1}
	Let $Y$ denote an equidimensional quasi-projective scheme of finite type over a field $F$, and let $r \ge 1$.  
	There is a natural isomorphism
	$$
		H^i(\gm^r,k) \cong \bigoplus_{j=0}^{\min(k,r)} H^{i-j}(F,k-j)^{\binom{r}{j}}.
	$$
\end{corollary}

\begin{proof}
	This follows by induction on $r$ by iterating Proposition \ref{suspension}, i.e., taking $Y = \gm^{r-1}$ in the inductive step.  
	Note that $H^{i-j}(F,k-j) = 0$ if $k < j$, so the direct sum stops at the minimum of $k$ and $r$.
\end{proof}

Since $H^i(F,k) = 0$ for $i > k$, we obtain in particular:

\begin{corollary} \label{cohomgmr2}
	The groups $H^i(\gm^r,k)$ vanish for all $i > k$.
\end{corollary}

 \subsection{Symbols in the complex computing motivic cohomology}  \label{symbolsmult}
 
Recall from Example \ref{n=2} 
that the coniveau spectral sequence gives rise to a homological complex $\varK$ with nonzero terms in degrees $2$, $1$, and $0$
given by
\begin{equation*} \label{bigcomplex}  
\varK \colon \quad	K_2 k(\mathbb{G}_m^2) \rightarrow \bigoplus_{D} K_1 k(D)  \rightarrow \bigoplus_{x} K_0 k(x),
\end{equation*}
the sums being taken over irreducible divisors and closed points respectively.   This complex
computes the cohomology groups $H^*(\mathbb{G}_m^2, 2)$ in degrees $2$ to $4$ from left to right. 
Therefore, by Corollary \ref{cohomgmr2}, the sequence is exact at middle and right, and its homology at the left is $H^2(\gm^2,2)$.
Let $\barK_2$ be the quotient of $\varK_2$ by the image of $H^2(\gm^2, 2)$ so that  we get a {\em short exact} sequence
$$ 0 \rightarrow \barK_2 \rightarrow \varK_1 \rightarrow \varK_0 \rightarrow 0.$$
 We shall denote the boundary maps in this complex by the generic symbol $\partial$.

The monoid $\Delta = \GL_2(\Q) \cap M_2(\Z)$ acts on the right on $\gm^2$ by the formula of \eqref{rightaction}.
As explained in \S \ref{smallercomplex}, the complex $\varK$ is correspondingly endowed with a left $\Delta$-action via pullback.
This action descends to an action on $\barK$.  For now, we use only the induced action of the group $\GL_2(\Z)$; 
we will employ the full $\Delta$-action in \S \ref{actions}.

Let us define special elements
$$ e \in \varK_0, \quad \varleft  a, c \varright  \in \varK_1, 	\quad \langle \gamma \rangle  \in \varK_2$$
attached to a primitive vector $(a,c) \in \Z^2$  or to a matrix $\gamma \in \GL_2(\Z)$ that satisfy 
\begin{equation} \label{boundary}
	\partial \langle \smatrix{a&b\\c&d} \rangle = \begin{cases} \langle a,c \rangle -\langle -b, -d\rangle & \ifs \det \gamma=1, \\
	 \langle -a,-c \rangle -\langle b, d\rangle & \ifs \det \gamma=-1 \end{cases}	
	\quad	
	 \mbox{and} \quad 
	\partial \langle a,c \rangle = e.
\end{equation}
These are:
\begin{itemize}
\item the $\GL_2(\Z)$-fixed class $e \in \varK_0$ of the element $1 \in \Z$ supported at the identity of $\gm$,
\item for a primitive vector $(a,c) \in \Z^2$ and the torus $S_{a,c} = \ker (\gm^2 \xrightarrow{(x,y) \mapsto x^ay^c} \gm)$ of \eqref{rankone},
the image $\langle a,c \rangle \in \varK_1$ of the invertible function 
$$ 1-z_1^b z_2^d \in \mathcal{O}(S_{a,c}-\{1\})^{\times} \hookrightarrow K_1( \Q(S_{a,c})),$$
where $\smatrix{a & b \\ c & d} \in \SL_2(\Z)$ extends $(a,c)$; 
this is independent of the choice of $(b,d)$, since another choice simply alters the function $z_1^b z_2^d$
by a multiple of $z_1^a z_2^c$, which is $1$ on $S_{a,c}$,

\item for
  $\gamma = \smatrix{a&b\\c&d} \in \GL_2(\Z)$,
  and its columns $v_1= (a,c)$ and $v_2= (b,d)$, the Steinberg symbol
$$\langle \gamma \rangle = \langle v_1, v_2 \rangle 
 	= \{1-z_1^az_2^c, 1-z_1^b z_2^d\} \in \varK_2. $$ 
Note that $\langle \gamma \rangle = \gamma^* \langle \smatrix{1&0\\0&1} \rangle$
is the image of $(1-z_1^az_2^c) \cup (1-z_1^bz_2^d) \in H^2(\gm^2 - S_{a,c} \cup S_{b,d},2)$.
\end{itemize}

The special elements of the form $e$, $\langle a,c \rangle$ for $(a,c) \in \Z^2$ primitive, and 
$\langle \gamma \rangle$ for $\gamma \in \GL_2(\Z)$ together span a subcomplex $\Symb$ of $\varK$ that we refer to as the \emph{symbol complex}.  We'll return to it in Sections \ref{cyclococyc} and \ref{K2boundary}.  
That these symbols satisfy \eqref{boundary} follows directly from the description in Example \ref{n=2} of the residue maps in $\varK$ in terms of tame symbols \eqref{tamesymbol} and divisors.  

\begin{remark}
We note for later use that, for $\gamma  = \abcdmat \in \GL_2(\Z)$,
the pullback $\gamma^* \langle 0,1 \rangle$ is supported on the torus $S_{b,d}$ which is 
the kernel of $(z_1, z_2) \mapsto z_1^b z_2^d$.
On this divisor, it is given by $\gamma^* (1-z_1^{-1}) = 1- z_1^{-a} z_2^{-c}$. 
Thus
\begin{equation} \label{gamma01} \gamma^* \langle 0,1 \rangle = \begin{cases} \langle b,d \rangle  & \ifs \det(\gamma)=1, \\ \langle -b, -d \rangle  & \ifs \det(\gamma)=-1. \end{cases}\end{equation}
\end{remark}

 \subsection{The cocycle} \label{cocycle}

 Pulling back the complex $\barK_2 \rightarrow \varK_1 \rightarrow \varK_0$
to the cyclic subgroup generated by $e \in \varK_0$, we get an extension
of $\Z$ by $\barK_2$, and so an extension class in
\begin{equation} \label{ext point of view} \Ext^1_{\Z[\GL_2(\Z)]}(\Z,\barK_2) = H^1(\GL_2(\Z), \barK_2).\end{equation}
We shall describe a cocycle representing this class more explicitly in Proposition \ref{existence} below. We then give an explicit recipe for $\Theta_{\gamma}$ as a sum of symbols $\langle \rho \rangle$ with $\rho \in \SL_2(\Z)$
and show that it lies in parabolic cohomology.  Because $H^2(\GL_2(\Z),\Z)$
is torsion, a multiple of $\Theta$ can actually be lifted to $\varK_2$; in Section \ref{K2boundary},
we sketch how to do this explicitly.

\begin{proposition} \label{existence}
	There is a $1$-cocycle  
	$$ 
		\Theta \colon \GL_2(\Z) \to \barK_2, \quad \gamma \mapsto \Theta_{\gamma},
	$$
	uniquely characterized by the
	property that
  	$$
		\partial \Theta_{\gamma} =   (\gamma^* - 1)  \langle 0,1 \rangle.
	$$
\end{proposition}

\begin{proof}
	Since $(\gamma^*-1)\langle 0,1 \rangle$ has trivial boundary $e-e = 0$, it is the boundary
	of a unique $\Theta_{\gamma} \in \barK_2$.  
	Since pullback is a left action, we have
	$$
		\partial\Theta_{\gamma \gamma'} = (\gamma^*(\gamma')^*-1)\langle 0,1 \rangle
		= \gamma^*((\gamma')^*-1)\langle 0,1 \rangle + (\gamma^*-1)\langle 0,1 \rangle
		= \partial(\gamma^*\Theta_{\gamma'} + \Theta_{\gamma}),
	$$
	for $\gamma, \gamma' \in \GL_2(\Z)$.  That $\Theta$ is a cocycle therefore follows by the exactness of $\barK$.
\end{proof}

Next, we give an explicit recipe for values of $\Theta$ in terms of our special symbols in $\varK_2$ using a standard variant of the
Euclidean algorithm, analogous to writing a geodesic between cusps on the modular curve as a sum of Manin symbols.  We make
the latter analogy precise in \S \ref{homology}.  

Given $\gamma = \smatrix{a&b\\c&d} \in \GL_2(\Z)$, the Euclidean algorithm allows us to find a sequence 
$(v_i)_{i=0}^k$ in $\Z^2$ for some $k \ge 0$ with $v_0 = (0,1)$ and $v_k = \det(\gamma) (b,d)$   and such that the $v_i = (b_i,d_i)$ satisfy 
$\det \smatrix{b_{i-1} & b_i \\ d_{i-1} & d_i} = 1$ for all $1 \le i \le k$.  We call such a sequence $(v_i)_{i=0}^k$ a 
\emph{connecting sequence} for $\gamma$.

\begin{proposition} \label{expform} 
	Let $\gamma = \smatrix{a&b\\c&d} \in \GL_2(\Z)$, and choose a connecting sequence $(v_i)_{i=0}^k$ for $\gamma$.  
	Then we have the following equality in $\barK_2$:
	$$
		\Theta_{\gamma} = \sum_{i=1}^k \langle v_i, -v_{i-1} \rangle.
	$$
\end{proposition}
Recall from \S \ref{symbolsmult} that $\langle v_i, -v_{i-1} \rangle$ is the symbol associated to the matrix with first column $v_i$ and second column  $-v_{i-1}$.
\begin{proof}
	By \eqref{gamma01} and \eqref{boundary} we have
	$$ 
		\partial \left( \sum_{i=1}^k \langle v_i, -v_{i-1} \rangle \right)
		= \sum_{i=1}^k (\langle v_i \rangle - \langle v_{i-1} \rangle) = \langle v_k \rangle - \langle v_0 \rangle 
		=  \gamma^* \langle 0,1 \rangle - \langle 0,1 \rangle.
	$$
	Since $\Theta_{\gamma}$ and $\sum_{i=1}^k \langle v_i,-v_{i-1} \rangle$ have the same boundary, they are
	equal in $\barK_2$.
\end{proof}

\begin{example}
	Take $\gamma = \smatrix{-1&0\\0&-1}$.  Then $v_0 = (0,1)$, $v_1 = (-1,0)$, and $v_2 = (0,-1)$ form a connecting
	sequence for $\gamma$, so 
	$$
		\Theta_{\gamma} = \langle (-1,0),(0,-1) \rangle + \langle (0,-1),(1,0) \rangle = \{ -z_1^{-1}, 1-z_2^{-1} \},
	$$
	which equals $-\{ -z_1, 1-z_2 \}$ in $\barK_2$.
\end{example}

We will use a perhaps slightly nonstandard notion of parabolic cohomology for $\GL_2(\Z)$, consistent with our use
of right actions of $\GL_2(\Z)$ on group schemes.   That is, we define the $\GL_2(\Z)$-parabolic cohomology group $H^1_P(\GL_2(\Z),M)$
for a $\Z[\GL_2(\Z)]$-module $M$ to be the intersection of the kernels of the restriction maps from $H^1(\GL_2(\Z),M) \to H^1(P,M)$, where $P$ runs over all stabilizers of nonzero elements of $\Z^2$
under the right action of $\GL_2(\Z)$. We say that a $1$-cocycle $\GL_2(\Z) \to M$ is \emph{parabolic} if its class lies in the parabolic cohomology group $H^1_P(\GL_2(\Z),M)$.

 \begin{proposition} \label{parabolic}
	The cocycle $\Theta$ is parabolic.
\end{proposition}

\begin{proof}
	Since the right action of $\GL_2(\Z)$ on $\mb{P}^1(\Q)$ is transitive, 
	it is enough to verify triviality upon restriction to the stabilizer $P_{\infty} = \{ \smatrix{1&0 \\ c&\pm 1} \mid c \in \Z\}$ 
	of $\infty$.  If we take $\gamma \in P_{\infty}$, then by \eqref{gamma01},  
	$$
		\partial\Theta_{\gamma} = (\gamma^*-1)\langle 0,1 \rangle = \langle 0,1 \rangle - \langle 0,1 \rangle = 0,
	$$ 
	so $\Theta_{\gamma} = 0$ in $\barK_2$.
 \end{proof}
 
 \begin{remark} \label{integral}
 	The parabolic cocycle $\Theta$ is also \emph{integral} in a sense we shall now describe. 
	For this, let us suppose the theory of motivic cohomology over schemes over Dedekind domains (see the work of Levine 
	\cite{levine-tech}, Geisser \cite{geisser}, and Spitzweck \cite{spitzweck}), which we denote as over fields. 
	Take the direct limit of second motivic cohomology groups
	$\varK_{2/\Z} = \varinjlim_U H^2(U,2)$, where $U$ runs over the open $\Z$-subschemes of $\gm^2{}_{/\Z}$ that 
	are complements of unions of kernels of morphisms 
	$\gm^2{}_{/\Z} \to \gm^2{}_{/\Z}$ with $(z_1,z_2) \mapsto z_1^az_2^c$ for some primitive $(a,c) \in \Z^2-\{0\}$.  There is a canonical
	injection $\varK_{2/\Z} \hookrightarrow \varK_2$, under which the inverse image of $H^2(\gm^2,2)$ is $H^2(\gm^2{}_{/\Z},2)$.  
	The statement is then that 
	$$
		\mbox{$\Theta$ takes values in $\varK_{2/\Z}/H^2(\gm^2{}_{/\Z},2)$}.
	$$ 
	This can be seen directly from the explicit formula of Proposition \ref{expform} or without recourse to this formula
	using Gysin sequences and Lemma \ref{fixedpart} over finite fields (supposing an expected compatibility of pushforwards
	and residues as in Lemma \ref{restra} that we did not endeavor to check). 
 \end{remark}
    
\subsection{Hecke actions} \label{actions}

 We now turn to the action of Hecke operators on the class of our $1$-cocycle $\Theta$.
To set the stage, suppose that $\Delta$ is a submonoid of $M_2(\Z) \cap \GL_2(\Q)$ and $\Gamma$ is a finite index subgroup of 
$\GL_2(\Z) \cap \Delta$.
We recall the explicit formulas for the action of Hecke operators of double cosets for $\Gamma \backslash \Delta / \Gamma$ 
on the cohomology $H^1(\Gamma, M)$ for any $\Z[\Delta]$-module $M$.

For $g \in \Delta$, write
\begin{equation} \label{doublecoset}
	\Gamma g \Gamma = \coprod_{j=1}^t g_j\Gamma.
\end{equation}
For $\gamma \in \Gamma$, there exists a permutation $\sigma \in S_t$ (the permutation group on $t$ letters) 
and elements
$\gamma_j \in \Gamma$ such that $\gamma g_j =  g_{\sigma(j)}\gamma_j$ for $1 \le j \le t$. 
For a $1$-cocycle $\theta \colon \Gamma \to M$ and $\gamma \in \Gamma$, we set
\begin{equation} \label{Heckecocyc}
	T(g)\theta(\gamma) =  \sum_{j=1}^t g_{\sigma(j)}\theta(\gamma_j),
\end{equation}
which in general depends on the chosen coset representatives $g_j$. The following lemma is well known and verified simply by writing out the definitions. 
\begin{lemma} \label{Heckelemma}
	For a $1$-cocycle $\theta \colon \Gamma \to M$, the cochain $T(g)\theta$ is a cocycle with class 
	independent of the choice of double coset decomposition.  In particular, $T(g)$ induces a well-defined action on 
	$H^1(\Gamma,M)$.  Moreover, this restricts to an action on parabolic cocycles and parabolic cohomology.
\end{lemma}

\begin{remark}  \label{leftright} 
	The Hecke operators $T(g)$ of Lemma \ref{Heckelemma} arise from left coset decompositions of $\Gamma g \Gamma$ for 
	$g \in \Delta$.
	Given a right $\Delta$-module $N$, the analogous construction to the above yields {\em right} Hecke operators $T^R(g)$ 
	using a decomposition of $\Gamma g \Gamma$ into right cosets, as often found in the literature 
	(see, e.g., \cite[Section 8.3]{shimura}).  
	
	The two operators are related as follows. 	Write $*$ for the anti-involution on $\GL_2(\Q)$ given by $g^* = (\det g) g^{-1}$.  
 If $\theta \colon \Gamma \to M$ is a left cocycle, then $
	\theta' \colon \gamma \mapsto \theta(\gamma^{-1})$
	is a $\Gamma$-cocycle for the right action of $\Delta$ on $M$ given by $(m, h) \mapsto h^* m$,
	and the rule $\theta \mapsto \theta'$ intertwines the actions of $T(g)$ and $T^R(g^*)$. 
	If the action on $M$ is trivial, then there is no distinction between left and right cocycles,
	and the actions of $T(g)$ and $T^R(g^*)$ coincide. 
\end{remark}

Returning to our case of interest, we again take
$$
	\Delta = M_2(\Z) \cap \GL_2(\Q)
$$
and $\Gamma = \GL_2(\Z)$.
The monoid $\Delta$ acts on the right on $\gm^2(Y)$ for
any smooth $\Q$-scheme $Y$ by the formula \eqref{rightaction}, and
this right action on $\gm^2$ induces a left pullback action on $\varK$ as in \S \ref{smallercomplex}.  

For example, the pullback by $\gamma \in \Delta$
of an invertible regular function $f$ on $S_{i,j}-\{1\}$
is the function on  $S_{ai+bj,ci+dj} = S_{i,j}\gamma^{-1}$ defined by
$$
	(\gamma^*f)(z_1,z_2) = f((z_1,z_2)\gamma) = f(z_1^a z_2^c, z_1^b z_2^d),
$$
and the divisor of $\gamma^*f$ is the pullback of the divisor of $f$. 	
This action descends to an action on $\barK$ also, since $M_2(\Z) \cap \GL_2(\Q)$
acts by pullback on $H^2(\gm^2,2)$, compatibly with its morphism to $\varK_2$. 
In the case of the matrix $\smatrix{\ell & 0 \\ 0 & \ell}$ we denote this action on $\varK$
more succinctly by $[\ell]^*$.  
 
For a prime $\ell$, let $T_{\ell} = T(g)$ for $g = \smatrix{ \ell \\ & 1 }$, where $T(g)$ as in \eqref{Heckecocyc} above for the coset representatives \footnote{This operator agrees with the Hecke operator $T^R(g^*)$ as defined via right cosets of $g^* = \smatrix{1 \\ & \ell}$ as in Remark \ref{leftright}.}
\begin{equation} \label{Gjdef} g_j = \Pmatrix{ \ell & j \\ & 1 } \mbox{ for } 0 \le j \le \ell - 1 \quad \mbox{ and } \quad 
g_{\ell} = \Pmatrix{ 1 \\ & \ell }\end{equation} 
in $\Gamma g \Gamma = \coprod_{j=0}^{\ell} g_j \Gamma$. By Lemma \ref{Heckelemma}, this gives an action of $T_{\ell}$ on cohomology independent of the latter decomposition.
Let us also define an 
endomorphism of the complex $\varK$
by the rule 
$$T_{\ell}^{\varK} = \sum_{j=0}^{\ell} g_j^* \colon \varK \to \varK.$$
This depends upon the choice of $g_j$ and
appears primarily as a computational aid. 

We can now compute the action of $T_{\ell}$ on the class of $\Theta$ from the action of $T_{\ell}^K$ on $e$ of \S \ref{symbolsmult}.

\begin{lemma} \label{Heckeidentity}
	For each prime $\ell$, we have an equality 
	$$
		T_{\ell}^{\varK} e = (\ell + [\ell]^*)e
	$$
	of elements of $\varK_0$.
\end{lemma}

\begin{proof}
	The left-hand side is the sum (with multiplicity) of the classes of the $\ell+1$ cyclic $\ell$-subgroups of $\mu_{\ell}^2$. 
	This is the sum of the class $[\ell]^*e$ of $\mu_{\ell}^2$ and $\ell$ copies of $e$. 
\end{proof}

In the following, note that $[\ell]^*$ acts on the cocycle $\Theta$ through its action on $\barK_2$.

\begin{proposition} \label{ThetaEis}
	The $1$-cocycle $(T_{\ell} -\ell -[\ell]^*)\Theta$ has trivial class in $H^1(\GL_2(\Z), \barK_2)$.
\end{proposition}

\begin{proof}
	As the residue of $\langle 0,1 \rangle$ is $e$, Lemma \ref{Heckeidentity} and the 
	$\Delta$-equivariance of the boundary maps in $\varK$ imply that
	$
		(T_{\ell}^{\varK} - \ell - [\ell]^*)\langle 0,1 \rangle  
	$
	has zero residue.
	Accordingly there exists $\psi \in \varK_2$  so that
	\begin{equation} \label{psi}
		\partial \psi = (T_{\ell}^{\varK} - \ell - [\ell]^*)\langle 0,1 \rangle.
	\end{equation}
	For $\gamma \in \GL_2(\Z)$, let $\sigma$ be a permutation of $\{0,\ldots,\ell\}$, and let $\gamma_j \in \GL_2(\Z)$ for
	$0 \le j \le \ell$ be such that $\gamma g_j = g_{\sigma(j)} \gamma_j$.
	We then have
	\begin{equation} \label{TT}
		 (\gamma^*-1)T_{\ell}^{\varK} \langle 0,1 \rangle
		 =(\gamma^* -1) \sum_{j=0}^{\ell} g_j^* \langle 0,1 \rangle 
		= \sum_{j=0}^{\ell} g_{\sigma(j)}^*(\gamma_j^*-1) \langle 0,1 \rangle
		= \partial \left( \sum_{j=0}^{\ell} g_{\sigma(j)}^* \Theta_{\gamma_j} \right)
		= \partial (T_{\ell}\Theta)_{\gamma},
	\end{equation} 
	where the last equality follows by definition \eqref{Heckecocyc} of our Hecke action on $\varK$.
	Comparing \eqref{psi} and \eqref{TT}, we see that 
	$$
		\partial (\gamma^*-1)\psi = (\gamma^*-1)(T_{\ell}^K-\ell-[\ell]^*)\langle 0,1 \rangle = \partial((T_{\ell}-\ell-[\ell]^*)\Theta)_{\gamma}.
	$$
	It follows that $((T_{\ell}-\ell-[\ell]^*)\Theta)_{\gamma}$ and $(\gamma^*-1)\psi$ coincide in $\barK_2$, and therefore 
	the cocycle $(T_{\ell}-\ell-[\ell]^*)\Theta$ is the coboundary of $\psi$.  
\end{proof}

\section{The cyclotomic cocycle} \label{cyclococyc}

We specialize the cocycle
of the previous section at an $N$-torsion point of $\gm^2$. 
 There are two points to be addressed:  classes in $K_2$ of the function field cannot {\em a priori} be specialized at a point,
 and the previous cocycle was valued not in this $K_2$ but its quotient by everywhere regular classes.
 To remedy the second issue, we narrow down the regular classes using trace maps. 
 
In \S \ref{fixedparts},  we calculate the fixed part of the motivic cohomology of $\gm^2$ under trace maps. We show that our explicit symbols are contained in the fixed parts of our big complex and use this to reduce the ambiguity in the values of $\Theta$.  In \S \ref{specialization}, we construct the explicit cocycle $\Theta_N$ by pulling back $\Theta$,
verify its explicit formula (Proposition \ref{explicitspecialized}), and demonstrate its Eisenstein property (Theorem \ref{ThetaNEis}) 
for prime-to-level Hecke operators.  In \S \ref{homology}, we compare with prior work: in particular, we show that $\Theta_N$ induces the map $\Pi_N$ of \eqref{cyclmap} that is the restriction of the explicit map of \cite{busuioc, sharifi}, and we verify the Eisenstein property of $\Pi_N$.

\subsection{Fixed parts via suspension}
  \label{fixedparts}

\subsubsection{Fixed parts of the cohomology of $\gm^2$}

The results of \S \ref{motivcohgmr} imply that the motivic cohomology group $H^2(\gm^2,2)$ breaks up as a direct sum of motivic cohomology classes that are $\Z$-multiples of $(-z_1) \cup (-z_2)$ and sums of classes pulled back via one of the two projection maps.
We want to be able to ``ignore'' the latter classes, and to kill them we will use trace maps.

We work in this subsection over a base field $F$.  For $r \ge 1$, we define the fixed part of the motivic cohomology group $H^i(\gm^r,k)$ as
\begin{equation} \label{fixedpointsGm}
	H^i(\gm^r,k)^{(0)} = \{ \alpha \in H^i(\gm^r,k) \mid ([p]_*-1)\alpha = 0  \text{ for all primes $p \neq \cha F$} \}.
\end{equation}
 In general, if $F$ has zero or sufficiently large characteristic, then this fixed part is the direct summand of $H^i(\gm^r,k)$ given by $H^i(F,k)$ in the decomposition of Corollary \ref{cohomgmr1}.  However, we shall only be interested in these groups in very specific cases: in particular, 
let us study them for $i = k \le r \in \{1,2\}$.

\begin{lemma} \label{fixedpartgm}
	The element $-z \in H^1(\gm,1)$ generates $H^1(\gm,1)^{(0)}$.
\end{lemma}

\begin{proof}
	The group $H^1(\gm, 1)$ consists of the invertible functions on $\mathbb{G}_{m/F}$. As such, 
	each element is uniquely of the form $\eta (-z)^k$ for $\eta \in F^{\times}, k \in \Z$. 
	
	Suppose that such a class is $[m]_*$-fixed for 
        $m$ prime to $\cha F$.  The global unit $-z$ is $[m]_*$-fixed for $m$ prime to $\cha F$ since
        \begin{equation} \label{-zfixed}
        		[m]_*(-z) = \prod_{i=0}^{m-1} (-\zeta_m^i z^{1/m}) = - z.
        \end{equation}
        Thus, we also have $[m]_* \eta = \eta^m$ for such $m$.  
        
       If $\cha F \neq 2$, then $\eta^2 =\eta$ and so $\eta=1$;          
       in characteristic $2$ the equality $\eta^3=\eta$ implies the same conclusion. 
       So $\eta=1$, and the claim follows.          
 \end{proof}

Let us turn to $\gm^2$ over $F$, on which we let $z_i$ denote the $i$th coordinate function.  

\begin{lemma} \label{fixedpart}
	The group $H^1(\gm^2,1)^{(0)}$ 
	is trivial, and if $F$ has characteristic not $2$ or $3$ or is a finite field, then
	$H^2(\gm^2, 2)^{(0)}$ is $\Z$-free of rank $1$, generated by $(-z_1) \cup (-z_2)$. 
\end{lemma}

\begin{proof}  
	Let us use $\nu_j$ to denote the class of $-z_j$ in $H^1(\gm^2,1)$ for short.
	As in Corollary \ref{cohomgmr1}, we have an isomorphism
	\begin{equation} \label{iterated}
		H^2(F, 2) \oplus H^1(F, 1) \oplus H^1(F, 1) \oplus H^0(F, 0) \xrightarrow{\sim} H^2(\gm^2,2), 
	\end{equation}
	where the maps are given by pullback and (left) cup product with $1$, $\nu_1$, $\nu_2$ and $\nu_1 \cup \nu_2$, 
	respectively.
	
	Let $m$ be prime to $\cha F$. By Example \ref{product example} and \eqref{-zfixed}, the class $\nu_1 \cup \nu_2$ is $[m]_*$-fixed.
	Any class $\eta$ that is pulled back from $\Spec F$  satisfies
	$[m]^*\eta = \eta$, so for such $\eta$ and any $\alpha$ among $1$, $\nu_1$, $\nu_2$ and $\nu_1 \cup \nu_2$, Corollary
	\ref{projformula} yields
	$$
		[m]_*(\alpha \cup \eta) = [m]_*(\alpha \cup [m]^*\eta) = ([m]_*\alpha) \cup \eta,
	$$
	which tells us that the trace maps preserve the summands in \eqref{iterated}.  
	Thus, we need only
	consider the summands individually.  So, let us suppose that $\alpha \cup \eta$ is $[m]_*$-fixed:
	\begin{itemize}
	\item[(i)]  If $\alpha = 1$, then $[m]_*\eta = m^2\eta$.  So, if $\eta$ is $[m]_*$-fixed, then
	$(m^2-1)\eta = 0$. 
	If $\cha F \ge 5$, then since this 
	is true for $m = 2$ and $m = 3$, we have $\eta = 0$.  In the case that $i = k = 1$, note that if $\cha F = 2$, then 
	$8\eta = 0$ implies $\eta = 0$ and if $\cha F = 3$, then $3\eta = 0$ implies $\eta = 0$.  If $i = k = 2$ and $F$
	is a finite field, then $H^2(F,2) = 0$, so $\eta = 0$.
	
	\item[(ii)] If $\alpha = \nu_j$, then we have
	$$
		[m]_*(\nu_j \cup \eta) = m(\nu_j \cup \eta),
	$$
	so if $\nu_j \cup \eta$ is $[m]_*$-fixed, then it is $(m-1)$-torsion.  
	If $i = k = 1$, then $H^0(F,0) \cong \Z$, so $\eta = 0$.
	In general, so long as $\cha F \ge 3$, then $\eta$ is trivial taking $m = 2$.  If $i = k = 2$ and $\cha F = 2$, then
	by taking $m=3$, we see that $\eta$ is $2$-torsion in $F^{\times}$, so trivial.
	\item[(iii)] 
	If $\alpha = \nu_1 \cup \nu_2$, then it is indeed $[m]_*$-fixed.
	\end{itemize}
\end{proof}

\subsubsection{Fixed parts of complexes} \label{fixedpartssubsec}

We return to the consideration of $\gm^2$ over $\Q$.  As explained in Remark \ref{tracemapsrmk}, the trace maps $[m]_*$ 
act on the complex $\varK$ as well. We define the fixed complex $\varK^{(0)}$ in exactly the same way as \eqref{fixedpointsGm}.
 
\begin{lemma} \label{mfixed}
	The symbols defined in \S \ref{symbolsmult} all lie in the fixed part of $\varK$:
	$$e \in \varK_0^{(0)}, \ \langle a,c \rangle \in \varK_1^{(0)}, \ \langle\gamma\rangle \in \varK_2^{(0)}.$$
 	
\end{lemma}	

\begin{proof}
	First, note that $[m]_*e = e$ for all $m$.  From \eqref{1zeqn}, we see that $\langle a,c \rangle \in \varK_1^{(0)}$.
	Finally, since $[m]_*$ on $\varK_2$ is given by the ``product of the pushforwards by $m$ in the first
	and second variable'' by \eqref{factor}, it fixes $\langle (1,0),(0,1) \rangle$.  We then note that $\gamma^*$
	and $[m]_*$ commute.
\end{proof}
		
\begin{proposition} \label{Thetalifts}
	The cocycle $\Theta$ lifts  to a cocycle valued in $\varK_2/\langle \{-z_1, -z_2\} \rangle$. 
\end{proposition}

\begin{proof}
 	Indeed, Proposition \ref{expform}  and Lemma \ref{mfixed} imply that for each $\gamma \in \GL_2(\Z)$, 
	the cocycle $\Theta_{\gamma}$ is valued in the image of $\varK_2^{(0)} \rightarrow \barK_2$. 
	By Lemma \ref{fixedpart}, that image is isomorphic  to $\varK_2^{(0)}/ \langle \{-z_1, -z_2\} \rangle$.   
	Thus $\Theta$ lifts to that group, and {\em a fortiori} to  $\varK_2/ \langle \{-z_1, -z_2\} \rangle$.   
\end{proof}

In the remainder of this section, we will implicitly regard $\Theta$ as valued in $\varK_2/\langle \{-z_1, -z_2\} \rangle$. 

\begin{remark} \label{symbolcomplex}
	The symbol complex $\Symb$ defined in \S \ref{symbolsmult} is contained in the fixed part under trace maps of the 
	limit complex $\varinjlim_I \varK_I \subset \varK$ \S \ref{smallercomplex}.  
	In fact, it is the fixed part of a certain motivic subcomplex $\varC$ of $\varinjlim_I \varK_I$ that we refer to as the
	\emph{small complex}.
	
	To make this precise, for $\gamma = \smatrix{a&b\\c&d} \in \SL_2(\Z)$, let us set $\varC_{\gamma} = \varK_I$
	for $I = \{(a,c),(b,d)\}$.
	The complex $\varC_{\gamma}$ has the form
	$$
		H^2(\gm^2 - S_{a,c} \cup S_{b,d},2) \to H^1(S_{a,c}-\{1\},1) \oplus H^1(S_{b,d}-\{1\},1) \to H^0(\{1\},0),
	$$
	where $S_{a,c}$ and $S_{b,d}$ are the rank one tori of \eqref{rankone}, and the last term is identified with $\Z$.  
	Using Gysin sequences and the results of Section \ref{motivcohgmr}, 
	it is not hard to see that $\varC_{\gamma}^{(0)}$ is canonically a direct summand of $\varC_{\gamma}$ such that
	\begin{itemize}
		\item $\varC_{\gamma,2}^{(0)} \cong \Z^4$ is generated by $\langle \pm(a,c), \pm(b,d) \rangle$,
		\item $\varC_{\gamma,1}^{(0)} \cong \Z^4$ is generated by $\langle \pm(a,c) \rangle$ and $\langle \pm(b,d) \rangle$, and
		\item $\varC_{\gamma,0}^{(0)} \cong \Z$ is generated by $e$.
	\end{itemize}
	The homology of $\varC_{\gamma}^{(0)}$ is then concentrated in degree $2$, being isomorphic to $H^2(\gm^2,2)^{(0)} \cong \Z$.
	
	Define the small complex $\varC \subset \varK$ to be the span of the $\varC_{\gamma}$ for $\gamma \in \SL_2(\Z)$.  
	One may verify that the small complex is, like its subcomplexes $\varC_{\gamma}$, 
	a quasi-isomorphic subcomplex of $\varK$, and from our description of each
	$\varC_I^{(0)}$, we see that $\varC^{(0)}$ is precisely the symbol complex $\Symb$.  
	In fact, $\Symb = \varC^{(0)}$ is a $\Z[\GL_2(\Z)]$-direct summand of $\varC$ with homology $H^2(\gm^2,2)^{(0)}$ in degree $2$.
	In particular, by Proposition
	\ref{expform}, our cocycle $\Theta$ takes values in $\barC_2^{(0)} = \varC_2^{(0)}/\langle (-z_1) \cup (-z_2) \rangle$.
\end{remark}
 
\subsection{Specialization at an $N$-torsion point} \label{specialization}

Fix a positive integer $N \ge 2$.  Let us fix the notation for the congruence subgroups of $\GL_2(\Z)$ that we shall use from this point forward.   That is, we set \begin{eqnarray} \label{tildegammadef} &\bGamma_0(N) =\left\{ \smatrix{a&b\\c&d} \in \GL_2(\Z) \mid N \mid c \right\}, \\
 \label{tildegammadef2}& \bGamma_1(N) = \left\{ \smatrix{a&b\\c&d} \in \bGamma_0(N) \mid d \equiv 1 \bmod N \right\}.
 \end{eqnarray}
We denote by $\Gamma_0(N)$ and $\Gamma_1(N)$ their respective intersections
with $\SL_2(\Z)$, as usual.  Setting $\Delta = M_2(\Z) \cap \GL_2(\Q)$, we also
have associated monoids
\begin{eqnarray} \label{Delta0def}
	&\Delta_0(N) = \left\{ \smatrix{a&b\\c&d} \in \Delta \mid (d,N) = 1 \text{ and } N \mid c \right\}, \\
 	&\label{Delta1def}  \Delta_1(N) = \left\{ \smatrix{a&b\\c&d} \in \Delta_0(N) \mid d \equiv 1 \bmod N  \right\}.
\end{eqnarray}

In this section, we specialize our cocycle $\Theta$ at the $N$-torsion point 
\begin{equation}  \label{sdef} 
	s \colon \Spec \Q(\mu_N) \rightarrow \gm^2
\end{equation} 
with value $(1,\zeta_N) \in \gm(\Q(\mu_N))^2$ to obtain a cocycle
$$
	\Theta_N \colon \tilde{\Gamma}_0(N) \to K_2(\Q(\mu_N))/\langle \{-1,-\zeta_N\} \rangle.
$$
 
\subsubsection{The specialized cocycle} 

We turn to the specialization of $\Theta$ at the $\Q(\mu_N)$-point $s$ of \eqref{sdef}, which is given by the map $z_1 \mapsto 1$ and $z_2 \mapsto \zeta_N$ on coordinate rings. The stabilizer of $s$ in $\GL_2(\Z)$ under its right action on $\gm(\Q(\mu_N))^2$ is the congruence subgroup
$\tilde{\Gamma}_1(N)$ of \eqref{tildegammadef2}.
For $j \in (\Z/N\Z)^{\times}$, let $\sigma_j \in \Gal(\Q(\mu_N)/\Q)$ be such that $\sigma_j(\zeta_N) = \zeta_N^j$.  Via the isomorphism
\begin{equation} \label{eazy_e}
	\tilde{\Gamma}_0(N)/\tilde{\Gamma}_1(N) \xrightarrow{\sim} (\Z/N\Z)^{\times}, \quad \smatrix{a & b \\ c & d} \mapsto d
\end{equation}
and its composite with $d \mapsto \sigma_d$, we may consider any $\Z[\Gal(\Q(\mu_N)/\Q)]$-module as a $\Z[\tilde{\Gamma}_0(N)]$-module. In particular, we let $\tilde{\Gamma}_0(N)$ act on $K_2(\Q(\mu_N))$ in this fashion.

Note that $s^*$ is not well-defined on the whole of $\varK_2$: there is no field map $\Q(\gm^2) \to \Q(\mu_N)$.  In order to pull back
the values of our cocycle via $s$, we show that for $\gamma = \smatrix{a & b \\ c & d} \in \tilde{\Gamma}_0(N)$, any lift to $\varK_2$ of $\Theta_{\gamma}$ lies in a sufficiently small subgroup of $\varK_2$ upon which $s^*$ can be defined.  

For this, let 
$$
	U_{\gamma} = \gm^2 - S_{b,d} \cup S_{0,1},
$$ 
which is to say the complement in $\gm^2$ of the subtori that are the kernels of $(z_1, z_2) \mapsto z_1^b z_2^d$ and $(z_1,z_2) \mapsto z_2$.
Since $(1,\zeta_N) \in U_{\gamma}$, the specialization map 
$$
	s^* \colon H^2(U_{\gamma},2) \to H^2(\Q(\mu_N),2) \cong K_2(\Q(\mu_N))
$$ 
is well-defined.
The residue of $\Theta_{\gamma}$ in $\varK_1$ is $\langle \det(\gamma)(b,d) \rangle - \langle 0,1 \rangle$.
Therefore, any lift of $\Theta_{\gamma}$ to $\varK_2$ is defined on $U_{\gamma}$ in the sense of Remark \ref{definedonU}.

Proposition \ref{Thetalifts} provides a canonical lift of $\Theta$ to a cocycle valued in $\varK_2/\langle \{-z_1, -z_2\} \rangle$, which
we also denote by $\Theta$. The value $\Theta_{\gamma}$ lies inside
$$\barK_2(N) :=\varinjlim_{s \in U}\, H^2(U,2)/\langle (-z_1) \cup (-z_2) \rangle,$$
where the limit runs over the open $\Q$-subschemes $U$
of the $\Q$-scheme $\gm^2$ containing $s$.  
 Specialization at $s$ now defines a morphism
\begin{equation} \label{sdef2} s^* \colon \barK_2(N) \rightarrow K_2(\Q(\mu_N))/Z_N\end{equation}
where 
  $Z_N$  is  the $\tilde{\Gamma}_0(N)$-stable subgroup of $K_2(\Q(\mu_N))$ generated by the
specialization $\{-1,-\zeta_N\}$ of the symbol $\{-z_1,-z_2\} \in K_2(\Q(\gm^2))$ under $s^*$.  
 It therefore makes sense to speak of 
 $$
 	\Theta_{N,\gamma} := s^*\Theta_{\gamma}
$$ as an element of $K_2(\Q(\mu_N))/Z_N$,  
 Note that
$$
	\{-1,-\zeta_N\} = \begin{cases}
	\{-1,-1\} & \ifs N \text{ is odd}, \\ 
	0 & \ifs N \text{ is even}, 
	\end{cases}  
$$
so $Z_N$ is a group of order dividing $2$.  
	
\begin{proposition} \label{cyclcocycle}
	The map
	$$
		\Theta_N \colon \tilde{\Gamma}_0(N) \to K_2(\Q(\mu_N))/Z_N, \quad \gamma \mapsto \Theta_{N,\gamma}
	$$
	is a parabolic cocycle.
 \end{proposition}

\begin{proof}
	That $\Theta_N$ is a cocycle follows from if we can show that $s^*$ as in \eqref{sdef2} is a homomorphism of 
	$\tilde{\Gamma}_0(N)$-modules. 
	Here, note that	
	$\bGamma_0(N)$
	acts on $K_2(\Q(\mu_N))$ as in \eqref{eazy_e}. 
	Write $\gamma = \smatrix{a&b\\c&d}$ and note that the the two maps
	$\Spec \Q(\mu_N) \rightarrow \mathbb{G}_m^2$
	defined by 
	 $\gamma \circ s$ and $s \circ \sigma_d$
	  coincide, since viewed on coordinate rings 
	both send $z_1$ to $1 = \zeta_N^c$ and $z_2$ to $\zeta_N^d$. 
	In particular, $\bGamma_0(N)$ acts on $\ms{K}_2(N)$, since for any $\Q$-subscheme $U$ of $\gm^2$ with $s \in U(\Q(\mu_N))$, 
	the composition $\gamma \circ s = s \circ \sigma_d$ is also a $\Q(\mu_N)$-point of $U$.
	This implies that  
	$
		s^* \circ \gamma^* = \sigma_d \circ s^*
	$
	on  $\barK_2(N)$.

	The proof of Proposition \ref{parabolic} argued that $\Theta$ is trivial on a lower-triangular 
	parabolic $P_{\infty}$ in $\GL_2(\Z)$. An arbitrary parabolic $Q$ of $\tilde{\Gamma}_0(N)$ has the form
	$Q = \mu P_{\infty} \mu^{-1} \cap \tilde{\Gamma}_0(N)$ for some $\mu \in \SL_2(\Z)$. For $\gamma \in P_{\infty}$, we have
	$\Theta_{\mu\gamma\mu^{-1}} = (1-(\mu\gamma\mu^{-1})^*)\Theta_{\mu}$. So long as $\Theta_{\mu}
	\in \ms{K}_2(N)$, we then have that $\Theta|_Q \colon Q \to \ms{K}_2(N)$ is a coboundary, and for this, it suffices that
	$\mu = \smatrix{a'&b'\\c'&d'}$ satisfies $N \nmid d'$. On the other hand, the set of $\mu$ with $N \mid d'$ is exactly the coset 
	$\tilde{\Gamma}_0(N)\smatrix{0&1\\1&0}$, and in this case $Q = \mu P_{\infty} \mu^{-1}$. We may then 
	suppose $\mu = \smatrix{0&1\\1&0}$, for which $Q = \{ \smatrix{\pm 1 & n \\ 0 & 1} \mid n \in \Z\}$. Since 
	$\Theta_N$ lifts to a map taking the value $\{1-\zeta_N,1-\zeta_N^{-1}\} = 0$ on $\smatrix{1 & -1 \\ 0 & 1}$
	and the value $\{-1,-\zeta_N\} \in Z_N$ on $\smatrix{-1 & 0 \\ 0 & 1}$, it is trivial on $Q$, and we have parabolicity.
\end{proof}

We need to modify the notion of connecting sequence of \S \ref{cocycle} adapted to the level $N$ structure.
Specifically, let us refer to a connecting sequence $(b_i,d_i)_{i=0}^k$ for $\gamma \in \tilde{\Gamma}_0(N)$ with the property that $N \nmid d_i$
for all $0 \le i \le k-1$ as an \emph{$N$-connecting sequence} for $\gamma$.
Note that an $N$-connecting sequence always exists, as we verify by a little fiddling.

\begin{lemma} \label{fiddling}
	Given a primitive vector $(b,d) \in \Z^2$ with $N \nmid d$, there exists a sequence $(v_i)_{i=0}^k$ in $\Z^2$ 
	with $v_i \wedge v_{i+1} = 1$ for $0 \le i < k$ such that $v_0 = (0,1)$, $v_k = (b,d)$, and $v_i = (b_i,d_i)$ with $N \nmid d_i$ 
	for all $0 \le i \le k$.
\end{lemma}

\begin{proof}
	Choose any connecting sequence $(v_i)_{i=0}^k$ with $v_k = (b,d)$.  Suppose that $v_i$ has second coordinate divisible by $N$.
	Then neither $v_{i-1}$ nor $v_{i+1}$ does. We will insert another sequence between $v_{i-1}$ and $v_{i+1}$,
	no element of which has second coordinate divisible by $N$.  For $v, w \in \Z^2$, consider $v \wedge w$ as an integer via
	the identification of $\bigwedge^2 \Z^2$ with $\Z$ using the basis vector $(1,0) \wedge (0,1)$.  Note that $ 
		v_{i-1} \wedge v_i = v_i \wedge v_{i+1} = 1$ 
	for $1 \le i \le k-1$. 
 Write $t = v_{i-1} \wedge v_{i+1}$. 
	We suppose that $t \geq 1$, the other case being easier.  
 	The sequence with $v_{i-1},v_i,v_{i+1}$ replaced by 
	$v_{i-1}, v_{i+1}+(1-t) v_i, \ldots, v_{i+1} - v_i, v_{i+1}$
 	has nearly the desired properties
	since 
	$$
		v_{i-1} \wedge (v_{i+1} + (1-t)v_i) = 
		(v_{i+1}-(j-t)v_i) \wedge (v_{i+1} - (j+1-t)v_i) = 
		 1 
	$$
	for $1 \le j \leq t-1$. 
	However, the last pair of adjacent vectors $x = v_{i+1}-v_i$ and $y = v_{i+1}$ satisfies $x \wedge y = -1$,
 	rather than $1$.  To remedy this, we replace $x,y$ by the sequence
	 $x, -y, -x, y$.
\end{proof}

\begin{remark} \label{smallNcplx} 
	As in Remark \ref{symbolcomplex}, Lemma \ref{fiddling} tells us that $\Theta$ restricted to $\tilde{\Gamma}_0(N)$ 
	takes values in (a quotient of) the degree $2$ term of the subcomplex of the small complex $\varC$ spanned by the 
	$\varC_{\gamma}$ for those $\gamma = \smatrix{a&b\\c&d} \in \SL_2(\Z)$ such that $N \nmid c$ and $N \nmid d$.
\end{remark}

We then have the following explicit formula for our cocycle.

\begin{proposition} \label{explicitspecialized}
	Let $\gamma = \smatrix{a & b \\ c &d} \in \tilde{\Gamma}_0(N)$, and let $(b_i,d_i)_{i=0}^k$ be an $N$-connecting sequence for 
	$\gamma$.
	Then
	$$
		\Theta_{N,\gamma} = \sum_{i=1}^k \{ 1- \zeta_N^{d_i}, 1-\zeta_N^{-d_{i-1}} \}.
	$$
\end{proposition}

\begin{proof}
	By Proposition \ref{expform}, we need only note that 
	$$
		s^*\langle (b_i,d_i),(-b_{i-1},-d_{i-1}) \rangle = \{1-\zeta_N^{d_i},1-\zeta_N^{-d_{i-1}}\}
	$$
	for $1 \le i \le k$. 
\end{proof}

Since $1-\zeta_N^c$ is an $N$-unit for all $c \not\equiv 0 \bmod N$ and the map $K_2(\Z[\mu_N,\frac{1}{N}]) \to K_2(\Q(\mu_N))$ is
an injection, we have the following corollary.

\begin{corollary} \label{N-integrality}
	The cocycle $\Theta_N$ takes values in $K_2(\Z[\mu_N,\frac{1}{N}])/Z_N$.
\end{corollary}

\begin{remark}
	One could use Remark \ref{integral} to avoid the explicit formula in proving this corollary (supposing the same expected property of
	integral motivic cohomology), since pullback by $(1,\zeta_N)$ defines a morphism $\ms{K}_{2/\Z}(N) \to K_2(\Z[\mu_N,\frac{1}{N}])$,
	where $\ms{K}_{2/\Z}(N) = \ms{K}_{2/\Z} \cap \ms{K}_2(N)$.
\end{remark}

In fact, we can do slightly better.

\begin{lemma} \label{finerintegrality}
	If $N$ is divisible by two distinct primes, then $\Theta_N$ takes values in $K_2(\Z[\mu_N])/Z_N$.  Otherwise,
	its restriction to $\Gamma_1(N)$ does.
\end{lemma}

\begin{proof}
	Fix a prime $\ell$ dividing $N$.  Let $F_{\ell}$ denote the residue field at a prime of $\Q(\mu_N)$ over $\ell$, and
	consider the tame symbol map 
	$$
		\delta_{\ell} \colon K_2(\Z[\mu_N,\tfrac{1}{N}])/Z_N \to F_{\ell}^{\times},
	$$
	of \eqref{tamesymbol}.
	The common kernel of the maps $\delta_{\ell}$ is $K_2(\Z[\mu_N])/Z_N$. Thus, it suffices 
	to see that $\delta_{\ell} \circ \Theta_N$ is trivial on the congruence subgroups of interest. 
	
	Suppose first that $N$ is divisible by two distinct primes.  For any prime $p \mid N$ with $p \neq \ell$,
	we have that $\tilde{\Gamma}_0(N) \subset \tilde{\Gamma}_0(p)$, so there exists
	a $p$-connecting sequence $(b_i,d_i)_{i=0}^k$ for any $\gamma \in \tilde{\Gamma}_0(N)$. 
	But then each $1-\zeta_N^{d_i}$ is a unit locally at primes over $\ell$, so 
	$\delta_{\ell}(\{1-\zeta_N^{d_i},1-\zeta_N^{-d_{i-1}}\})$ vanishes. By Proposition 
	\ref{explicitspecialized}, we then have $\delta_{\ell}(\Theta_{N,\gamma}) = 1$, independent of $\ell$.
	
	Next, suppose that $N$ is a power of a prime $\ell$.  Given $\gamma = \smatrix{a&b\\c&d} 
	\in \tilde{\Gamma}_0(N)$, there exists an $\ell$-connecting sequence $(b_i,d_i)_{i=0}^k$ for $\gamma$.  Then
	each $1-\zeta_N^{d_i}$ has valuation $1$ at $(1-\zeta_N)$, so
	$$
		\delta_{\ell}(\{1-\zeta_N^{d_i},1-\zeta_N^{-d_{i-1}}\}) = -\frac{1-\zeta_N^{-d_{i-1}}}{1-\zeta_N^{d_i}} \bmod
		(1-\zeta_N),
	$$
	which reduces to $\frac{d_{i-1}}{d_i}$ in $\F_{\ell}^{\times}$.  Proposition \ref{explicitspecialized} then yields that
	$\delta_{\ell}(\Theta_{\gamma}) = \det(\gamma) d^{-1} \bmod \ell$, which is trivial if $\gamma \in \Gamma_1(N)$.
\end{proof}

\subsubsection{Hecke equivariance} \label{Hecke-equiv}

We next consider the Hecke equivariance of $\Theta_N$.  
Let us set $\Phi = \langle \{-z_1,-z_2\} \rangle \subset \varK_2$ for simplicity, which we also view as a subgroup of $H^2(\gm^2,2)$.
Over the next few lemmas, we show that the class of $\Theta$ in $H^1(\GL_2(\Z),\varK_2/\Phi)$ is annihilated by all of
the operators $T_{\ell}-\ell-[\ell]^*$ for odd primes $\ell$, as well as by $2(T_2-2-[2]^*)$, in order to show the analogous Eisenstein property of $\Theta_N$ in Theorem
\ref{ThetaNEis}.
  
\begin{lemma} \label{K2s}
For any finite index subgroup $\Gamma$ of $\GL_2(\Z)$, the inclusion 
$\barK_2(N) \hookrightarrow \varK_2/\Phi$
induces an injection on $H^1(\Gamma, -)$.
\end{lemma}

\begin{proof}
From Gysin sequences, we see that there is an exact sequence
\begin{equation} \label{GysinK2N} 
	0 \rightarrow \barK_2(N) \rightarrow \varK_2/\Phi  \rightarrow \bigoplus_{s\in D} k(D)^{\times},
\end{equation}
where the sum ranges over divisors $D$ containing $s$.
The lemma follows if we know that the finite index subgroup $\Gamma$ of $\bGamma_0(N)$
has trivial invariants on the right-hand group. 

We claim that the orbit of any divisor $D$ containing $(1,\zeta_N)$ on $\mathbb{G}_m^2$ under $\GL_2(\Z)$ is infinite, so no element
of the direct sum can be fixed by the finite index subgroup $\Gamma$.
 Such a divisor is the vanishing locus of an $f \in \Q[z_1^{\pm 1}, z_2^{\pm 1}]$ that 
 is unique up to units, i.e., up to some $c z_1^i z_2^j$ with $c \in \Q^{\times}$ and $i, j \in \Z$. 
Define the support $\mathrm{supp}(f)$ of $f$ to be  the set of $(a,b) \in \Z^2$ for which
the coefficient of $z_1^a z_2^b$ is nonzero. Note $|\mathrm{supp}(f)| \geq 2$. For any hyperbolic element $\gamma \in \Gamma$, the   
diameter of $\mathrm{supp}(f)\gamma^n$ 
increases without bound as $n \rightarrow \infty$. In particular, $\mathrm{supp}(f)\gamma^n$ cannot be a translate of $\mathrm{supp}(f)$ for sufficiently large $n$, and therefore $D \gamma^n \neq D$ for all $n \ge 1$.
\end{proof}

\begin{lemma} \label{4-torsion} 	The kernel of the map on $H^1(\GL_2(\Z),-)$ induced by the quotient map $\varK_2/\Phi \twoheadrightarrow \barK_2$ is $2$-torsion.
\end{lemma}
 
\begin{proof} 
	Recall from Corollary \ref{cohomgmr1} and Lemma \ref{fixedpart} that $H^2(\gm^2,2)$ is the direct sum of subgroups
	generated by symbols of the form $a \cup b$, $(-z_1) \cup b$, $(-z_2) \cup b$, and $(-z_1) \cup (-z_2)$ with $a, b \in \Q^{\times}$.
	It follows that, as a $\Z[\GL_2(\Z)]$-module, the group 
	$H^2(\Gm^2, 2)/\Phi$ is a direct sum of copies of modules $A$ and $W \otimes_{\Z} A$ with $A$ having trivial $\GL_2(\Z)$-action,
 	where $W$ is the group $\Z^2$ endowed with the standard left $\GL_2(\Z)$-action.
 
 	By the universal coefficient sequence (which is split), it is then enough to verify that the groups $H_1(\GL_2(\Z), \Z)$
 	and $H_i(\GL_2(\Z), W)$ for $i \in \{0,1\}$ are $2$-torsion.  In fact $H_1(\GL_2(\Z), \Z) \cong \GL_2(\Z)^{\mr{ab}} \cong 
	(\Z/2\Z)^2$. Since
	$\SL_2(\Z)$ acts transitively on $W-\{0\}$,   the group $H_0(\GL_2(\Z),W)$ vanishes, and
	 $H_1(\GL_2(\Z), W)$ is a quotient of $H_1(\SL_2(\Z),W)$. It then suffices to show the latter group is killed by $2$. 
	
	The group $\SL_2(\Z)$ is an amalgamated free product of the cyclic $4$-subgroup generated
	by $S = \smatrix{0 & 1\\ -1 & 0}$ and the cyclic $6$-subgroup generated by $T = \smatrix{1 & -1\\ 1 & 0}$
 	over the $2$-subgroup generated by $\sigma=S^2=T^3$.  Thus we have a Mayer-Vietoris sequence
 	$$
		\cdots \to H_1(\langle S \rangle,W) \oplus H_1(\langle T \rangle,W) \rightarrow H_1(\SL_2(\Z),W) 
 	\rightarrow H_0(\langle \sigma \rangle,W) \rightarrow \cdots.
	$$
 	The first two groups vanish because neither $S$ nor $T$ have invariants, and the last is $(\Z/2\Z)^2$. 
\end{proof}

 \begin{lemma}  \label{ThetaEis2} 
 	The operator $T_{\ell}-\ell-[\ell]^*$ kills the class of $\Theta$ (resp., $2 \Theta$) in 
 	$H^1(\GL_2(\Z),\varK_2/\Phi)$ for $\ell \neq 2$ (resp., for $\ell=2$).
\end{lemma}

\begin{proof}
	Let $\tau_{\ell}$ denote the class of $(T_{\ell}-\ell-[\ell]^*) \Theta$ in the latter group.
	Lemma \ref{ThetaEis} implies that $\tau_{\ell}$ lies in the kernel of the homomorphism
	\begin{equation} \label{fdef}
		f \colon H^1(\GL_2(\Z),\varK_2/\Phi) \to H^1(\GL_2(\Z), \barK_2)
	\end{equation}
	of Lemma \ref{4-torsion}, so is $2$-torsion.
	 In particular, we have the statement for $\ell = 2$.
	
	This kernel of $f$ is a quotient of $H^1(\GL_2(\Z),H^2(\gm^2,2)/\Phi)$.
	By Lemma \ref{fixedpart} and the decomposition \eqref{iterated}, the latter group is a direct sum of subgroups on which either 
	every $[m]_*$ acts by one of the scalars $m$ and $m^2$.
 	In particular,   $\tau_{\ell}$ is killed by any operator
	$$\epsilon = \sum_{m=1}^{\infty} a_m [m]_*$$
	with $a_m \in \Z$ and $\sum ma_m = \sum m^2 a_m=0$. 

	Next, let us note that the actions of $T_{\ell}$ and $[m]_*$ on $H^1(\GL_2(\Z), \varK_2/\Phi)$
	commute if $\ell \nmid m$ since by Lemma \ref{pullbackdiagram}, the trace map $[m]_*$ commutes with the pullbacks 
	used in the definition of $T_{\ell}$ for $\ell \nmid m$.
	Consequently, $\tau_{\ell}$ is fixed by each such $[m]_*$.
	Therefore, so long as the $a_m$ are zero for $m$ not prime to $\ell$, the operator
	$\epsilon$ above acts on $\tau_{\ell} $ by the scalar $\sum a_m$, and we conclude that
	$ \sum a_m \cdot \tau_{\ell}=0$
	for $a_m$ as above such that $a_m = 0$ if $\ell $ divides $m$. 
 	Taking $a_1 = -a_2 = 3a_3 = 3$, we see that  $\tau_{\ell}$ is zero  
 	for $\ell \ge 5$ as desired. Taking $a_4 = 1$, $a_2= -6$, and $a_1 = 8$, we see that $3\tau_3 = 0$,
	which is sufficient as $2\tau_3 = 0$ as well.
 \end{proof}
 
Let $\Delta_0(N)$ be the monoid of matrices with lower-left entry divisible by $N$ and lower-right entry
prime to $N$: see \eqref{Delta0def}. 
For a prime $\ell$ and $g = \smatrix{ \ell \\ & 1 }$ as before, we let $T_{\ell} = T(g)$ for $\ell \nmid N$, 
 with $T(g)$ as in Section \ref{actions}.

\begin{theorem} \label{ThetaNEis}  
	For primes $\ell$ not dividing $2N$, we have
	$$
		T_{\ell}\Theta_N= (\ell +\sigma_{\ell}) \Theta_N
	$$
	in $H^1(\bGamma_0(N),K_2(\Q(\mu_N))/Z_N)$.  If $N$ is odd, we have $2(T_2-2-\sigma_2)\Theta_N = 0$.\footnote{In fact,
	one can verify by explicit computation that at least the restriction of $(T_2-2-\sigma_2)\Theta_N $ to $\Gamma_1(N)$
	is trivial.} 
\end{theorem}

\begin{proof}
	In Lemma \ref{ThetaEis2} we proved 
	that for $\ell \nmid 2N$, the cocycle $(T_{\ell} -\ell -[\ell]^*)\Theta$ is cohomologous to zero when considering $\Theta$ as a 
	$\GL_2(\Z)$-cocycle
	with target $\varK_2/\langle\{-z_1, -z_2\}\rangle$.  For $\ell \nmid N$, the elements $g_j$ of \eqref{Gjdef} lie
	in $\Delta_0(N)$ and still provide left coset representatives of $\bGamma_0(N) \smatrix{\ell \\ & 1} \bGamma_0(N)$.
	By Lemma \ref{K2s}, the class of $(T_{\ell} -\ell -[\ell]^*)\Theta$ then remains zero
	when $\Theta$ is considered as a $\bGamma_0(N)$-cocycle with target $\barK_2(N)$.
	
	Moreover,  the map $s^* \colon \barK_2(N) \rightarrow K_2(\Q(\mu_N))/Z_N$, 
	 is equivariant for the action of 
	$\Delta_0(N)$ in the sense that $\sigma_d \circ s^* = s^* \circ \delta$,
	where $\delta =\abcdmat \in \Delta_0(N)$; in particular $\sigma_{\ell} \circ s^* = s^* \circ [\ell]^*$.	
	Therefore, 
	$$s^* (T_{\ell} - \ell - [\ell]^*) \Theta = (T_{\ell} - \ell - \sigma_{\ell}) \Theta_N$$
	is cohomologous to zero,  as a cocycle with target in $K_2(\Q(\mu_N))/Z_N$.
	
	The same argument goes through for $\ell = 2$ if $N$ is odd by multiplying everything by $2$.
\end{proof}

\subsection{Maps on the homology of $X_1(N)$} \label{homology}

In this section, we compare our constructions with others in the literature. We show how the cocycle $\Theta_N$ induces a map on the homology of the usual closed modular curve $X_1(N)$ over $\C$, which is to say the quotient of the extended upper half-plane $\mb{H}^*$ by the congruence subgroup $\Gamma_1(N)$ of $\SL_2(\Z)$.
This agrees with the map constructed independently by Busuioc \cite{busuioc} and the first author \cite{sharifi}, which can be
defined explicitly on Manin symbols on a slightly larger homology group of $X_1(N)$, taken relative to some of its cusps.  We show
that  this induced map factors through the quotient of homology by an Eisenstein ideal away from the level, providing a complement to a result
of Fukaya and Kato \cite{fk} on $p$-parts for $p \mid N$ that was a conjecture of the first author.
 
\subsubsection{Maps defined on Manin symbols}

Let us suppose that $N \ge 4$.  Let $C_1(N) = \Gamma_1(N) \backslash \mb{P}^1(\Q)$ denote the cusps in the modular curve $X_1(N)$, which is taken over $\C$ in this section.  
For $\alpha, \beta \in \mb{P}^1(\Q)$, let $\{\alpha \to \beta\}$ denote the class
in the relative homology group $H_1(X_1(N),C_1(N),\Z)$ of the geodesic in $\mb{H}^*$ from $\alpha$ to $\beta$.  
If $\alpha$ and $\beta$ are equivalent cusps, then $\{\alpha \to \beta \}$ lies in the homology of $X_1(N)$.

Let us set
$$
	\vgamma = \{0 \to \gamma \cdot 0\} \in H_1(X_1(N),\Z)
$$
for $\gamma \in \Gamma_1(N)$.  This class is independent of the choice of element $0 \in \mb{H}^*$,
and there is a commutative diagram 
$$
\begin{tikzcd}
\Gamma_1(N) \arrow[dr,"\gamma \mapsto \vgamma"']  \ar[r] & H_1(Y_1(N), \Z) \ar[d]  \\
& H_1(X_1(N), \Z),
\end{tikzcd}
$$
where the horizontal and vertical arrow are the standard maps, and all three maps are surjections.

For an abelian group $M$ with an action of complex conjugation, we let $M_+$ denote the maximal quotient on which complex
conjugation acts trivially.   

\begin{proposition} \label{maphomology}
	There is a unique $(\Z/N\Z)^{\times}$-equivariant homomorphism 
	$$
		\Pi_N \colon H_1(X_1(N),\Z)_+ \to K_2(\Z[\mu_N])/Z_N.
	$$
	that sends the image of $\vgamma$ to $\Theta_{N,\gamma}$ for all $\gamma \in \Gamma_1(N)$.
\end{proposition}

\begin{proof}
    	Since the action of $\tilde{\Gamma}_0(N)$ on $K_2(\Q(\mu_N))$ is trivial on $\tilde{\Gamma}_1(N)$,
    	the restriction of $\Theta_N$ to $\Gamma_1(N)$ induces a $(\Z/N\Z)^{\times}$-equivariant homomorphism 
    	\begin{equation} \label{first}
		H_1(Y_1(N),\Z)  \xrightarrow{\sim} H_1(\Gamma_1(N), \Z) \to K_2(\Z[\mu_N,\tfrac{1}{N}])/Z_N,
	\end{equation}
   	where $d \in (\Z/N\Z)^{\times}$ acts  by diamond operators on the first term and by the Galois element $\sigma_d$
    	with $\sigma_d(\zeta_N) = \zeta_N^d$ on the last. 
    	This homomorphism actually takes values in the subgroup $K_2(\Z[\mu_N])/Z_N$ by Lemma \ref{finerintegrality}.
    
    	The composition in \eqref{first} factors through  $H_1(Y_1(N), \Z)  \rightarrow H_1(Y_1(N), \Z)_+$, since it
    	is invariant by the natural action $Q := \tilde{\Gamma}_1(N)/\Gamma_1(N)$ on the left-hand side.
    This $Q$ is a group of order $2$, and its nontrivial
    element acts on  $H_1(Y_1(N),\Z)$ by complex conjugation $z \mapsto -\bar{z}$.
	
	Finally, the composition in \eqref{first} also factors through $H_1(Y_1(N), \Z) \rightarrow H_1(X_1(N), \Z)$.
	That is, the cocycle $\Theta_N$ is a coboundary, hence trivial, on all parabolic subgroups of $\Gamma_1(N)$, 
	which are right stabilizers of nonzero elements of $\mb{P}^1(\Q)$.
	These parabolics
	are also left stabilizers of elements 
	of $\mb{P}^1(\Q)$ inside $\mb{H}^*$ and thereby
   	generate the kernel of $\Gamma_1(N)^{\mathrm{ab}} \rightarrow H_1(X_1(N), \Z)$. 
\end{proof}

Let $C_1^{\circ}(N) \subset C_1(N)$ denote the set of cusps not lying over $\infty \in \Gamma_0(N) \backslash \mb{P}^1(\Q)$.
Given $u, v \in \Z/N\Z$ with $(u,v) = (1)$, let 
$$
	[u:v] = \left\{ \frac{b}{d} \to \frac{a}{c} \right\} = \gamma \{ 0 \to \infty \},
$$ 
where $\smatrix{a&b \\ c&d} \in \SL_2(\Z)$ with $(u,v) = (c,d) \bmod N\Z^2$.  These Manin symbols
for $u, v \neq 0$ generate $H_1(X_1(N),C_1^{\circ}(N),\Z)$.  In fact, this relative homology group has a presentation on the 
Manin symbols with relations
\begin{eqnarray}\label{maninsymbol}
	[u:v] = -[-v:u] &\mr{and}&
	[u:v] = [u:u+v] + [u+v:v],
\end{eqnarray}
the latter for $u \neq -v$ (cf.~\cite[3.3.7]{fk} and \cite[\S 5.4]{sharifi-aws}).
It also has an action of diamond operators $\langle j \rangle$ for $j \in (\Z/N\Z)^{\times}$, given explicitly by
$$
	\langle j \rangle [u:v] = [ju:jv].
$$

Let us set $\Z' = \Z[\frac{1}{2}]$.  
In general, for an abelian group $M$ with an action of complex conjugation, 
let us use $m_+$ to denote the image
of $m \in M$ in $(M \otimes_{\Z} \Z')_+$.
The presentation of $H_1(X_1(N),C_1^{\circ}(N),\Z')_+$ as a $\Z'$-module on the generators $[u:v]_+$
has the additional relations $[u:v]_+ = [-u:v]_+$ for all $u, v \neq 0$.  

The following construction is due to Busuioc \cite{busuioc} and the first author \cite[Proposition 5.7]{sharifi}.  We give a proof that also gives some idea of where it becomes necessary to invert $2$.
 
\begin{proposition}[Busuioc, Sharifi] \label{comparison}
	There is a $(\Z/N\Z)^{\times}$-equivariant homomorphism
	$$
		\Pi^{\circ}_N \colon H_1(X_1(N),C_1^{\circ}(N),\Z')_+ \to (K_2(\Z[\mu_N,\tfrac{1}{N}]) \otimes_{\Z} \Z')_+, \quad
		[c:d]_+ \mapsto \{1-\zeta_N^c,1-\zeta_N^d \}_+.
	$$
\end{proposition}

\begin{proof}
 	For $\alpha, \beta \in \Z[\mu_N,\frac{1}{N}]^{\times}$, we denote by $\{\alpha, \beta\}_+$
	the projection of the Steinberg symbol to  $(K_2(\Z[\mu_N,\frac{1}{N}]) \otimes_{\Z} \Z')_+$.
	 Since we kill $2$-torsion, we have
	\begin{equation} \label{v1}  \{-1, \alpha\}_+ = \{\zeta_N, \zeta_N\}_+ = 0. \end{equation} 
 Now take $x,y \in \mu_N-\{1\}$. Then  
		$$
		 \{ 1-x, y \}_+ = \frac{1}{2}(\{ 1-x, y \}_+ + \{ 1-x^{-1}, y^{-1} \}_+)  = \frac{1}{2}\{ -x, y \}_+   \stackrel{\eqref{v1}}{=} 0,
	$$
	where the first equality is from invariance under complex conjugation and the second uses bilinearity. Therefore,
	$\{1-\zeta_N^a, 1-\zeta_N^b\}_+$ is invariant under changing the sign of either $a$ or $b$,
	whence the first relation of \eqref{maninsymbol}.	
	The second relation of \eqref{maninsymbol} follows from this invariance and 
 	$$ 
		\{1-x, 1-x^{-1} y^{-1}\} + \{1-xy, 1-y^{-1}\} = \{1-x, 1-y^{-1}\}
	$$
	for $xy \neq 1$, this equality holding  without inverting $2$ and taking quotients trivial under complex conjugation.
 	In turn, this follows from the relation $\{\eta, 1-\eta\}=0$
 	with 
	\begin{eqnarray*}
		\eta = \frac{1-x}{1-xy}&\mr{and}&  1-\eta 
		 = \frac{1-y^{-1}}{1-x^{-1}y^{-1}}.
	\end{eqnarray*}
\end{proof}

The restriction of $\Pi^{\circ}_N$ to $H_1(X_1(N),\Z)_+$ agrees with the map induced by our cocycle $\Theta_N$.
The first statement in the following is due to Fukaya and Kato \cite[Theorem 5.3.3]{fk} for $p \mid N$, after taking $\zp$-coefficients,
and in general, a direct proof can be found in \cite[Lemma 5.4.1]{sharifi-aws}.  For us, the first statement follows from the second, as 
$\Pi_N$ takes values in $K_2(\Z[\mu_N])/Z_N$ by Proposition \ref{maphomology} (following Lemma \ref{finerintegrality}, which
is related to the aforementioned results).

\begin{proposition}
	The restriction of $\Pi^{\circ}_N$ to $H_1(X_1(N),\Z)_+$ takes values in $(K_2(\Z[\mu_N]) \otimes_{\Z} \Z')_+$ 
	and agrees with the composition of $\Pi_N$ with the quotient map from $K_2(\Z[\mu_N,\frac{1}{N}])/Z_N$.
\end{proposition}

\begin{proof}
	We may write any element of $H_1(X_1(N),\Z)$ as $\vgamma$ for some 
	$\gamma \in \Gamma_1(N)$.  Let $(b_i,d_i)_{i=0}^k$ be an $N$-connecting sequence for this $\gamma$, so in particular
	$(b_0,d_0) = (0,1)$ and $(b_k,d_k) = (b,d)$.  Then
	$$
		\vgamma = \left\{ 0 \to \frac{b}{d} \right\} = \sum_{i=1}^k \left\{ \frac{b_{i-1}}{d_{i-1}} \to \frac{b_i}{d_i} \right \}
		= \sum_{i=1}^k [d_i:d_{i-1}]_+ = \sum_{i=1}^k [d_i:-d_{i-1}]_+
	$$
	is sent by $\Pi_N^{\circ}$ to
	$\sum_{i=1}^k \{1-\zeta_N^{d_i}, 1-\zeta_N^{-d_{i-1}} \}_+$.
	By Proposition \ref{explicitspecialized}, this sum is the image of $\Theta_{N,\gamma} = \Pi_N(\vgamma)$. 
\end{proof}

\subsubsection{Eisenstein property} \label{Eisprop}

For a prime $\ell$, we define the Hecke operator $T_{\ell}$ (denoted $U_{\ell}$ if $\ell \mid N$)  on $H_1(X_1(N),C_1(N),\Z)$ to be that arising from a right coset decomposition of $\Gamma_1(N) \smatrix{ 1 \\ & \ell } \Gamma_1(N)$. Its adjoint, or dual, $T_{\ell}^*$ is similarly the right Hecke operator for $\smatrix{\ell \\ & 1}$ (denoted $U_{\ell}^*$ if $\ell \mid N$).  

\begin{remark}
The operators $T_{\ell}$ on relative homology are adjoint to the corresponding right coset operators on compactly supported cohomology $H^1_c(Y_1(N),\Z)$, which project to operators that agree with the (left coset) operators $T_{\ell}$ on $H^1(\Gamma_1(N),\Z)$ previously defined by Remark \ref{leftright}.  

Note that $H_1(X_1(N),C_1(N),\Z)$ is the left $\Gamma_1(N)$-coinvariant group of the group of degree zero divisors in $\Z[\mb{P}^1(\Q)]$ under the standard left $(M_2(\Z) \cap \GL_2(\Q))$-action. Thus, if we choose a set of right coset representatives for the double coset of $\smatrix{1 \\ & \ell}$ and define $T_{\ell}$ on $\Z[\mb{P}^1(\Q)]$ by the sum of their actions, then this induces the $T_{\ell}$-action on relative homology.
\end{remark}

The adjoint operators preserve the subgroup $H_1(X_1(N),C_1^{\circ}(N),\Z)$, but the operators $U_{\ell}$ for $\ell \mid N$ do not.
Let us consider the adjoint Hecke algebra 
\begin{equation} \label{adjhecke}
	\mb{T}^*_N \subset \End_{\Z}(H_1(X_1(N),C_1^{\circ}(N),\Z)),
\end{equation}
which also acts on $H_1(X_1(N),\Z)$.  Inside this algebra, we have the prime-to-level and full Eisenstein ideals
\begin{eqnarray*}
	I'_N = (T^*_{\ell}- 1 - \ell\langle \ell \rangle^* \mid \ell \nmid N \text{ prime})
&
\mr{and} 
&
	I_N = I'_N + (U^*_{\ell}-1 \mid \ell \mid N \text{ prime}).
\end{eqnarray*}
Since $\langle \ell \rangle^* = \langle \ell \rangle^{-1}$ and 
$T_{\ell}^* = \langle \ell \rangle^{-1} T_{\ell}$ for $\ell \nmid N$ in $\mb{T}^*_N$, note that
\begin{equation} \label{compareEis}
	T^*_{\ell} - 1 - \ell\langle \ell \rangle^* = \langle \ell \rangle^{-1} (T_{\ell} - \ell - \langle \ell \rangle).
\end{equation}

The first author has frequently floated the following conjecture that $\Pi_N$ is \emph{Eisenstein}, so factors through the quotient of homology by the action of $I_N$ (or equivalently, that $\Pi_N(Tx) = 0$ for all $T \in I_N$ and $x \in H_1(X_1(N),\Z')_+$) and, moreover, induces an isomorphism on the quotient by $I_N$.

\begin{conjecture}[Sharifi] \label{eisconj} \	\begin{enumerate}
		\item[a.] The map $\Pi_N$ factors through a map
		$$
			\varpi_N \colon H_1(X_1(N),\Z')_+ \otimes_{\mb{T}^*_N} \mb{T}^*_N/I_N \to (K_2(\Z[\mu_N]) \otimes_{\Z} \Z')_+.
		$$
		\item[b.] The map $\varpi_N$ is an isomorphism.
	\end{enumerate}
\end{conjecture} 

Part a of Conjecture \ref{eisconj} is a stronger form of an earlier conjecture \cite[Conjecture 5.8]{sharifi} that the tensor product of $\Pi_N$ with the identity on $\zp$ for a prime $p \mid N$ is Eisenstein.  The earlier conjecture was proven by Fukaya and Kato in \cite[Theorem 5.3.5]{fk}.\footnote{In fact, \cite{sharifi}, the first author constructed a conjectural inverse to $\varpi_N \otimes \mr{id}_{\zp}$ on most primitive eigenspaces in the case that $p \nmid \varphi(N)$, and Fukaya and Kato proved an important result in its direction in \cite{fk}.} In fact, they showed the 
following stronger result.

\begin{theorem}[Fukaya-Kato]
	For $p \mid N$, the map $\Pi_N^{\circ} \otimes_{\Z} \id_{\zp}$ factors through a map
	$$
		\varpi^{\circ}_N \colon H_1(X_1(N),C_1^{\circ}(N),\zp)_+ \otimes_{\mb{T}^*_N} \mb{T}^*_N/I_N
		\to (K_2(\Z[\mu_N,\tfrac{1}{N}]) \otimes_{\Z} \zp)_+.
	$$
\end{theorem}

Though we expect that $\Pi^{\circ}_N$ is Eisenstein in general, the induced map $\varpi^{\circ}_N$ is not always an isomorphism.  A special case of this conjecture is considered by Lecouturier in \cite[Conjecture 4.32]{lecouturier} (see also Conjecture 4.33 therein, 
which follows from the latter conjecture).

The proof of the result of Fukaya and Kato arises through a description of $\Pi_N$ as the composition of two maps: first, a Hecke-equivariant map $z_N$ that takes Manin symbols to cup products of Siegel units (i.e., Beilinson-Kato elements), and second, a specialization map induced by pullback at the cusp $0$.  The proof of the Hecke equivariance of $z_N$ goes through a string
of Iwasawa-theoretic and Hida-theoretic constructions and the computation of a $p$-adic regulator.  Their result then follows from the fact that the specialization at zero factors through $I'_N$ and is also trivial on the operators $U^*_{\ell}-1$ applied to Beilinson-Kato elements. 

Though we do not use it to study $\Pi_N$, we give a construction of a motivic version of the map $z_N$ and prove its prime-to-level Hecke equivariance in Section \ref{universal}.  Instead, as a consequence of what we have already done, we obtain a result over $\Z'$ for the prime-to-level Eisenstein ideal without any use of Beilinson-Kato elements.\footnote{The unpublished manuscript \cite{stevens} of Stevens contains another approach through which it may be possible to obtain this result.}  
In fact, by Theorem \ref{ThetaNEis}, we have the following.

\begin{theorem} \label{varpi}
	The map $\Pi_N$ factors through a map
	$$
		\varpi_N \colon H_1(X_1(N),\Z)_+ \otimes_{\mb{T}^*_N} \mb{T}^*_N/I'_N \to K_2(\Z[\mu_N])/Z_N.
	$$
\end{theorem}

\begin{proof}
	From Theorem \ref{ThetaNEis}, we have that $T_{\ell}\Theta_N = (\ell+\sigma_{\ell})\Theta_N$ as homomorphisms 
	from $\Gamma_1(N)$ to $K_2(\Z[\mu_N])/Z_N$.  Since $\sigma_{\ell} \circ \Pi_N = \Pi_N \circ \langle \ell \rangle$
	and noting \eqref{compareEis}, it suffices by Proposition \ref{comparison} to check that 
	$$
		(T_{\ell}\Theta_N)_{\gamma} = \Pi_N(T_{\ell}\vgamma).
	$$
	
	For $g = \smatrix{\ell \\ & 1}$, we may choose left coset representatives 
	of $\Gamma_1(N)g\Gamma_1(N)$ as in \eqref{doublecoset}
	with bottom right entry $1$ modulo $N$ as follows: for $1 \le j < \ell$, set $g_j = \smatrix{\ell & j \\ & 1}$,
	and set  $g_{\ell} = \delta_{\ell} \smatrix{1 \\ & \ell}$ with $\delta_{\ell}$ as above.  These agree with the matrices in \eqref{Gjdef}
	aside from $g_{\ell}$.
	For the map $\Pi_N$ constructed in Proposition \ref{maphomology}, for $\gamma \in \Gamma_1(N)$,
	we then have
	$$
		(T_{\ell}\Theta_N)_{\gamma} = \sum_{j=0}^{\ell} g_{\sigma(j)}^* \Theta_{N,\gamma_j} = \sum_{j=0}^{\ell} \Theta_{N,\gamma_j} 
		= \sum_{j=0}^{\ell} \Pi_N(\vgamma_j),
	$$
	where $\gamma g_j = g_{\sigma(j)} \gamma_j$ for $\sigma$
	a permutation of $\{0,\ldots,\ell\}$ and $\gamma_j \in \Gamma_1(N)$.
	
	On the other hand, let $h_j = \smatrix{\ell  \\ & \ell}g_j^{-1}$ be the adjoint of $g_j$ so that
	$h_j \gamma^{-1} = \gamma_j^{-1} h_{\sigma(j)}$.
	Since $\{ \alpha \to \beta \} + \{ \beta \to \epsilon \} = \{ \alpha \to \epsilon \}$ for $\alpha, \beta,\epsilon
	\in \mb{H}^*$ and $\{ 0 \to \mu^{-1} \cdot 0 \} = -\vec{\mu}$ for $\mu \in \Gamma_1(N)$, we have
	$$
		T_{\ell}\vgamma = -T_{\ell} \{0 \to \gamma^{-1} 0\} 
		= -\sum_{j=0}^{\ell} \{ h_j0 \to \gamma_j^{-1} h_{\sigma(j)} 0 \}
		= -\sum_{j=0}^{\ell} \{ h_{\sigma(j)} 0 \to \gamma_j^{-1} h_{\sigma(j)} 0 \}
		= \sum_{j=0}^{\ell} \vgamma_j,
	$$
	hence the result.
\end{proof}

\section{The $\gm^2$-cocycle via toric geometry} \label{K2boundary} 
This section exists to provide a different viewpoint on the above results
and minor improvements upon some of them. We will describe a map
$$ \mbox{chain complex of $S^1$} \longrightarrow [\varK_2 \rightarrow \varK_1]$$
in the derived category of abelian groups with $\GL_2(\Z)$-action. 
This map can be used to recover the previous cocycle, and even lift it to $\varK_2$.
Moreover it allows us to outline the connection of our results with 
  equivariant motivic cohomology, as discussed in \S \ref{emc}. 
  The key point in the argument is  to utilize 
  the behavior of $K_2$ classes along the boundary of toric compactifications. 

  The geometric construction  that we give is  closely related to joint work in progress of the second-named
author with Bergeron, Charollois and Garcia (although that work does not deal with $K$-theory, rather with differential forms).  
However, the viewpoint of this section is also close to that taken by numerous other authors on related questions, 
among which we mention Nori, Sczech, Stevens, Solomon, and Garoufalidis--Pommersheim \cite{Nori, Scz, So, stevens, GP}.
Particularly relevant is a very recent paper of Lim and Park \cite{lim-park}, which completes the work of Stevens and lifts a ``Shintani cocycle'' to the ``Stevens cocycle'' along a $\mr{dlog}$ map.
 As with several of the named references,
\cite{lim-park} works with cocycles for $\GL_n(\Q)$ valued in a module of distributions; as such, it does not directly relate
to the type of toric geometry that we emphasize here but nonetheless seems very closely related to 
  an infinite level version of our construction.

At certain points
one could proceed by symbols and relations, but we have tried to avoid this.
Our point of view would extend without complication to higher dimensions, for instance.

 \subsection{Residues on $K_2$ of the function field of a torus} \label{restorus}

\subsubsection{Some toric geometry}  \label{toricsetup}
  
It will be helpful to proceed a bit more canonically. Let $T = \gm^2$,
let $X = X_*(T)$ be the cocharacter group of $T$, and set $X_{\R}  = X \otimes_{\Z} \R \cong \R^2$. 
Fix an orientation on $X_{\R}$, which in particular allows us to make the identification $\bigwedge^2 X \cong \Z$;
for $x,y \in X$, we accordingly write $x \wedge y \in \Z$. 
Let $X^* = X^*(T)$ be the character group of $T$, and 
denote by 
$$
	\langle \ \, , \ \rangle \colon X \times X^* \to \Z
$$
the pairing that  describes the composition $\gm \to \gm$.

Let us view the torus $\gm^2$ as $T = \Spec \Q[X^*]$, and let $\Q(T)$ be the function field of $T$.
For each primitive $\lambda \in X$, let $V_{\lambda} \subset X^*$ be the dual  
cone of characters which pair {\em non-positively} with $\lambda$. 
 Let $\Q[V_{\lambda}]$ be the monoid algebra of $V_{\lambda}$.
 Since each element $\nu$ of $V_{\lambda}$ is a regular function
 on $T$, we have inclusions
 $$ \Q[V_{\lambda}] \hookrightarrow \Q[X^*] \hookrightarrow \Q(T).$$
 In particular, the first inclusion  
induces an open immersion $T \rightarrow T_{\lambda}$,
where 
$$
	T_{\lambda} = \Spec \Q[V_{\lambda}].
$$

The toric variety $T_{\lambda}$ has the following properties, all of which are readily proven
by choosing coordinates.\footnote{For instance, we may suppose
that $\lambda$ is the cocharacter $t \mapsto (t^{-1}, 1)$ of $\mathbb{G}_m^2$, that
$V_{\lambda}$ is the group of characters $(z_1, z_2) \mapsto z_1^i z_2^j$ with
$i \geq 0$, and that $T_{\lambda}$ is the compactification $\mathbb{A}^1 \times \mathbb{G}_m$.}  
The limit $Q_{\lambda} = \lim_{x \rightarrow \infty} \lambda(x)$ exists in the partial compactification $T_{\lambda}$ of $T$.
In other words, the map $t \mapsto \lambda(t)$, considered as a morphism $\mathbb{G}_m \rightarrow T_{\lambda}$,
extends over $\infty \in \mathbb{P}^1$. 
 The complement $D_{\lambda} = T_{\lambda} - T$ is a divisor on $T_{\lambda}$,
and $Q_{\lambda}$ belongs to this divisor.  The 
vanishing order of any $\chi \in X^*$   along $D_{\lambda}$
given by $-\langle \lambda, \chi \rangle$. 
The stabilizer of $Q_{\lambda}$ under the torus action of $T$ on $T_{\lambda}$ is precisely $\lambda(\mathbb{G}_m)$, and this provides
a $T$-equivariant  identification
$$ D_{\lambda} \cong T/\lambda(\mathbb{G}_m) $$
under which $Q_{\lambda}$ is taken to the identity.
Moreover, any choice of $\mu \in X$ with $\mu \wedge \lambda = 1$ induces an isomorphism
$\mathbb{G}_m \xrightarrow{\mu} T \rightarrow T/\lambda(\mathbb{G}_m)$
which permits us to identify $D_{\lambda}$ with $\gm$.
 
 \subsubsection{Residues of classes in $K_2(\Q(\gm^2))$}
 \label{bdy}

We continue with the notation of \S \ref{toricsetup}. Our key result, Proposition \ref{reslambda} below,
describes the boundary behavior of classes in $K_2$ of the function field of $T$ along toric boundary divisors. 

 Let $S^1$ denote the circle, viewed as the quotient of $X_{\R}-\{0\}$ by positive scalings:
 $$
 	S^1 = (X_{\R}-\{0\})/\R_+.
$$
 We shall identify points of $S^1$ with rays $\R_+ x \subset X_{\R}$
 for $x \in X_{\R}$ nonzero 
 (i.e., half-lines with boundary the origin). 
 A point in $S^1$ is {\em rational} if it is the image of an element of $X$, i.e., if the associated ray passes through a point of $X$.

\begin{proposition} \label{reslambda}
	For $\kappa \in K_2(\Q(T))$, there is a 
	locally constant function $n=n_{\kappa} \colon S^1 - \Sigma \rightarrow \Z$
	with $\Sigma$ a finite set of rational points, having the following property:
	\begin{quote}
	if $n$ is defined on the ray $\R_{+} \lambda$,  then the residue of $\kappa$ along $D_{\lambda} \cong \mathbb{G}_m$
	has the form $c z^{n(\lambda)}$ for some scalar $c \in \Q^{\times}$.
 	\end{quote}
\end{proposition}

\begin{proof}
It is sufficient to analyze the case that
$\kappa = \{f, g\}$, where $f$ and $g$ are nonzero elements of $\Q[X^*]$, for such symbols generate $K_2(\Q(T))$.  

We may write any $f \in \Q[X^*]$ as a finite sum
$$
	f= \sum_{\chi \in X^*} a_{\chi}(f) \chi
$$
with $a_{\chi}(f) \in \Q$.  Let $\A(f)$ denote the finite set 
\begin{equation} \label{supp}
	\A(f) = \{ \chi \in X^* \mid a_{\chi}(f) \neq 0 \}.
\end{equation}

For a nonzero $\lambda \in X
\otimes_{\Z} \R$, 
consider the function $\phi_{f,\lambda} \colon \A(f) \rightarrow \R$ given by 
$$
	\phi_{f,\lambda}(\chi) = \langle \lambda, \chi \rangle
$$
on $\chi \in \A(f)$.  If $\phi_{f,\lambda}$ is injective on the finite set $\A(f)$, then we let $\chi_{f, \lambda}$
be the unique element $\chi \in \A(f)$ maximizing $\langle \lambda, \chi \rangle$.
It is invariant under rescaling $\lambda$ by a positive real number.  
Letting $\Sigma_f$ denote the (finite!) collection of rays $\R_+ \lambda$ for which $\phi_{f,\lambda}$ is not injective, we then have a locally
constant function
 $$
 	\Rays - \Sigma_f \rightarrow X^*, \qquad \R_+ \lambda \mapsto \chi_{f,\lambda}.
 $$ 
 
Now fix a rational point of $S^1 -\Sigma_f$, corresponding to the ray $\R_+ \lambda$
for some primitive $\lambda \in X$, and write
$$\chi_f = \chi_{f,\lambda}, \quad a_f = a_{\chi_f}(f), \quad \mbox{and} \quad v_f = -\langle \lambda, \chi_f \rangle.$$
As above, $v_f$ is the vanishing order of $f$ along $T_{\lambda}$; it may be negative. 
 Note that
$f \cdot \chi_f^{-1}$
extends to $T_{\lambda}$, because for any $\chi \in \A(f)$,
the ratio 
$\chi \chi_f^{-1}$ has non-positive pairing with $\lambda$.  Since $\chi\chi_f^{-1}$ vanishes on $D_{\lambda}$ for $\chi \in \A(f)$ with 
$\chi \neq \chi_f$, the value of $f \cdot \chi_f^{-1}$ along $T_{\lambda}-T$ is the constant $a_f$.

Now suppose that $\R_+ \lambda \notin \Sigma_f$ also does not belong
to the set $\Sigma_g$ for $g \in \Q[X^*]$.
The image of $\{f, g \}$ in $K_1(\Q(D_{\lambda}))$ is therefore the tame symbol given by
$$ (-1)^{v_fv_g} \frac{g^{v_f}}{f^{v_g}} =
c \frac{\chi_g^{v_f}}{\chi_f^{v_g}},
$$
where $c$ is the constant $(-1)^{v_fv_g} a_g^{v_f}a_f^{-v_g} \in \Q^{\times}$. 
The right-hand side defines a function on $T$, constant on $\lambda(\mathbb{G}_m)$,  which
extends over $T_{\lambda}$, and thus can be restricted to $D_{\lambda}$. 

Note that the value of $\chi_g^{v_f}\chi_f^{-v_g} \in X^*$
on a cocharacter $\mu$ is given by
$$ \langle \mu,\chi_f\rangle \langle \lambda,\chi_g \rangle - \langle \lambda,\chi_f \rangle \langle \mu,\chi_g \rangle = 
(\mu \wedge \lambda)(\chi_f \wedge \chi_g).$$
Recall that we are identifying $D_{\lambda}$ with $\mathbb{G}_m$
via any cocharacter $\mu \colon \mathbb{G}_m \rightarrow T$ with $\mu \wedge \lambda=1$;
with respect to this identification,  the tame symbol above is identified with $c z^n$, where $n = \chi_f\wedge \chi_g$. 
In particular, $n=n(\lambda)$ is locally constant on the set $\Rays - \Sigma_f \cup \Sigma_g$ of rays. 
\end{proof}

\begin{example} \label{firstquadrant}
	Take $f = 1-z_1$ and $g = 1-z_2$. 
	The sets $\A(f)$ and $\A(g)$ in \eqref{supp} are $\{0, (1,0)\}$ and $\{0, (0,1)\}$, respectively. 
	Then $\Sigma_f$  consists of the ray $\R_+ (0,1)$ together with its negative, and $\Sigma_g$ is the ray $\R_+ (1,0)$
	together with its negative.   
	For $\lambda = (a,b)$, we have
	\begin{eqnarray*} 
		\chi_f = \begin{cases} (1,0)  & \ifs a >0 \\ 0  & \ifs a < 0 \end{cases} &\text{and}&
		\chi_g = \begin{cases} (0,1) & \ifs b>0 \\ 0 & \ifs b<0. \end{cases}
	\end{eqnarray*}
	Therefore, if we choose the standard orientation where $(1,0) \wedge (0,1) = 1$, 
	then  $n=\chi_f \wedge \chi_g$ is given by the function 
	\begin{equation} \label{EX12}
		(a,b) \mapsto \begin{cases} 1 & \ifs a>0 \mbox{ and }  b>0, \\ 0 & \ifs a <0 \mbox{ or } b<0, \end{cases}
	\end{equation} 
	which is to say, the characteristic function of the counterclockwise
	arc from $(1,0)$ to $(0,1)$ on $S^1$,  or equivalently of the first quadrant in $\R^2$. 
\end{example}

\subsubsection{Values on symbols}

The map $n$ of Proposition \ref{reslambda}
can be described as a homomorphism
from $K_2(\Q(T))$ to the
set of locally constant functions on $S^1$, defined on the complement
of a finite set of rational points. If we identify two such functions 
when they agree off of a finite set, then the target becomes
a group under pointwise addition, and the map $n$ a group homomorphism. 
 Call this group 
 $\Cones_1$:
 $$ \Cones_1 = \{\mbox{$\Z$-valued locally constant functions on $S^1-\Sigma$, with $\Sigma \subset S^1_{\Q}$ finite}\}/\sim,$$
 where $\sim$ is the equivalence relation of agreeing off of a finite set.

\begin{example} 
For any $\ell, \ell' \in S^1_{\Q}$. 
let $[\ell, \ell']$ be the {\em counterclockwise} arc from $\ell$ to $\ell'$,
which we identify with an element of  $\Cones_1$ via its characteristic function.
(Thus, if $\ell=\ell'$, then $[\ell,\ell']$ is the zero element.) 

Observe that the group $\Cones_1$ has a presentation with generators the elements $[\ell, \ell']$ and relations  
\begin{equation} \label{ellbasic} [\ell, \ell''] =[\ell, \ell'] + [\ell', \ell''] \end{equation}
for $\ell'$ lying on the counterclockwise arc from $\ell$ to $\ell''$ (including both endpoints).
Indeed, writing $\mathsf{G}$ for the abstract group so presented,
the homomorphism $\mathsf{G} \rightarrow \Cones_1$ is readily seen to be surjective.
On the other hand, by recursive use of \eqref{ellbasic}, any element of $\mathsf{G}$ can be written
as a finite sum $\sum_i m_i [a_i, b_i]$ where $m_i \in \Z$ and where the intervals are disjoint
except at their endpoints, and the condition of vanishing in $\Cones_1$
then implies that the sum must be empty.

We can then reformulate Example \ref{firstquadrant}
as saying that
$ \{1-z_1, 1-z_2\} \mapsto [(1,0), (0,1)]$ under $n$. 
More generally, if $\nu_1, \nu_2$ form a positively oriented basis of $X$ (i.e., $\nu_1 \wedge \nu_2=1$),
then 
\begin{equation} \label{fq2} n \colon \{1-\nu_1^*, 1-\nu_2^*\} \mapsto [\R_+\nu_1,\R_+ \nu_2],\end{equation}
where $\nu_1^*, \nu_2^* \in X^*$ are the dual basis elements.
\end{example}

\begin{remark}
For later use, we note that for any $\ell_1, \ell_2, \ell_3 \in S^1_{\Q}$ we have
 \begin{equation} \label{above} [\ell_1, \ell_3] =   [ \ell_1, \ell_2 ]  + [ \ell_2, \ell_3 ]  - \delta(\ell_1, \ell_2, \ell_3), 
\end{equation}
 where $\delta = 0$ when $\ell_2$ lies on the counterclockwise arc from $\ell_1$ to $\ell_3$ including endpoints,
 and $\delta=1$ otherwise.  (In particular,  $[\ell_1, \ell_2] + [\ell_2, \ell_1] = 1$ unless $\ell_1=\ell_2$). 
 Note that it follows from this that
 $\delta$ satisfies the  (homogeneous) cocycle relation
\begin{equation} \label{deltacocycle} \delta(\ell_1, \ell_2, \ell_3) - \delta(\ell_0, \ell_2, \ell_3)+ \delta(\ell_0, \ell_1, \ell_3) - \delta(\ell_0, \ell_1, \ell_2) = 0.\end{equation} 
\end{remark}

\subsection{Comparison of chain complexes} \label{main}

We continue to suppose that $T = \mathbb{G}_m^2$, providing an identification $X = \Z^2$. 
The right automorphism group of $T$ is an algebraic group which, consistent with our prior conventions, we 
consider as acting on the right on $T$.
By functoriality, we obtain the usual right action of $\GL_2(\Z)$ on $X$ regarded as row vectors.

\subsubsection{Alternate description of the chain complex for $S^1$}
 
The group $\Cones_1$ introduced above fits into a chain complex that computes the 
homology of $S^1$:

\begin{lemma} \label{nabla}
Let $\Cones_0$ be the group of finitely-supported $\Z$-valued functions on the rational points of $S^1$,
and define $\nabla \colon \Cones_1 \rightarrow \Cones_0$ via  $$
	\nabla f(x) = f(x^-) - f(x^+),
$$
where $f(x^+)$ (resp., $f(x^-)$) is the limit of $f(y)$ as $y$ approaches $x$ clockwise (resp., counterclockwise).
Then there is an isomorphism
$$ 
	[\Cones_1 \xrightarrow{\nabla} \Cones_0] \xrightarrow{\sim} \Chains_*(S^1)
$$  
in the derived category of $\Z[\GL_2(\Z)]$-modules, where $\Chains_*(S^1)$ denotes the singular chain complex of $S^1$.
Here, the left $\GL_2(\Z)$-action on both sides is induced by the right $\GL_2(\Z)$-action on $X$. 
 \end{lemma}
 
 \begin{proof}
  Indeed, the complex 
  $$\cdots \rightarrow \Chains_2(S^1) \xrightarrow{d_2} \Chains_1(S^1) \xrightarrow{d_1} \Chains_0(S^1) \rightarrow 0.$$
   of singular chains is  quasi-isomorphic to 
its truncation
$\coker(d_2) \rightarrow \Chains_0(S^1)$.
There is an obvious injection $\Cones_0 \rightarrow \Chains_0(S^1)$, as well as a $\GL_2(\Z)$-equivariant map
$\Cones_1 \rightarrow  \coker(d_2)$
which sends  $[\ell, \ell']$ to the singular simplex  $[0,1] \rightarrow S^1$
that proceeds at constant speed from $\ell$ to $\ell'$;
to verify this is well-defined one just checks the relation \eqref{ellbasic}.  

Since
\begin{equation} \label{lll} \nabla [\ell, \ell'] = 1_{\ell'} - 1_{\ell},\end{equation} 
these maps provide a morphism of complexes.
To see that is a quasi-isomorphism, note that the homology of the
complex  $[\Cones_1 \xrightarrow{\nabla} \Cones_0] $
is $\Z$ in both degrees. That is, the cokernel of $\nabla$ is 
generated by the image of any function that assigns a single rational point on $S^1$
the value $1$,
and the kernel of $\nabla$ is generated by the constant function with value $1$ in $\Cones_1$. 
These map to generators of $H_0(S^1,\Z)$ and $H_1(S^1,\Z)$, respectively.
 \end{proof}

\subsubsection{The motivic complex via toric geometry}

For $v \in X$ primitive, let $[\R_+v]$ denote the characteristic function of the image of $\R_+v$ in $S^1$. 
 The indexing of symbols in the following proposition differs from our prior indexing of symbols in $\varK$,
which was effectively done by characters, rather than cocharacters.  
 
\begin{proposition}  \label{quasi-iso}
	There is a morphism
     	$$
     		f \colon [\Cones_1 \rightarrow \Cones_0 \rightarrow \Z] \rightarrow [\varK_2 \rightarrow \varK_1 \rightarrow \varK_0],
    	$$
    	of complexes of $\Z[\GL_2(\Z)]$-modules, 
     	where the right-hand complex $\varK$ is the coniveau complex computing $H^*(\Gm^2, 2)$ of \S \ref{symbolsmult}.
    	More explicitly, there is a commutative diagram
      	$$
     		\begin{tikzcd}
     		\Cones_1 \arrow{r}{\nabla} \arrow{d}{f_2} & \Cones_0  \arrow{d}{f_1} \arrow{r}{g \mapsto \sum_{P} g(P)} & \Z \arrow{d}{f_0}  \\
    		K_2(\Q(T)) \arrow{r}{\partial} &   \bigoplus_{D} K_1(\Q(D)) \arrow{r} &  \bigoplus_{x} K_0(k(x)), 
     		\end{tikzcd}
     	$$
    	where $f_0(1) = e$ (as in \S \ref{symbolsmult}), 
	$$
		f_1([\R_+ v]) = 1-v^{\inv}|_{\nu(\mathbb{G}_m)},
	$$ 
	where $v \in X$ is primitive and $v^{\inv} \colon v(\gm) \to \gm$ denotes the inverse of $v$, 
 	and
	\begin{equation} \label{f2def}
	 	f_2([\R_+ \nu_1,  \R_+ \nu_2]) = \{1-\nu_1^*, 1-\nu_2^*\}
	\end{equation} 
	for $\nu_1, \nu_2 \in X$ with $\nu_1 \wedge \nu_2 = 1$ and $\nu_1^*, \nu_2^* \in X^*$ the dual basis to $\nu_1, \nu_2$. 
 	Moreover, $f_2$ sends the constant function with value $1$ to the symbol $\{-z_1, -z_2\}$.
 \end{proposition}

 \begin{proof}  
     	It is clear that unique $f_0$ and $f_1$ exist having the specified values
     	and that the right-hand square is commutative. 
	For the $\GL_2(\Z)$-equivariance of $f_1$, we compute 
$$
	f_1(\gamma \cdot [\R_+ \nu]) = f_1([\R_+ \nu\gamma^{-1}]) = (1-(\nu\gamma^{-1})^{\inv})|_{\nu\gamma^{-1}(\gm)} = \gamma^* f_1([\R_+ \nu]),
$$ 
where, on the right, $\gamma$ acts as usual by pullback of the right $\Gamma$-action on $T$.
 	Since $\nabla$ is injective, it remains only to construct $f_2$ satisfying \eqref{f2def}, 
     	and to verify that the left-hand square commutes.
     
    	Observe that if $\nu_1, \nu_2 \in X \cong \Z^2$ satisfy $\nu_1 \wedge \nu_2 = 1$ and have   
     	dual basis $\nu_1^*$, $\nu_2^*$, then by \eqref{tamesymbol}, we have
     	\begin{equation} \label{simple} 
     		\partial \{1-\nu_1^*, 1-\nu_2^*\} =  (1-\nu_2^{\inv})|_{\nu_2(\mathbb{G}_m)} - (1-\nu_1^{\inv})|_{\nu_1(\mathbb{G}_m)}.
	\end{equation}
     	which, together with \eqref{lll}, shows that the left-hand square formally commutes on $[\R_+ \nu_1, \R_+ \nu_2] \in \Cones_1$,
	given the property \eqref{f2def}.
     
        	Let $\Symb_2$ be the subgroup of $K_2(\Q(T))$ generated by all symbols $\{1-\nu_1^*, 1-\nu_2^*\}$ with $\nu_i^* \in X^*(T)$ 
	and $\nu_1 \wedge \nu_2=1$.   Then by \eqref{fq2} we have a commutative diagram:
     	$$
     	\begin{tikzcd}[column sep=small]
         	\Symb_2 \arrow[d, "n", "\textrm{Prop}. \ref{reslambda}"']  \arrow[r,"\partial"] & 
          	\bigoplus_{D} K_1(k(D))  &   \{1-\nu_1^*, 1-\nu_2^*\} \arrow[d, mapsto] \arrow[r, mapsto] &  
          	(1-\nu_2^{\inv})|_{\nu_2(\mathbb{G}_m)} - (1-\nu_1^{\inv})|_{\nu_1(\mathbb{G}_m)}  \\
         	\Cones_1 \arrow[r, "\nabla" ,""']   &  \Cones_0  \arrow[u]& \mbox{$[\R_+ \nu_1, \R_+ \nu_2]$} \arrow[r,mapsto] & 
        		{[\R_+ \nu_2] - [\R_+ \nu_1].} \arrow[u,mapsto]
     	\end{tikzcd}
      	$$
     
     	We claim that the map $n$ of Proposition \ref{reslambda}
    	restricts to an isomorphism $ \Symb_2 \xrightarrow{\sim} \Cones_1$. 
      	Once this is proved, it follows from \eqref{simple} that we can take $f_2$ to be the inverse of $n$.
     
 	By \eqref{fq2} again, the image of $n$ on $\Symb_2$ contains all $[\ell, \ell']$, so $n$ is surjective.
    	For injectivity, take $\kappa \in \Symb_2$  with $n(\kappa) = 0$; then the diagram shows that $\partial \kappa = 0$. 
    	It follows, then, that $\kappa$ lies inside the image of $H^2(T,2)$.
    	Since $\kappa$ is $[m]_*$-fixed by Lemma \ref{mfixed}, it follows from Lemma \ref{fixedpart} that $\kappa$ is a multiple of 
	$\{-z_1, -z_2\}$. But as in Example \ref{firstquadrant}, the symbol
    	$$
    		\{-z_1, -z_2\} = \left\{ \frac{1-z_1}{1-z_1^{-1}}, \frac{1-z_2}{1-z_2^{-1}} \right\}
    	$$ 
    	maps under $n$ to the sum of of the characteristic functions of the four (strict) quadrants of $\R^2$,
    	which agrees in $\Cones_1$ with the constant function $1$.  Therefore, $f_2(1) = \{-z_1,-z_2\}$, and $\kappa = 0$. 
\end{proof}

\begin{remark}
   	The morphisms $f_i$ are degreewise injective, and the image of $f$ is the symbol complex $\Symb$ of \S \ref{symbolsmult}
  	and Remark \ref{symbolcomplex}. By said remark, the complex $[\Cones_2 \rightarrow \Cones_1 \rightarrow \Z]$ is thereby
	quasi-isomorphic to the fixed part of the small complex $\varC$ under the trace maps of \S \ref{fixedparts}.
\end{remark}

\subsection{The cocycle and Laurent series} \label{Laurent}

We relate our discussion to an invariant of rational cones  that has appeared in the literature, 
and we recover the cocycle $\Theta$ from the considerations of the prior subsection.

\subsubsection{Connection to cocycles valued in Laurent series}

We use ``exponential coordinates'' near the identity. That is,
given the coordinate functions $z_1, z_2$ on $\Gm^2$, 
we introduce  formal coordinates $u_1, u_2$ at the identity satisfying $z_i=e^{u_i}$.
For  $\nu=z_1^m z_2^n \in X^*$, we then formally have  $\nu=e^{m u_1 + nu_2}$. 
We will also regard the $u_i$ as being linear functions on the Lie algebra $\mathrm{Lie}(\mathbb{G}_m^2)$
via the isomorphism of formal groups
\begin{equation} \label{identification} 
(\mathbb{G}_m^2,1) \xrightarrow{\sim} (\mathrm{Lie}(\mathbb{G}_m^2), 0).
\end{equation}
 
 Consider the composite map
\begin{equation} \label{old cocycle} 
\theta_L \colon \Cones_1  \xrightarrow{f_2} K_2(\Q(\mathbb{G}_m^2)) \rightarrow \{\mbox{meromorphic $2$-forms on $\Gm^2$}\} \rightarrow \Q(\!(u_1,u_2)\!), \end{equation}
where the second map sends a Steinberg symbol $\{f, g\}$ to $\frac{df}{f} \wedge \frac{dg}{g}$,
 and the third map takes a meromorphic form $\omega$ to  $\frac{\omega}{du_1 \wedge du_2}$ in the Laurent series field $\Q(\!(u_1, u_2)\!)$, which we take to be the quotient field of $\Q\ps{u_1,u_2}$.
 
In particular,  
given $\nu_1,\nu_2 \in X$ with $\nu_1 \wedge \nu_2=1$ and dual basis written as $\nu_1^* = e^{\lambda_1}, \nu_2^* = e^{\lambda_2}$
(with the $\lambda_i$ linear forms in the $u_i$), we calculate $\theta_L$ on $[\R_+ \nu_1, \R_+ \nu_2]$ as follows:  
\begin{equation} \label{XXX2} [\R_+ \nu_1, \R_+ \nu_2] \mapsto \{1-\nu_1^*,1-\nu_2^*\} \mapsto  \frac{d\nu_1^*}{1-\nu_1^*} \wedge \frac{d\nu_2^*}{1-\nu_2^*}  \mapsto  \frac{1}{(1-e^{-\lambda_1})(1-e^{-\lambda_2})}
 \end{equation}
 We then regard the last term 
 as the element 
 $\sum_{\mu \in \Z_{\geq 0} \lambda_1 + \Z_{\geq 0} \lambda_2} e^{-\mu}$ 
 of $\frac{1}{\lambda_1\lambda_2} \Q\ps{u_1, u_2} \subset \Q(\!(u_1,u_2)\!)$.

 In fact we have;
 \begin{lemma}
Suppose that $\ell_1, \ell_2 \in S^1$ 
 with $\ell_1 \wedge \ell_2> 0$; then 
\begin{equation} \label{tL1} \theta_L([\ell_1, \ell_2]) = \sum_{\substack{\mu \in X^*\\ \langle \mu, \ell_i \rangle \geq 0}} e^{-\mu}.\end{equation}
\end{lemma}

\begin{proof}
 To verify this, we note that it is possible to choose
 a sequence $(x_0)_{i=0}^k$ in $X$ with $\ell_1 = \R_+ x_0$ and $\ell_2 = \R_+ x_k$ 
and $x_i \wedge x_{i+1}=1$ for all $0 \le i < k$ and, moreover, $x_i$ lie on the counter-clockwise arc
from $\ell_0$ to $\ell_1$.

We then apply \eqref{above} and \eqref{XXX2} recursively to obtain
$$\theta_L([\ell_1, \ell_2]) = \sum_{i=0}^{k-1} \theta_L([x_i, x_{i+1}]) = \sum_{i=0}^{k-1} \sum_{\mu \in C_k^*} e^{-\mu},$$
where $C_k^*$ is the dual cone to the cone spanned by $x_i, x_{i+1}$.  Now it was observed by Brion \cite[2.4, Th{\'e}or{\`e}me]{brion}
(see \cite[Prop 8.2(b)]{BP} for exposition and exact definitions) that the rule associating  a  cone $C$ to $\sum_{\mu \in C^*} e^{-\mu}$
 is additive with respect to decompositions of cones into subcones, and therefore the right-hand side is given by \eqref{tL1} as claimed.  
  \end{proof}

\subsubsection{Recovering the cocycle $\Theta \colon \GL_2(\Z) \rightarrow \barK_2$}
Now, and in the remainder of this section, we also work with the {\em left} action of $\Gamma=\GL_2(\Z)$ on $S^1$
via the rule $\gamma \cdot \R_+ \nu :=  \R_+ \nu \gamma^{-1}$, where we continue to think of $X$ 
as row vectors. (Equivalently, if we regard $X$ as column vectors, then $\gamma \in \Gamma$ acts by left multiplication by its transpose-inverse.)

Set $\ell_0= \R_+ (-1,0)$ so that $f_1(\ell_0)$ is given by the function $1-z_1^{-1}$ on the subtorus $\{ (x, 1) \mid x \in \gm \}$,
which is to say the symbol $\langle 0,1 \rangle$ of \S \ref{symbolsmult}.  Consider the function
\begin{equation} \label{theta recover} 
	\widetilde{\Theta} \colon \Gamma \rightarrow K_2(\Q(\gm^2)), \quad \widetilde{\Theta}(\gamma) = f_2([\ell_0, \gamma \ell_0]).
\end{equation}
Its composition with the quotient map to $\barK_2$ is a cocycle by virtue of \eqref{above}. 
In fact, this composition coincides with $\Theta$, as Proposition \ref{quasi-iso} and \eqref{lll} yield that the residue of $\widetilde{\Theta}(\gamma)$ is
$$ 
	\partial f_2([\ell_0,\gamma\ell_0]) = f_1(\nabla([\ell_0,\gamma\ell_0]))
	= f_1(1_{\gamma \ell_0} - 1_{\ell_0}) = (\gamma^*-1)f_1(\ell_0) = (\gamma^* -1) \langle 0,1 \rangle.
$$ 
 
\begin{example}
	Let us recover the formula for $\Theta_{\gamma}$ of Proposition \ref{expform}, where $\gamma = \smatrix{a&b\\c&d} \in \Gamma$.
	Here, for $\nu \in X$, we abbreviate $\R_+\nu$ by $\nu$.
	Given a connecting sequence $(v_i)_{i=0}^k$ for $\gamma$ as in \S \ref{cocycle},
	we can write $[(-1,0), (\det \gamma) (-d,b)]$ as a sum $\sum_{i=1}^k   [W v_{i-1}, W v_i]$ with $W = \smatrix{0&-1\\1&0}$.
	Now if $\nu_1 \wedge \nu_2= 1$, then $f_2$ sends $[\nu_1, \nu_2]$ to $\langle -W\nu_2, W \nu_1 \rangle$
    	in the notation of our previous section. 
	So 
	$$
		\tilde{\Theta}(\gamma) = 
		f_2( \left[ (-1,0), (\det \gamma) (-d,b) \right]) =  \sum_{i=1}^k f_2([W v_{i-1}, W v_i]) =     \sum_{i=1}^k  \langle  v_i, -v_{i-1} \rangle. $$  
	In effect, the notation of this section absorbed the negative signs by working with the character
	group of $\mathbb{G}_m^2$, rather than its cocharacter group.  
\end{example}

\subsection{Lifting the cocycle}  \label{lifting}

Let $\Theta$ be the cocycle of Proposition \ref{existence}. 
The obstruction to lifting the class of $\Theta$ to $\varK_2$ lies in the torsion group $H^2(\GL_2(\Z), \Z)$.  It follows that 
a multiple of $\Theta$ lifts. Here, we will show how to write down an explicit lift of the restriction of $12\Theta$ to $\SL_2(\Z)$ using Proposition \ref{quasi-iso}. 

To lighten notation, let us write $f_2(\ell_1, \ell_2)$ for $f_2([\ell_1, \ell_2])$ and $\thetaLaurent(\ell_1, \ell_2)$
for $\thetaLaurent([\ell_1,\ell_2])$. 
By \eqref{above} and the fact that $f_2(1) = \{-z_1,-z_2\}$ proven in Proposition \ref{quasi-iso}, we have
 \begin{equation} \label{property} f_2(\ell_1, \ell_2) + f_2(\ell_2, \ell_3) - f_2(\ell_1, \ell_3)  =  \delta(\ell_1, \ell_2, \ell_3) \{-z_1, -z_2\},\end{equation}
 and therefore by the definition \eqref{old cocycle} of $\thetaLaurent$, we have
 \begin{equation} \label{property1.5} \thetaLaurent(\ell_1, \ell_2) + \thetaLaurent(\ell_2, \ell_3) - \thetaLaurent(\ell_1, \ell_3)  =  \delta(\ell_1, \ell_2, \ell_3). \end{equation}

We will introduce a ``correction'' term to $\thetaLaurent$ to eliminate the right-hand side of \eqref{property1.5}.\footnote{One can extract the correction term from the literature by computing explicitly with Dedekind-Rademacher sums (compare Remark \ref{DR}),
but let us see how it comes out of our existing constructions. }  Here is the basic idea in a primitive form that does not quite work. Suppose 
there were a reasonable way to evaluate a value of $\thetaLaurent$ at the origin $(u_1, u_2) = (0,0)$. Let $\thetaLaurent^0$ denote the 
resulting function from pairs $(\ell_1, \ell_2)$ to $\Z$, which again satisfies \eqref{property1.5}. Then by the latter property and
\eqref{property}, the corrected function
$$f_2'(\ell_1, \ell_2) =f_2(\ell_1, \ell_2) - \thetaLaurent^0(\ell_1, \ell_2)  \{-z_1, -z_2\}$$
would be a homogeneous cocycle, i.e.,  
$
	f_2'(\ell_1, \ell_2) + f_2'(\ell_2, \ell_3) = f_2'(\ell_1, \ell_3).
$
Since $\{-z_1,-z_2\}$ has trivial image in $\barK_2$, the resulting $1$-cocycle
$\gamma \mapsto f_2'(\ell_0, \gamma \ell_0)$ would lift $\Theta$ from $\barK_2$ to $\varK_2$, as in \eqref{theta recover}. 

We will not be able to implement this precisely as stated, but we will
be able to do it after replacing the role of the rays $\ell_1, \ell_2$ by elements of $\SL_2(\Z)$. 
To make sense of the evaluation of a value of $\thetaLaurent \in \Q(\!(u_1, u_2)\!)$ at the origin $(u_1, u_2) = (0,0)$, we need the auxiliary data of the second column of the matrices in $\SL_2(\Z)$; this allows us to take a limit 
as $(u_1, u_2) \rightarrow 0$ in a specified direction. 

Elements in the image of $\thetaLaurent$ have the form $L^{-1} P$, with $P \in \Q\ps{u_1, u_2}$ 
and $L$ a product of linear forms.
Accordingly, we can unambiguously speak of its ``degree zero'' component $\thetaLaurent^0$,
namely $\frac{ P^{(\deg L)}} {L}$ with $P^{(\deg L)}$ the homogeneous component of $P$ of degree 
the degree of $L$. 
This degree zero component is now valued in the {\em rational} functions $\Q(u_1, u_2)$,
or more intrinsically via \eqref{identification} as rational functions on the Lie algebra $\Lie (\mathbb{G}_m^2)$. 

  For example, using $(1-e^{-x})^{-1} = x^{-1}  + \frac{1}{2} + \frac{x}{12} + O(x^2)$, we deduce that the degree zero component  in \eqref{XXX2}
  for $\nu_1, \nu_2 \in X$ with $\nu_1 \wedge \nu_2 = 1$ and dual basis $\nu_1^* = e^{\lambda_1}, \nu_2^* = e^{\lambda_2}$ is given by
\begin{equation} \label{tl3}
	\thetaLaurent^0( \nu_1, \nu_2) 
	= \frac{1}{4} + \frac{1}{12} \left(\frac{\lambda_1}{\lambda_2}+ \frac{\lambda_2}{\lambda_1}\right) \in  \Q(u_1, u_2).
\end{equation}

Now given auxiliary vectors $\nu_1'$ and $\nu_2'$ 
that are linearly independent from $\nu_1$ and $\nu_2$ respectively, we define
a regularized value of $\thetaLaurent$ by choosing a decomposition of $\thetaLaurent^0(\nu_1,\nu_2)$ as a sum
\begin{equation} \label{decomp} \thetaLaurent^0(\nu_1, \nu_2) = A_1 + A_2,\end{equation}
where $A_i$ is a  homogeneous rational function  in $(u_1, u_2)$ with poles only along the image of $\nu_i$; that is to say, if $\nu_i(t)=(t^{a_1},t^{a_2})$,
then $A_i$ has poles along the line spanned by $(a_1, a_2)$.   In the example above, for instance, we may take $A_1 = \frac{1}{4} + \frac{1}{12} \frac{\lambda_1}{\lambda_2}$ and  $A_2= \frac{1}{12} \frac{\lambda_2}{\lambda_1}$.  
We now define
\begin{equation} \label{tL0} \thetaLaurent^0(\nu_1, \nu_1', \nu_2, \nu_2') = A_1(\nu_1') + A_2(\nu_2') \in \tfrac{1}{12}\Z 
\end{equation}
Here $A_i(\nu_i')$ means that if $\nu_i'(t) = (t^{a_1}, t^{a_2})$, then we evaluate $A_i$ at  $(a_1, a_2)$. 
The decomposition $A_1+A_2$ isn't unique,  but it {\em is} unique up to the constant terms, 
so the right-hand side of \eqref{tL0} doesn't depend on the choice of decomposition.
 We should regard this
as a ``regularized value of $\thetaLaurent^0(\nu_1, \nu_2)$ at zero'',
where $\nu_1'$ and $\nu_2'$ are used to perform the regularization. 
 
\begin{proposition}\label{phi} \
\begin{enumerate}
    	\item[(i)]
    	There is a unique function $\phi \colon \SL_2(\Z) \times \SL_2(\Z) \rightarrow \frac{1}{12} \Z$
    	which is left $\SL_2(\Z)$-invariant and satisfies
     	\begin{equation} \label{phicocycle}  
     		\phi(\gamma_1, \gamma_2) + \phi(\gamma_2, \gamma_3) - \phi(\gamma_1, \gamma_3)  = 
		\delta(\gamma_1 \ell_0, \gamma_2 \ell_0, \gamma_3 \ell_0)
	\end{equation}
	for all $\gamma_1, \gamma_2, \gamma_3 \in \SL_2(\Z)$.
         \item[(ii)] For $\gamma_1, \gamma_2 \in \SL_2(\Z)$, set
         $\nu_i = \R_+ \gamma_i (-1,0)$ and $\nu_i' = \R_+ \gamma_i (0,-1)$ for $i \in \{1,2\}$.  If 
    $\R\nu_1 \neq \R\nu_2$, then 
      $$ \phi(\gamma_1, \gamma_2) =  \thetaLaurent^0(\nu_1, \nu_1', \nu_2, \nu_2').\footnote{One can also compute readily a formula in the other cases $\nu_1 = \pm \nu_2$, but we do not do so here for brevity. }$$
    \item[(iii)]  The function $\SL_2(\Z) \to \varK_2$ given by
  	$$
		\gamma \mapsto 12 f_2(\ell_0, \gamma \ell_0) - 12\phi(I_2, \gamma) \{-z_1, -z_2\} 
	$$
    	is a cocycle lifting $12 \Theta|_{\SL_2(\Z)}$ from $\barK_2$ to $\varK_2$, where $I_2$ denotes the $2$-by-$2$ identity matrix.
\end{enumerate}
\end{proposition}
 
 \begin{proof} 
	For the uniqueness in (i), note that the difference between any two such functions $\phi$ is a homogeneous cocycle. The group of 	such cocycles is $H^1(\SL_2(\Z),\frac{1}{12}\Z) = 0$, since the abelianization of $\SL_2(\Z)$ is torsion.
 	Part (iii) follows from the discussion at the beginning of the section, the cocycle being well-defined since $12\phi$ is $\Z$-valued.

 	Take $\gamma_1,  \gamma_2,  \gamma_3 \in \SL_2(\Z)$, and define $\nu_i,  \nu_i'$ for $i \in \{1,2,3\}$ accordingly, as in 
	the discussion preceding the proposition.
 	If the lines $\R \nu_1, \R \nu_2, \R \nu_3$ are all distinct (i.e., not merely the rays, but the lines themselves), then we claim that 
 	\begin{equation} \label{property2} 
		\thetaLaurent^0(\nu_1, \nu_1', \nu_2, \nu_2') + \thetaLaurent^0(\nu_2, \nu_2', \nu_3, \nu_3') - \thetaLaurent^0(\nu_1, \nu_1', \nu_3, \nu_3')  
		=  \delta(\R_+ \nu_1, \R_+ \nu_2, \R_+ \nu_3).
	\end{equation}

 	Let us denote the right-hand side of \eqref{property2} more simply by $\delta$. 
	We are going to deduce the equality of \eqref{property2} from \eqref{property1.5}, replacing the role of $\ell_i$ therein by $\nu_i$. 
	Splitting $\thetaLaurent^0(\nu_i, \nu_j) = A_{ij}+A_{ji}$, 
	where $A_{ij}$ has poles along $\nu_i=0$ and $A_{ji}$ has poles along $\nu_j=0$,
	the left-hand side of \eqref{property1.5}   is
	$$ 
		(A_{12} - A_{13}) + (A_{21}+A_{23}) + (-A_{31} + A_{32}),
	$$
	and since each of the three quantities in parentheses has a distinct polar locus, each
 	must be a constant $c_i$, where these constants values add up to $c_1+c_2+c_3 = \delta$.  {\em A fortiori}, the same is true after evaluating each parenthesized quantity thus:
	$$ 
		(A_{12} - A_{13})(\nu_1') + (A_{21}+A_{23})(\nu_2') +(-A_{31} + A_{32})(\nu_3') = c_1+c_2+c_3=\delta,
	$$
 	which proves \eqref{property2}.  
 
  Define $\phi$ on pairs $(\gamma_1, \gamma_2)$ with $\R \nu_1 \neq \R \nu_2$ by the formula in (ii). 
  The identity \eqref{property2} then expresses precisely that the coboundary computation \eqref{phicocycle} is valid when the $\nu_i$ 
  are non-proportional.
 It also uniquely specifies a way to extend this $\phi$
 to all pairs $(\gamma_1, \gamma_2)$: we just choose $\gamma_3$
 in generic position with respect to both of them and use \eqref{phicocycle} to define $\phi(\gamma_1, \gamma_2)$. 
 That this is independent
 of choice of $\gamma_3$ follows from the cocycle identity \eqref{deltacocycle} for $\delta$.  This proves the remainder of (i) and (ii).
 \end{proof}
 
\begin{remark} \label{DR}
The function $\phi(I_2, \gamma)$ is closely related to the Rademacher $\varphi$-function
(see \cite{KM} as a reference on the latter). 
We evaluate it in the generic case to illustrate this.

  Suppose that the transpose-inverse of $\gamma$ equals  $\smatrix{p & p'\\ q&  q' }$ with $q > 0$. 
 To compute 
$\phi(\mathrm{I}_2, \gamma)$, we must first of all compute $\thetaLaurent^0((-1,0),(-p,-q))$, recalling (ii) of Proposition \ref{phi}. 
By \eqref{tL1}, we must compute the sum 
\begin{equation} \label{rswt}  \sum_{(n_1, n_2) \in \Z^2 \cap C} e^{-n_1u_1 - n_2u_2},
\end{equation}
where $C$ is the cone spanned by the dual basis  $(-1,\frac{p}{q}), (0,-\frac{1}{q})$ to $(-1,0), (-p,-q)$.
 Now $$\Z^2 \cap C = \{ \alpha (1, -\tfrac{p}{q}) + \beta(0, \tfrac{1}{q})  \mid   \beta \equiv p \alpha \bmod q,\
 \alpha, \beta \in \Z_{\ge 0} \}.$$
Writing $\lambda_1=u_1-\frac{p}{q}u_2$ and $\lambda_2=\frac{1}{q}u_2$ we compute \eqref{rswt} as
 $$    \sum_{\substack{(\alpha, \beta) \in \Z_{\geq 0}^2\\\beta \equiv p \alpha \bmod q}} e^{\alpha \lambda_1 +\beta \lambda_2} = 
\sum_{(\alpha, \beta) \in \Z_{\geq 0}^2} q^{-1}  \left(\sum_{\zeta \in \mu_q} \zeta^{ \beta-p\alpha}\right)
 e^{\alpha \lambda_1 + \beta \lambda_2} = 
 q^{-1}  \sum_{\zeta \in \mu_q} \frac{1}{(1-\zeta^{-p} e^{\lambda_1})(1-\zeta e^{\lambda_2})}.$$
All terms above except the term for $\zeta=1$ are already regular at $(0,0)$;  
the $\zeta=1$ term contributes $\frac{1}{4} +  \frac{1}{12} (\lambda_1 \lambda_2^{-1} 
 + \lambda_2 \lambda_1^{-1})$ by the same computation as \eqref{tl3}, and thus
 the degree zero term equals 
$$  \frac{1}{q}  \left(  \frac{1}{4} +  \frac{1}{12} \left(\frac{\lambda_1}{\lambda_2} 
 + \frac{\lambda_2}{\lambda_1}\right) \right)  + \frac{1}{q} \sum_{\zeta \in \mu_q - \{1\}} \frac{1}{(1-\zeta)(1-\zeta^{-p})}
 $$
 The second term equals $\frac{1}{4} - \frac{1}{4q} + s(p,q)$, where $s(p,q)$ is the standard 
 Dedekind sum: see \cite[(18a) and (33a)]{RG}.
 Noting that
\begin{eqnarray*} \frac{\lambda_1}{q \lambda_2}\Big|_{(0,-1)} =- \frac{p}{q} &\mr{and}&
 \frac{\lambda_2}{q \lambda_1}\Big|_{(-p', -q')} = 
 \frac{q'/q}{qp'-pq'} =
-  \frac{q'}{q},
\end{eqnarray*} 
we get 
 $$ \phi(I, \gamma) =   \frac{1}{4} +   s(p,q)  - \frac{1}{12} \frac{p+q'}{q}    = 
  \frac{\varphi(\gamma)}{12} +\frac{1}{4}
 $$
 if $q > 0$ by \cite[Theorem 2.2]{KM}.
  \end{remark}
 
\subsection{Interpretation via equivariant motivic cohomology} \label{eqmotivic1}
Let us explain how the constructions of this section should be regarded as providing
a class in equivariant motivic cohomology, and 
  outline how one recovers our cocycle directly from this.
  Our construction is {\em ad hoc};  a suitable theory of equivariant motivic cohomology is not (to our knowledge) developed in the literature.

 To simplify our discussion,  we take coefficients in $\Z' := \Z[\frac{1}{6}]$;
 all cohomology groups should be understood with $\Z'$-coefficients. 
 Let $\mathsf{D}_{\Gamma}$ be the derived category of $\Z'[\Gamma]$-modules for $\Gamma=\GL_2(\Z)$. 
  Let $\varK^{\circ}$ be defined analogously to the complex $\varK$ 
 of \eqref{bigcomplex}, but taking $\Z'$-coefficients and replacing
 $\gm^2$ by $\gm^2-\{1\}$; we grade it cohomologically so that it becomes 
 supported in degrees $[-2,0]$.
This complex computes the motivic cohomology of $\gm^2-\{1\}$ with $\Z(2)$-coefficients in degrees $[2,4]$.
With our grading, the motivic cohomology in degree $4+i$ is the cohomology of $\varK^{\circ}$ in degree $i$ for $i \in [-2,0]$. 

As a provisional definition of a particular equivariant motivic cohomology group, we set
$$H^{3}_{\Gamma}(\mathbb{G}_m^2-\{1\},2) = \Hom_{\mathsf{D}_{\Gamma}}(\Z', \varK^{\circ}[-1]).$$
Now  $\varK^{\circ}$ does not compute the motivic cohomology of $\gm^2-\{1\}$
in full, only its truncation to degrees $2$ and greater.  In place of $\varK^{\circ}$, a proper definition of motivic cohomology
would employ a complex (e.g., of Bloch or Voevodsky) which computes the full motivic cohomology of $\gm^2-\{1\}$.\footnote{In fact, $\varK^{\circ}$ is the image of a morphism from a complex which does compute motivic cohomology, i.e., the total complex of a quasi-isomorphic truncation of the double complex underlying the coniveau spectral sequence.} 
However, since $\Gamma$ has no cohomology in degrees greater than $2$ upon inverting $6$,
the above would be isomorphic to a more reasonable definition of equivariant $H^3$. 

Let us produce a class in this $H^3_{\Gamma}$. 
Lemma \ref{nabla} and Proposition \ref{quasi-iso} together furnish a map 
$$
	h \colon \Chains_*(S^1) \rightarrow \varK^{\circ}[-1]
$$
in $\mathsf{D}_{\Gamma}$.
Since  $\Chains_*(S^1)$ has cohomology in degrees $-1,0$, the standard action of $\Gamma$ on $S^1$ induces
an exact triangle 
$$
 	\Z'(\det)[1] \rightarrow  \Chains_*(S^1) \rightarrow \Z'
$$
in $\mathsf{D}_{\Gamma}$, where $(\det)$ refers to twisting the action by $\det \colon \Gamma \rightarrow \langle - 1\rangle$.
Had we taken $\Z$-coefficients, the resulting extension class in $H^2(\Gamma,\Z(\det))$ would have been
the equivariant Euler class of $\R^2$ (i.e., the Euler class of the vector bundle on the classifying space
 $\mr{EGL}_2(\Z)/\GL_2(\Z)$ of $\GL_2(\Z)$ given by $(\mr{EGL}_2(\Z) \times \R^2)/\GL_2(\Z)$).
On the other hand, since $H^i(\Gamma,\Z'(\det)) = 0$ for $i \in \{1,2\}$,
 there is a unique splitting in $\mathsf{D}_{\Gamma}$:
 $$\Chains_*(S^1)  \cong \Z' \oplus \Z'(\det)[1]$$
 compatible with the above sequence. 
 In this way, $h$ splits into components  $h=h_{-1}+h_{-2}[1]$ with 
 $h_{-i} \in  \Hom_{\mathsf{D}_{\Gamma}}(\Z'(\det^{i-1}), \varK^{\circ}[-i])$ for $i \in \{1,2\}$.
This $h_{-1}$ gives a class in
  $ H^3_{\Gamma}(\mathbb{G}_m^2-\{1\},2)$,
and the class of $\Theta$ should be (we did not check details) recovered 
via 
\begin{equation} \label{les}  
 H^3_{\Gamma}(\mathbb{G}_m^2-\{1\}, 2)  \xrightarrow{\mathrm{restrict}} 
 H^3_{\Gamma}( \Q(\mathbb{G}_m^2), 2) \xrightarrow{\mr{s.s.}} H^1(\Gamma, K_2  \Q(\mathbb{G}_m^2)),\end{equation}
 where the last map comes out of a spectral sequence 
 $H^i(\Gamma,H^j(\Q(\gm^2),2)) \Rightarrow H^{i+j}_{\Gamma}(\Q(\gm^2),2)$
 computing equivariant cohomology in terms of $\Gamma$-cohomology on motivic cohomology. 
  
\section{The square of a universal elliptic curve} \label{E2} 
 In the present section we construct the
big cocycles ${}_n \Theta$ of \eqref{bigcocycE2} for primes $n \nmid N$.
Here, the role of $\Gm^2$ is replaced by the self-product $\mc{E}^2$ of the universal
elliptic curve over a modular curve.
As in the $\Gm$-case, our analysis is based on a homological complex
$\varK$ in degrees $[2,0]$
that computes the motivic cohomology $H^{4-i}(\mc{E}^2, 2)$. 
 Two key differences are:
 
\begin{itemize}
\item The motivic cohomology of $\mc{E}^2$ is more complicated than that of $\gm^2$. However, 
through the theory of Fourier-Mukai transform, one can obtain
a reasonable understanding of various isotypical pieces under trace maps. We employ work of Deninger and Murre, taking care with the coefficients. 

\item The complex $\varK$ is not exact in degree zero (even after taking fixed parts).   This has the following consequence. In the 
$\Gm$-case, we made use of an element $e \in \varK_0$
that is the class of the identity of $\Gm^2$. The analogue here arises from the identity section of $\mc{E}^2$, but this
is no longer a boundary from $\varK_1$. The element $e_n \in \varK_0$ that we use is supported on $n$-torsion
for an auxiliary integer $n$: see \eqref{endef}. The symbols we work with have correspondingly more involved definitions but do satisfy the same relations as before.
 \end{itemize}

The contents of the various subsections are as follows.
In \S \ref{FM}, we give an integral refinement of a result of Deninger-Murre \cite{dm} on the decomposition of the motivic cohomology of an abelian variety into isotypical components for the action of trace maps, with a particular view towards the fixed parts that we employ.  In \S \ref{abstract}, we give an abstract construction of a cocycle $\Theta^Z$ as in \eqref{abstractcocycG2} attached to a trace-fixed, $\GL_2(\Z)$-invariant, degree zero formal sum $Z$ of points.  In \S \ref{symbolsE2}, we define our explicit symbols in the terms of the big complex and show that they are
trace-fixed.  \S \ref{expcocycE2} contains the construction of the cocycles ${}_n\Theta$.  In \S \ref{HeckeE}, we consider the compatibility of ${}_n \Theta$ with two types of prime-to-level Hecke operators, those acting on $\GL_2(\Z)$-cocycles and those arising from as correspondences on motivic cohomology.  In Theorem \ref{heckeE2}, we prove that  the two resulting actions agree on the class of ${}_n \Theta$.

\subsection{Fixed parts via the Fourier-Mukai transform} \label{FM}

Let $Y$ be a smooth, separated, connected scheme of finite type over a field $F$ of characteristic $0$,
and let $A$ be a family of abelian varieties of relative dimension $g$ over $Y$.  
Set $d = \dim Y$. 

Let $\Z' = \Z[\frac{1}{(2g+1)!}]$.
For any integer $i$, we set 
$$
	H^i(A,\Z'(g)) = H^i(A,g) \otimes_{\Z} \Z'.
$$ 
As in Section \ref{tracemaps}, there are trace maps $[m]_*$ on $H^i(A,\Z'(g))$.  There are also pullback maps $[m]^*$,
and since multiplication by $m$ has degree $m^{2g}$ on $A$, we have the relation 
\begin{equation} \label{tracepullback}
	[m]_* [m]^* =  m^{2g}.
\end{equation}

We next prove that $H^i(A,\Z'(g))$ is the sum of its isotypic components for pullback maps. (We are eventually interested 
in trace maps $[m]_*$ as these can also be defined for open subschemes, but we will deduce such results from those on pullbacks.)
 The argument  follows \cite[Theorem 2.19]{dm}, with an appeal to the integral Grothendieck-Riemann-Roch of Pappas \cite{pappas} to allow us to work over $\Z'$.  

\begin{theorem} \label{fixedgp}	
	For $i \in \Z$, each class $\alpha \in H^i(A,\Z'(g))$ is the sum of components $\alpha = \sum_{s=0}^{2g} \alpha_s$
	where $[m]^* \alpha_s = m^{2g-s} \alpha_s$ for all $m \in \mathbb{N}$.   
\end{theorem}
  
\begin{proof} 
        Let $A^{\vee}$ be the dual abelian scheme 
        that represents $\mr{Pic}_{A/Y}^0$.  
        Let $\mathrm{P}$ be the Poincar{\'e} bundle on $A \times_{Y} A^{\vee}$.
        We may form the Chern character in motivic cohomology: 
        $$  
        		e^{c_1(\mathrm{P})} := \sum_{r=0}^{2g+d}   \frac{c_1(\mathrm{P})^r}{r!}    \in 
        	\bigoplus_{r=0}^{2g+d} H^{2r}(A \times_{Y} A^{\vee}, \Z[\tfrac{1}{r!}](r)).
        $$
        Note that the  powers of $c_1(\mathrm{P})$ beyond $2g+d$ vanish, 
        because they lie in a Chow group that is evidently zero. 
 	Since $(1 \times [m]^*) \mathrm{P} \cong \mathrm{P}^{\otimes m}$, we have
        $$
        		(1 \times [m]^*) c_1(\mathrm{P})^r = m^r c_1(\mathrm{P})^r.
        $$
        
        The diagram 
       \begin{equation} \label{FMdiag}
        	A  \xleftarrow{\pi_1} A\times_{Y} A^{\vee} \xrightarrow{\pi_2} A^{\vee}
        \end{equation}
        provides a morphism defined on $\alpha \in H^i(A,\Q(g))$ by 
        $$
        		\mathcal{F}(\alpha) = (\pi_2)_* (\pi_1^* \alpha \cup e^{c_1(\mathrm{P})}),
	$$
	known as the Fourier-Mukai transform.
	By its definition, $\mathcal{F}$ breaks up as a sum of operators,
	with the $r$th component $\mathcal{F}_r$ corresponding to $\frac{c_1(\mathrm{P})^r}{r!}$.  That is, 
	$\mathcal{F} = \sum_{r=0}^{2g+d} \mc{F}_r$, where, paying attention to denominators, we have
	\begin{equation} \label{FT}
		\mc{F}_r \colon H^i(A,\Z(g)) \to H^{i+2(r-g)}(A^{\vee},\Z[\tfrac{1}{r!}](r)).
        \end{equation}
       	Any element of the image of $\mc{F}_r$ transforms under the image of each $[m]^*$ by $m^r$.
	
	Now, let $\mc{F}^{\vee}$ be defined dually, with the dual abelian variety $A^{\vee}$ in place of $A$. 
	For motivic cohomology with $\Q$-coefficients, Deninger and Murre show in \cite[Corollary 2.22]{dm} that 
	\begin{equation} \label{FT0} \mc{F}^{\vee} \circ \mc{F}=(-1)^g[-1]^* \end{equation} 
	(for usual Chow groups, but the argument applies equally to higher Chow groups).\footnote{It may be 
	helpful to note that \eqref{FT0} is not a formality  as it is, for example, at the level of coherent sheaves. 
	After applying Grothendieck-Riemann-Roch to the coherent sheaf equality, one needs to verify certain
	equalities of Todd classes;  these classes vanish with rational cohomology over $\C$
	because they arise from flat bundles, but these arguments do not apply in the current setting. Rather,
	Deninger and Murre first show this   \cite[Proposition 2.13]{dm} on $H^i(A,\Q(g))$ 
	up to terms in cohomological degree greater than $i$.  That equality holds on the nose follows from the less precise statement
         without any appeal to the coefficients used.    }
         
         The use of rational coefficients in \cite{dm} is mandated not just by the denominators in the Chern character, but
         by two applications of the Grothendieck-Riemann-Roch theorem (GRR), both of which arise from \cite[Lemma 2.8]{dm}
         and are used in the subsequent proposition. 
         The first and most consequential application is for the projection morphism 
         $A^{\vee} \times_S A \to A$, and the second is for the identity section $e_A \colon S \to A$. 
         
        	 Recall that GRR concerns the behavior of the cup product $\mr{CT}(\mc{G}) = e^{c_1(\mc{G})} \cup \mr{td}(T_X)$ 
         of the Chern character of a coherent sheaf $\mc{G}$ on a smooth quasi-projective variety $X$ over $F$ and
         the Todd class of the tangent bundle $T_X$ of $X$ under
         pushforward by a projective morphism $f \colon X \to Y$, where $Y$ is another such variety.  For $e$ the relative
         dimension, it says more precisely regarding the degree $2r$ component that 
         $$
         	\mr{CT}_r(f_*(\mc{G})) = f_*\mr{CT}_{r+e}(\mc{G}) \in H^{2r}(Y,\Q(r)).
         $$
      	 By \cite[Theorem 2.2]{pappas}, 
         GRR remains true integrally if $F$ has characteristic $0$ upon inverting the primes dividing $(e+r+1)!$ if $e \ge 0$
         and $(r+1)!$ if $e < 0$.
         In the two cases of interest to us, we are concerned with $\mr{CT}_r$ for $r \le g$, and $e = g$ and $e = -g$, respectively.
         In particular, both applications of GRR go through with coefficients in $\Z'=\Z[\frac{1}{(2g+1)!}]$.
	Consequently, \eqref{FT0} remains valid with coefficients in $\Z'$, and we will use it, as such, in what follows. 
	
	Now let us return to the equality \eqref{FT0}, which we examine when restricted to $H^i(A, \Z'(g))$. 
	We may write $\mc{F}^{\vee} \circ \mc{F}  = \sum_{s= 0}^{2g+d} \sum_{t=0}^{2g+d} \mc{F}^{\vee}_{t} \circ \mc{F}_s$,
	and  by definition the composition $\mc{F}^{\vee}_{t} \circ \mc{F}_s$, restricted to 
	$H^i(A, \Z'(g))$, has image in $H^{i+ 2(t-g) + 2(s-g)}(A,\Z'(s+t-g))$. Therefore \eqref{FT0}
	implies that each such term with $t+s \neq 2g$ must vanish. 
	In other words, 
	$$
		\mc{F}^{\vee} \circ \mc{F}|_{H^i(A, \Z'(g))} = \sum_{r=0}^{2g} \mc{F}^{\vee}_{2g-s} \circ \mc{F}_s = (-1)^g [-1]^* |_{H^i(A, \Z'(g))},
	$$
	where we now understand all the operators to act on motivic cohomology with $\Z'$-coefficients.
 
       	 Take $\alpha \in H^i(A,\Z'(g))$, and apply $\mathcal{F}^{\vee} \circ \mathcal{F}$ to $\alpha' := (-1)^g [-1]^* \alpha$.
	 Now \eqref{FT0} implies that the result is $\alpha$,
	 so writing $\alpha_s =  \mathcal{F}_{2g-s}^{\vee} \mathcal{F}_s \alpha' \in H^i(A,\Z'(g))$,
	 we have
	       	\begin{equation} \label{alphaarg}
		\alpha = \sum_{s=0}^{2g} \alpha_s,
       	\end{equation}
	Then $[m]^*$ acts as $m^{2g-s}$ on $\alpha_s$ for all $m \ge 1$, as required. 
\end{proof}	

We can now compute the fixed parts of motivic cohomology groups under trace maps for integers relatively prime to a fixed positive
integer $n$.  Let $\mb{N}_n$ denote the monoid of positive integers prime to $n$.  We consider the groups $H^i(A,\Z'(g))$ as
$\Z'[\mb{N}_n]$-modules for the trace maps.  For $s \ge 0$, set
$$
	H^i(A,g)^{(s)}  = \{ \xi \in H^i(A,\Z'(g)) \mid ([m]_*-m^s)\xi = 0 \text{ for all } m \in \mb{N}_n \}.
$$
        
\begin{proposition} \label{fixedpartsFM}  
	We have a direct sum decomposition
	$$
		H^i(A,\Z'(g)) = \bigoplus_{s = 0}^{2g} H^i(A,g)^{(s)},
	$$
	of $\Z'[\mb{N}_n]$-modules, which is natural in $A$ over $Y$. The group $H^i(A,g)^{(0)}$ is zero unless $i = 2g$,
	and $H^{2g}(A,g)^{(0)}$ is naturally isomorphic to $\Z'$ as a $\Z'[\mb{N}_n]$-module.
\end{proposition}       
         
 \begin{proof}       
 	Write $\alpha \in H^i(A,\Z'(g))$ as $\alpha = \sum_{s=0}^{2g} \alpha_s$ as in Theorem \ref{fixedgp}. 
        	It follows from \eqref{tracepullback} that, for any prime $\ell \nmid n$, we have $(\ell^{2g-s}[\ell]_*- \ell^{2g})\alpha_s = 0$ for each $s$.           	For each $0 \le t \le 2g$, let
        \begin{equation} \label{projectop}
        		\phi_t(\ell) =\prod_{\substack{s=0\\ s\neq t}}^{2g}   \ell^{2g} ([\ell]_*-\ell^s)
	\end{equation}
        	so that $\phi_t(\ell)\alpha = \phi_t(\ell)\alpha_t$.  Then $\phi_t(\ell)$ acts on $\alpha_t$
	by the scalar 
   	$$
        		r_t(\ell) =\prod_{\substack{s=0\\ s \neq t}}^{2g}  \ell^{2g} (\ell^t-\ell^s).
	$$
	There exist primes $\ell_1, \ldots, \ell_h$  not dividing $n$ and
	$c_1, \ldots, c_h \in \Z'$ such that $\sum_{j=1}^h c_j r_t(\ell_j) = 1$.
	(If $p > 2g+1$, then
	$p$ doesn't divide $r_t(\ell)$ whenever $\ell$ is a primitive root modulo $p$.) 
 	The element
	\begin{equation} \label{phitdef}
		\phi_t = \sum_{j=1}^h c_j \phi_t(\ell_j) \in \Z'[\mb{N}_n]
	\end{equation}
	then satisfies $\phi_t(\alpha) =  \alpha_t$. 
	This element $\phi_t$ defines a projection of $H^i(A,\Z'(g))$ onto $H^i(A,g)^{(t)}$;	 so $H^i(A,g)^{(t)}$ is a $\Z'[\mb{N}_n]$-module direct summand of $H^i(A,\Z'(g))$.  
	
	Suppose now that $\alpha \in H^i(A,g)^{(0)}$.  Then $\alpha = \phi_0(\alpha) = \alpha_0$.
	Referring to \eqref{alphaarg}, this $\alpha$ is necessarily of the form
	$\mathcal{F}^{\vee}_{2g} \beta_0$, where 
	 $\beta_0 \in H^{i-2g}(A^{\vee},\Z')$.  Since 
        $H^{i-2g}(A^{\vee},\Z') = 0$ for $i \neq 2g$, we actually have
        $\alpha = 0$ unless $i = 2g$.  For $i = 2g$, we have the canonical identification 
        $H^0(A^{\vee},\Z') \cong \Z'$, and
        $\mc{F}^{\vee}_{2g}$  carries $1 \in \Z'$ 
        (up to sign) to the fundamental class of the zero section in $H^{2g}(A, \Z'(g))$.  
        To verify the final statement, we have by \eqref{FT0} that  
       $\mathcal{F}_0$ carries the fundamental
        class of the zero section to a generator for $H^0(A^{\vee}, \Z')$. In summary, $H^{2g}(A, g)^{(0)}$
        is a free $\Z'$-module of rank one, generated by the fundamental class of the zero section. 
\end{proof}
  
  
\subsection{The abstract cocycles} \label{abstract}

Let us now specialize to the case of an elliptic curve $E$ over a smooth separated, connected scheme $S$ of finite type over a field $F$.
We will apply Theorem \ref{fixedgp} to $A = E^2$ over $S$.   For the remainder of this section, we set
$$
	\Z' = \Z[\tfrac{1}{30}].
$$

Just as in the case of $\mathbb{G}_m$, we write down a complex computing the cohomology of $E^2$.
Recall from Example \ref{n=2} that the complex $\varK$
given in homological degrees $2$ to $0$ by
\begin{equation} \label{basiccomplex} K_2 k(E^2)  \rightarrow   \bigoplus_{D} K_1 k(D)  \rightarrow 
\bigoplus_{x} K_0 k(x), \end{equation}
the sums being taken over irreducible divisors and codimension $2$ points of $E^2$ respectively, 
computes (from left to right) the cohomology $H^*(E^2, 2)$ in degrees $2$ to $4$. 

Unlike the case of $\gm^2$, none of these cohomology groups of $E^2$ need vanish.
However, by Remark \ref{tracemapsrmk}, the complex admits trace maps $[m]_*$.   Set 
 $$ \varK^{(0)} = \{ \alpha \in \varK \otimes_{\Z} \Z' \mid ([p]_*-1)\alpha = 0 \text{ for all but finitely many primes $p$} \}.$$
Then $\varK^{(0)}$ can be regarded as the direct limit   $  \varinjlim_{n} ({}_n \varK^{(0)})$,
over integers $n$ 
ordered by divisibility,  of complexes 
  ${}_n \varK^{(0)}$ defined by the fixed part of $\varK \otimes_{\Z} \Z'$ under all $m \in \mb{N}_n$.

\begin{lemma} \label{bigcplxleftex}
	The sequence $0 \to \varK_2^{(0)} \to \varK_1^{(0)} \to \varK_0^{(0)}$ is exact. 
\end{lemma}

\begin{proof}
	It is enough to prove the same assertion for ${}_n \varK^{(0)}$, since
	the claim then follows by taking the direct limit.  In the following discussion, ``fixed parts'', or a superscript ``$(0)$'',
	refers to being fixed under $\mathbb{N}_n$. 
 	
	Consider the exact sequence  
	$$
		0 \to H^2(E^2,\Z'(2)) \to \varK_2 \xrightarrow{\partial_2} \ker(\varK_1 \to \varK_0) \to H^3(E^2,\Z'(2)) \to 0
	$$
	of $\Z'[\mb{N}_n]$-modules.  
	The map on fixed parts induced by $\partial_2$
	is injective as $H^2(E^2,2)^{(0)}$ is trivial by Proposition \ref{fixedpartsFM}.
	
	If $y \in   {}_n\varK_1^{(0)}$ has trivial residue (i.e., dies in $\varK_0$), then it maps to $H^3(E^2,2)^{(0)}$, which equals $0$
	by Proposition \ref{fixedpartsFM}. 
	Thus, there exists $x \in \varK_2$ with $\partial_2(x) = y$.  For any $m \in \mb{N}_n$, since
	$[m]_*-1$ annihilates $y$, the element $([m]_*-1) x$ lies in the kernel of $\partial_2$. 
	So, there in turn exists $z_m \in H^2(E^2,\Z'(2))$ that maps to $([m]_*-1) x$. 
	Equation \eqref{phitdef} provides an element $\phi_0 \in \Z'[\mb{N}_n]$ that projects any $H^i(E^2,\Z'(2))$
	onto its fixed subspace. For $i = 2$, the fixed part is   trivial, so   
	$$
		([m]_*-1)\phi_0x = \phi_0([m]_*-1)x = \phi_0 z_m = 0.
	$$
	In other words, we have $\phi_0 x \in {}_n \varK_2^{(0)}$.   
	Moreover, $\phi_0$ fixes any 
	element of ${}_n \varK_1^{(0)}$, so $\phi_0 y=y$. Since
	$$
		\partial_2(\phi_0 x) = \phi_0 y 
		= y,
	$$
	the sequence is exact at ${}_n \varK_1^{(0)}$.
\end{proof}


Now there is a  surjective
 \emph{degree map} 
$$
	\deg \colon \varK_0^{(0)} \to \Z',
$$
obtained by composing $\varK_0^{(0)} \rightarrow H^4(E^2, 0)^{(0)}$
with the isomorphism of the latter group with $\Z'$ furnished by
Proposition \ref{FM}.

\begin{proposition} \label{abstractcocyc}
	Let $Z \in \varK_0^{(0)}$ be a $\GL_2(\Z)$-fixed class with $\deg(Z) = 0$ 
	such that there exists\footnote{We must assume the existence of $\eta$ because 
	it	
	is not clear that the resulting sequence $0 \to \varK_2^{(0)} \to \varK_1^{(0)} \to \varK_0^{(0)} \to \Z' \to 0$ should be exact at
	$\varK_0^{(0)}$.}
	 $\eta \in \varK_1^{(0)}$ with $\partial \eta = Z$.  Then
	there is a $1$-cocycle
	$$ 
		\Theta^Z \colon \GL_2(\Z) \longrightarrow  \varK_2^{(0)}, \quad \gamma \mapsto \Theta^Z_{\gamma},
	$$
	where $\Theta^Z_{\gamma}$ is uniquely characterized by the
	property that
  	$$
		\partial\Theta^Z_{\gamma} = (\gamma^*-1) \eta.
	$$
	Moreover, the class of $\Theta^Z$ is independent of the choice of $\eta$.
\end{proposition}

\begin{proof}
	By Lemma \ref{bigcplxleftex},
	a unique $\Theta^Z_{\gamma} \in \varK_2^{(0)}$ 
	with residue $(\gamma^*-1)\eta$ exists.  That the resulting function $\Theta^Z$ is a cocycle follows 
	just as in Proposition \ref{existence}.
	
	If $\eta' \in \varK_1^{(0)}$ also satisfies $\partial\eta' = Z$, then $\eta'$ gives rise to another cocycle $\Theta'$.
	By the left exactness in Lemma \ref{bigcplxleftex}, there exists $\psi \in \varK_2^{(0)}$ with $\partial\psi = \eta-\eta'$.  The cocycles
	 $\Theta^Z$ and $\Theta'$ are cohomologous since $\Theta^Z_{\gamma} - \Theta'_{\gamma} = (\gamma^*-1)\psi$.
\end{proof}

Note that if we can also choose $\eta$ to be fixed by a parabolic subgroup
of $\GL_2(\Z)$, then the argument of Proposition \ref{parabolic} implies that $\Theta^Z$ is parabolic (with the same meaning as in that proposition). 

In the remaining sections, we specialize to the case that $E$ is the universal elliptic curve over a modular curve $Y_1(N)$.
In this setting, we will proceed more computationally and produce not only a particularly nice choices of $Z$
(supported on torsion) 
but also nice choices of $\eta$ 
entirely parallel to the $\gm$-case. 
To do this, we first of all set up a class of natural symbols in $\varK$ with which we can compute.

\subsection{Symbols}  \label{symbolsE2}

Fix an integer $N \geq 4$. 
We will work over the base scheme $Y := Y_1(N)$ over $\Q$
whose $S$-points for a $\Q$-scheme $S$ parameterize pairs $(E, P)$
of an elliptic curve $E/S$ and a section $P$ of $E[N]$ 
that is everywhere of exact order $N$ (i.e., the associated map from $\Z/N\Z$ to $E[N]$ is a closed immersion 
of group schemes over $S$).
Though we often omit $N$ from the notation, it should be understood throughout the
remainder of this section that we are working at level $\Gamma_1(N)$. 

Our elliptic curve will be taken to be the universal elliptic curve $\mc{E}$ over $Y$.   Let $\pi \colon \mc{E} \to Y$ be the structure morphism. 
We shall write, for short, 
$$ \mc{E}^2 =\mc{E} \times_Y \mc{E}$$
for the square of the universal elliptic curve over $Y$.  We let $\pi_i \colon \mc{E}^2 \to \mc{E}$ for $i \in \{1,2\}$ denote the $i$th projection map.

It will often be useful to add auxiliary $\Gamma_0(m)$-structure to $Y$. 
For a positive integer $m$ prime to $N$, let $Y_m$ denote the modular curve over $\Q$ 
corresponding to level structure $\Gamma_1(N) \cap \Gamma_0(m)$.
For a $\Q$-scheme $S$, the points of $Y_m(S)$ are equivalence classes of triples $(E,P,K)$, where
$(E, P) \in Y(S)$ and $K$ is an \'etale-locally cyclic $S$-subgroup scheme of order $m$.

\subsubsection{Symbols on $\mc{E}$} \label{symboldef}

Fix a prime number $n \nmid N$.  Denote by $\mc{E}'$ the pullback of $\mc{E}$ from $Y$ to  the modular curve $Y' = Y_n$.
 The curve $\mc{E}'$ is equipped with a canonical cyclic subgroup scheme $\mc{K}$ of order $n$. 

We first define some auxiliary divisors and rational functions on $\mc{E}'$ and $\mc{E}$ that we use to construct our symbols in the
big complex $\varK$.  Note that any $S$-subgroup scheme $\mathsf{G} \subset E[n]$
  of the $n$-torsion of an  elliptic curve $E$ over a  base variety $S$  defines a class in $H^0(E[n],0)^{(0)}$. 
  Namely, $\mathsf{G}$ is a union of connected components of $E[n]$, 
and we associate to  $\mathsf{G}$   the sum of these components
considered in $H^0(E[n], 0)$; this is automatically $[m]_*$-fixed for $m$ relatively prime to $n$.
\begin{itemize}
\item Set 
\begin{align} \label{deltadef}
	\newf = {}_n \newf = n^2(0) - \mc{E}[n] \in H^0(\mc{E}[n],0)^{(0)} &&\mr{and}&&
	\newf'  = {}_n \newf' = n \mc{K} - \mc{E}'[n] \in H^0(\mc{E}'[n],0)^{(0)}.
\end{align}
\end{itemize}
 
Note that we have an exact sequence
$$
	0 \to H^1(\mc{E},1) \to H^1(\mc{E}-\mc{E}[n],1) \to H^0(\mc{E}[n],0) \to H^2(\mc{E},1) \to 0.
$$
Since 
$H^2(\mc{E},1)^{(0)} \cong \Z'$ by Theorem \ref{fixedgp},
and any element of $H^1(\mc{E},1)$ is necessarily an invertible local constant,  a variant of an argument of Kato
\cite[1.10]{kato}\footnote{Kato works with $\Z$-coefficients but avoids $n \in \{2,3\}$.  Uniqueness of a trace fixed element with
a given trace fixed degree zero divisor follows from $H^1(\mc{E},1)^{(0)} = 0$.
Existence follows from the stronger statement that $H^1(\mc{E},1) = H^1(\mc{E},1)^{(2)}$ (see the proof of Lemma \ref{actiononunits}), 
the commutativity of trace maps, and the fact that the greatest common divisor of all $\ell^2-1$ for $\ell$ prime to $n$ divides $24 \in (\Z')^{\times}$.}
yields an exact sequence
\begin{equation} \label{thetaseq}
	0 \to H^1(\mc{E}-\mc{E}[n],1)^{(0)} \xrightarrow{\mathrm{div}} H^0(\mc{E}[n],0)^{(0)} \xrightarrow{\deg} \Z' \to 0,
\end{equation}
where the degree map  is surjective since the class defined by the zero-section has degree $1 \in \Z'$.  
We also have the analogous sequence for $\mc{E}'$.
Since $\newf$ and $\newf'$ have degree zero, we may make the following definition.

\begin{itemize}
\item 
Let 
\begin{eqnarray*}
	\theta = {}_n \theta \in H^1(\mc{E} - \mc{E}[n],1)^{(0)} &\mr{and}& \theta' = {}_n \theta' \in H^1(\mc{E}' - \mc{E}'[n],1)^{(0)}
\end{eqnarray*}
be the unique ``theta functions'' with divisors given by $\delta$ and $\delta'$:
\begin{eqnarray} \label{thetasdef}  
	\partial \theta = \delta &\mr{and}& \partial \theta'=\delta'.	
\end{eqnarray}
\end{itemize}

The morphisms $Y' \rightarrow Y$ and $\mc{E}' \rightarrow \mc{E}$, as well as $(\mc{E}')^2 \rightarrow \mc{E}^2$, 
where we write
$$ 
	(\mc{E}')^2 = \mc{E}^2 \times_{Y} Y',
$$
are finite {\'e}tale of degree $n+1$.\footnote{This is the first point where we use that $n$ is prime. Though it should be possible to extend our constructions below to general $n$, from our point of view it would unnecessarily complicate the discussion.}
Let us denote the norm (i.e., pushforward) maps on motivic cohomology induced by these morphisms by $\Norm$.   Not only do these norms act only on the motivic cohomology of $\mc{E}'$,
but by Lemma \ref{restra}, they also give give a map of complexes $\varK' \rightarrow \varK$,  with $\varK'$ being the analogue of $\varK$ for $(\mc{E}')^2$.
  For example, in the direct sum of zero $K$-groups of residue fields of codimension $1$ points on $\mc{E}$, we have $\Norm(\mc{K}) = \mc{E}[n] + n (0)$ 
and $\Norm(\mc{E}'[n]) =  (n+1)\mc{E}[n]$. (For the first, for example, the norm of $\mc{K}$ gives on each elliptic curve fiber
of $\mc{E} \rightarrow Y$ the sum of {\em all} cyclic subgroups of order $n$, which counts the origin with multiplicity $n+1$ and all other points with multiplicity $1$.)
Thus, we have
$$ \Norm(\newf') = n\Norm(\mc{K}) - (n+1)\mc{E}[n] = n^2(0) - \mc{E}[n] = \newf.$$
 
\subsubsection{Symbols on $\mc{E}^2$} \label{symbolsE2subsec}

We continue to fix an auxiliary prime $n \nmid N$.
We are going to define 
symbols $\langle a,c \rangle_n \in \varK_1^{(0)}$ for primitive pairs $(a,c) \in \Z^2-\{0\}$ and $\langle \gamma \rangle_n \in \varK_2^{(0)}$ for $\gamma \in \GL_2(\Z)$ satisfying relations identical to \eqref{boundary}, but now with the degree zero element $e$ in the $\gm^2$-setting replaced by a special $\GL_2(\Z)$-fixed and trace-fixed cycle $e_n$ that depends on our choice of $n$. The symbols, which also depend on $n$, allow us to give an explicit description of the
abstract cocycle $\Theta^Z$ of Proposition \ref{abstractcocyc} in the case that $Z=e_n$. 

As before, an element of $\Delta = M_2(\Z) \cap \GL_2(\Q)$ provides a morphism
$\mc{E}^2 \rightarrow \mc{E}^2$ 
over $Y$ via right multiplication.   
We denote by $T_n^{\varK}$ the operator on $\varK$ given by the sum of pullbacks by the representatives $g_j$ of \eqref{Gjdef} (replacing $\ell$ by $n$).
While the operator $T_n^{\varK}$ depends on the choice of these coset representatives, its action on $\GL_2(\Z)$-invariant elements does not.  

 \begin{itemize}
	\item 
	In $\varK_0$, we form the element
	\begin{equation} \label{endef}
		 \newE_n = n \left( n^3(0) - nT_n^{\varK}(0) + \mc{E}[n]^2 \right).
	\end{equation}
	Here, we view $H^0(\mc{E}[n]^2,0)$ as a subgroup of $\varK_0$ by the map taking a formal sum of irreducible cycles in $\mc{E}[n]^2$ to the corresponding element of the direct sum of copies of $\Z$ given by the zeroth $K$-groups of those cycles.
The element $\newE_n$ is $\GL_2(\Z)$-fixed as a sum of fixed terms.\footnote{To see this for $T_n^{\varK}(0)$, recall that left multiplication of $g_j$ by an element of $\GL_2(\Z)$ is right multiplication of some $g_{j'}$ by an element of $\GL_2(\Z)$, and $(0)$ is $\GL_2(\Z)$-fixed.}  Note that 
$\newE_n = V_n^{\varK}(0)$, where
\begin{equation} \label{Vn0}
	V_n^{\varK} = n^4 - n^2T_n^{\varK} + n[n]^*.
\end{equation}
The element $e_n$ has degree zero as $T_n^{\varK}$ has degree $n(n+1)$ and $[n]^*$ has degree $n^2$. 
We will explain the significance of this particular choice of $\newE_n$ in Remark \ref{enremark}.

 	\item In $\varK_1$, we form
	\begin{equation} \label{10def}
		\langle 1, 0 \rangle_n =  \newf \boxtimes \theta - \Norm (\newf' \boxtimes \theta').
	\end{equation}
	The external product $\newf \boxtimes \theta$ here should be understood to mean the restriction of the function  
	$\pi_2^* \theta$ on $\mc{E} \times_Y (\mc{E}-\mc{E}[n])$ to the divisor defined by $\pi_1^{-1}(\newf)$.  This defines 
	a class in the direct sum of the multiplicative groups of function fields of the irreducible divisors composing 
	$\mc{E}[n] \times_Y \mc{E}$, so also an element of $\varK_1$.  
	Similarly, $\newf' \boxtimes \theta' \in \varK'_1$ is the external product with respect to $(\mc{E}')^2 = \mc{E}' \times_{Y'} \mc{E}'$.  	 

	More generally, we set
	$$
		\langle a,c \rangle_n = \gamma^*\langle 1,0 \rangle_n  \in \varK_1.
	$$
	where	$\gamma = \smatrix{a & b \\ c & d} \in \SL_2(\Z)$ is arbitrary with first column $(a,c)$.

	\item  In $\varK_2$, we form  
	\begin{equation} \label{1001def}
		\langle \smatrix{1&0\\0&1} \rangle_n = \theta \boxtimes \theta - \Norm (\theta' \boxtimes \theta').
	\end{equation}
 	 Here, $\theta \boxtimes \theta$ denotes the Steinberg symbol $\{\pi_1^* \theta, \pi_2^* \theta\}$, and $\theta' \boxtimes
	 \theta' \in \varK'_2$ is defined analogously. 
	 
	 In general, for $\gamma = \smatrix{a & b \\ c & d} \in \GL_2(\Z)$, we set
	 \begin{equation} \label{gammalrndef} 
	 	\langle \gamma \rangle_n  = \gamma^* \langle \smatrix{1&0\\0&1} \rangle_n \in \varK_2.
	 \end{equation} 
 \end{itemize}
 
In Lemma \ref{actiononunits}, we will show that $\langle a,c \rangle_n$ is independent of the second column of $\gamma$ used in
 defining it. For now, let us fix such a choice and show that our symbols are $[m]_*$-fixed for all $m \in \mb{N}_n$. 
 
\begin{lemma} \label{Symbols are Fixed}
	The symbols $e_n$, $\langle a,c \rangle_n$, and $\langle \gamma \rangle_n$ defined above satisfy
 	$$
		e_n \in \varK_0^{(0)},\  \langle a,c \rangle_n \in \varK_1^{(0)}, \text{ and } \langle \gamma \rangle_n\in \varK_2^{(0)}.
	$$
\end{lemma}

\begin{proof} 	Let $m \in \mb{N}_n$.
	First, note that $\delta$, $\theta$, $\delta'$, $\theta'$ are $[m]_*$-fixed, and therefore their exterior products are as well
	(see Example \ref{product example}).  
	Since $\mc{E}' \to \mc{E}$ commutes with the multiplication-by-$m$ map $[m]$, 
	the norm maps $\mathrm{N}$ in \eqref{10def} and \eqref{1001def} commute with $[m]_*$.  
	The Hecke operator $T_n^{\varK}$ in \eqref{endef} commutes with $[m]_*$ in that each
 	$$
 		\begin{tikzcd}
  		\mc{E}^2 \arrow{d}{[m]} \arrow{r}{g_j} & \mc{E}^2 \arrow{d}{[m]} \\
  		\mc{E}^2 \arrow{r}{g_j} & \mc{E}^2 
  		\end{tikzcd}
  	$$
 	for $0 \le j \le n$ is a Cartesian square.
	Also, $\mc{E}[n]^2$ and each $g_j^*(0)$ is $[m]_*$-fixed since 
	$[m]$ is an automorphism of the corresponding subgroup schemes.  It follows that all of the symbols are $[m]_*$-fixed.
\end{proof} 

 \begin{remark} \label{enremark}
 Let us explain where the strange definition of $ \newE_n$  in \eqref{endef} comes from. 
The main issue at hand is that one cannot find a function on $\mc{E}$ with a single pole at the origin,
and therefore (if one is to produce explicit formulas) one needs to choose a $\GL_2(\Z)$-fixed element of
$\varK_0^{(0)}$ somewhat carefully. We will sketch  the important
feature that this particular formula has. 

Take a geometric point $s$ of $Y$ with associated elliptic curve $E = \mc{E}_s$,
and fix a basis  for $E[n]$, i.e., an isomorphism of abelian groups $E[n] \simeq (\Z/n\Z)^2$.
This then identifies  $E[n] \times E[n]$ with $(\Z/n\Z)^2 \times (\Z/n\Z)^2$;
regarding the two copies of $(\Z/n\Z)^2$ as the top and bottom rows of a $2 \times 2$ matrix,
we may thus regard $E[n] \times E[n] \simeq M_2(\Z/n\Z)$. 

Using these coordinates,   the fiber of $e$ above $s$ 
 is the formal sum $\sum_{M \in M_2(\Z/n\Z)} \phi_n(M) M$, where
$$  
	\phi_n(M) =  \begin{cases} 
	n^4-n^3-n^2+n & \ifs \rank(M) = 0, \\ n-n^2  & \ifs \rank(M) = 1, \\ n & \ifs \rank(M) = 2. \end{cases} 
$$
In detail, this function $\phi_n$ is the sum of three functions $\phi_n^{0}$, $\phi_n^1$, and $\phi_n^2$ 
corresponding to the three terms in \eqref{endef}: $\phi_n^{(0)}$, arising from the term $n^4 (0)$,
equals $n^4$ in rank zero and is otherwise zero; $\phi_n^{(1)}$, arising from the term 
$-n^2T_n^{\varK}(0)$, equals $-n^2 \mathrm{deg}(T_n) = -n^3-n^2$ in rank $0$ and $-n^2$ in rank $1$,
and finally $\phi_n^{(2)}$, arising from the term $n E[n]^2$, is simply the constant function with value $n$.

The significance of this particular function $\phi_n$ is that if we push it forward to a $\Z$-valued function on $(\Z/n\Z)^2$
along any of the maps
$$M_2(\Z/n\Z) \rightarrow (\Z/n\Z)^2$$
which come by taking product with a fixed element of $(\Z/n\Z)^2$, 
then the result is {\em zero}. This characterizes it up to a scalar amongst $\GL_2(\Z/n\Z)$-invariant functions. 

Let us explain why this is a natural property to ask for. 
In the context of Proposition \ref{abstractcocyc}, if one wants
an explicit formula for $\Theta^Z$ as an external product
of $\theta$-functions, it is natural to ask that
$Z$ be an external product of the divisors of those $\theta$-functions.
In the coordinates just introduced, these correspond to functions
on $M_2(\Z/n\Z)$ of the form $\Phi(M)= f_1(M_1) f_2(M_2)$, where $M_1$ and $M_2$ are the rows of $M$
 and the $f_i \colon (\Z/n\Z)^2 \rightarrow \Z$ both satisfy $\sum_{x \in \Z/n\Z} f_i(x) = 0$. 
 If such a function $\Phi$ is additionally $\GL_2(\Z/n\Z)$-invariant,
 then its pushforward to $(\Z/n\Z)^2$ along any map $M \rightarrow vM$ with $\nu \in (\Z/n\Z)^2$ is zero.
\end{remark}

\subsection{The explicit cocycle for $n$} \label{expcocycE2}
 
We turn to the construction of our cocycle for a prime $n \nmid N$ and the verification of its explicit formula in terms of the symbols of \S \ref{symbolsE2}.  Recall from \S\ref{cocycle} that a cocycle is parabolic if it has trivial image in the cohomology of all stabilizers of 
nonzero elements of $\Z^2$ under the right action of $\GL_2(\Z)$.  Much as in Proposition \ref{expform}, for $\gamma \in \GL_2(\Z)$ with columns $v_1$ and $v_2$, we write $\langle v_1, v_2 \rangle_n$ for $\langle \gamma \rangle_n$.  Recall also that we defined the notion of a connecting sequence in \S \ref{cocycle}.   
\begin{theorem} \label{expformE2} Let $n$ be a prime not dividing $N$.
\begin{itemize}
\item[a.]
	There is a parabolic $1$-cocycle
	${}_n\Theta \colon \GL_2(\Z) \to \varK_2^{(0)}$
	uniquely characterized  by
		\begin{equation} \label{Thetachar} \partial({}_n\Theta_{\gamma}) 
		= (\gamma^*-1)\langle 0,1 \rangle_n \end{equation}  for all $\gamma \in \GL_2(\Z)$.
\item[b.] 		
	For $\gamma = \smatrix{a&b\\c&d} \in \GL_2(\Z)$ and a connecting sequence $(v_i)_{i=0}^k$ for $\gamma$, we have
	\begin{equation} \label{Thetachar2} 
		{}_n \Theta_{\gamma} = \sum_{i=1}^k \langle v_i, -v_{i-1} \rangle_n.
	\end{equation}
 \end{itemize}
 \end{theorem}
 
 \begin{remark}
 	In fact, Theorem \ref{expformE2} still holds over $\Z[\frac{1}{6}]$, so long as we suppose that $n \neq 5$.
 \end{remark}

In order to prove Theorem \ref{expformE2}, we first compute the residues of our symbols.

\begin{lemma} \label{residuesymbol0}
	The residue of $\langle 1,0 \rangle_n$ is $\newE_n$. 
\end{lemma}

\begin{proof}
	By \eqref{thetasdef}, we have
	\begin{eqnarray*}
		\partial(\delta \boxtimes \theta) = \delta \boxtimes  \partial \theta = \delta \boxtimes \delta &\mr{and}& 
		\partial(\delta' \boxtimes \theta') = \delta' \boxtimes \partial \theta' =  \delta' \boxtimes \delta',
	\end{eqnarray*}
	where for instance $\delta \boxtimes \delta$ denotes the evidently defined external product. 
 	By Lemma \ref{restra}, taking residues commutes with norms, and therefore
	$$
		\partial \langle 1,0 \rangle_n = \newf \boxtimes \newf- \Norm(\newf' \boxtimes \newf').
	$$
	
 	For the norm $\Norm$ corresponding to $(\mc{E}')^2 \to \mc{E}^2$, we have
 	$$\Norm(\mc{E}'[n]^2) = (n+1) \mc{E}[n]^2, \quad \Norm(\mc{K} \boxtimes \mc{E}'[n]) = \mc{E}[n]^2 +n (0) \boxtimes \mc{E}[n],
	\quad \text{and} \quad \Norm(\mc{K} \boxtimes \mc{K}) = T_n^{\varK}(0),
	$$
	where each of the equalities are inside $\varK^0$. (The final identity is a straightforward computation. See the comparison of
	\eqref{Tell1} and \eqref{Tell2} in the proof of Theorem \ref{heckeE2} below.)
	For $\newf' = n \mc{K} - \mc{E}'[n]$ as in \eqref{deltadef}, we then compute that
	\begin{align*}		
		\Norm(\newf' \boxtimes \newf') &= n^2 \Norm(\mc{K} \boxtimes \mc{K}) - n \Norm(\mc{K} \boxtimes \mc{E}'[n]) - 
		n \Norm(\mc{E}'[n] \boxtimes \mc{K}) + \Norm(\mc{E}'[n]^2)
		\\ &=  n^2 T_n^{\varK}(0) - n^2((0) \boxtimes \mc{E}[n] + \mc{E}[n] \boxtimes (0)) + (-n+1)\mc{E}[n]^2.
	\end{align*}
	 Recalling that $\delta = n^2(0) - \mc{E}[n]$ from \eqref{deltadef}, we have 
	 $$
	 	\newf \boxtimes \newf = n^4 (0) -  n^2 ((0) \boxtimes \mc{E}[n] + \mc{E}[n] \boxtimes (0)) + \mc{E}[n]^2,
	 $$	 
	 and we conclude from the formula \eqref{endef} defining $\newE_n$ that
	 $\newf \boxtimes \newf- \Norm(\newf' \boxtimes \newf') = e_n$.	
\end{proof}

\begin{lemma} \label{actiononunits} 	Let $\gamma = \smatrix{ a&b \\ c&d } \in \SL_2(\Z)$. The symbol $\langle a, c \rangle_n = \gamma^*\langle 1, 0 \rangle_n$ 
	does not depend on the choice of $(b,d)$ in $\gamma$, and its residue is $\newE_n$.
\end{lemma}

\begin{proof} 
	Since the residue map $\varK_1 \to \varK_0$ 
	is $\GL_2(\Z)$-equivariant and $\newE_n$ is $\GL_2(\Z)$-invariant, the symbol $\gamma^*\langle 1,0 \rangle$ has
	residue $\newE_n$ by Lemma \ref{residuesymbol0}. 

	For the first statement, it is enough to see that $\smatrix{1 & 1 \\ 0 & 1}^*$ fixes $\langle 1, 0 \rangle_n$.  For this,
	recall that $\smatrix{1 & 1 \\ 0 & 1}$ acts on points of $\mc{E}^2$ via the recipe $(E, P, Q) \mapsto (E, P, P+Q)$.  
	Both $\smatrix{1 & 1 \\ 0 & 1}^* \langle 1,0 \rangle_n$ and $\langle 1,0 \rangle_n$ are  
	meromorphic functions on 
	$\mc{E}[n] \times_{Y} \mc{E}$
	with the same residue $\newE_n$.
 	Moreover, they are both invariant under all maps $[m]_*$
	with $m \in \mb{N}_n$ by Lemma \ref{Symbols are Fixed}. They differ, then, by a regular function $f$ on $\mc{E}[n] \times_Y \mc{E} $
	that is fixed under all such $[m]_*$.    Now any regular function on $\mc{E}[n] \times_{Y} \mc{E}$ is necessarily
	constant along fibers of the map $\mc{E} \rightarrow Y$ in the second variable, and 
	thus $f$ is pulled back from a function $\bar{f}$ on $\mc{E}[n]$. Then
	the fact that $f$ is fixed  implies that $([m]_* \bar{f})^{m^2} = \bar{f}$ (where the exponent $m^2$
	arises from the degree of the map $[m]$ in the second variable). 
	 
	Now, if one takes $m \equiv 1 \bmod n$, then $[m]$ fixes $\mc{E}[n]$, and one deduces that
	$\bar{f}^{m^2} = \bar{f}$ for such $m$. 
	In particular, the value of $\bar{f}$ at any complex point is an $(m^2-1)$th root of unity, so
	 $\bar{f}$ is a constant on both of the geometric components of $\mc{E}[n]$ (the identity section
	and its complement). Since both of these components are preserved by every $[m]$, we have $\bar{f}^{m^2} = \bar{f}$ for all $m \in \mb{N}_n$. 
 	Such a function necessarily satisfies $\bar{f}^{24}=1$, and the class it induces in $H^1(\mc{E}[n] \times_{Y} \mc{E},\Z'(1))$ is therefore trivial. 
\end{proof}

\begin{lemma} \label{symbolfacts}
	Let $\gamma = \smatrix{ a&b \\ c&d } \in \GL_2(\Z)$.
 	Then the residue of $\langle \gamma \rangle_n$ is given by
	$$
		\partial \langle \gamma \rangle_n = \begin{cases} \langle a,c \rangle_n -\langle -b, -d\rangle_n & \ifs \det(\gamma) = 1, \\
		\langle -a,-c \rangle_n - \langle b,d \rangle_n & \ifs \det(\gamma) = -1. \end{cases}.
	$$
 \end{lemma}

\begin{proof}
	We omit the subscripts $n$ in this proof for brevity of notation, and handle the case $\det(\gamma)=1$, the
	other case being similar. By definition \eqref{tamesymbol} of the tame symbol, we have
	$$ \partial \theta \boxtimes \theta = \newf \boxtimes \theta - \theta \boxtimes \newf$$
	and similarly for $\theta' \boxtimes \theta'$ in $\varK'$.   
 	Taken together with the compatibility of residues with norm of Lemma \ref{restra}, these imply
	$$
		\partial\langle \smatrix{1&0\\0&1} \rangle_n = (\newf \boxtimes \theta - 
		\theta \boxtimes \newf) - \Norm(\newf' \boxtimes \theta'- \theta'\boxtimes \newf').
	$$
	We then compute 
	$$
		\langle -b,-d \rangle_n =  \Pmatrix{ -b & a \\ -d & c}^* \langle 1,0 \rangle_n = \Pmatrix{ a & b \\ c & d }^* 
		\underbrace{ \Pmatrix{ 0 & 1 \\ -1 & 0}^*(\newf \boxtimes \theta - \Norm(\newf' \boxtimes \theta'))}_{\theta \boxtimes \newf
		- \Norm(\theta' \boxtimes \newf')}.  	
	$$
 	 The step under the braces follows readily from the fact that  $\theta, \newf, \theta', \newf'$ are all invariant under $[-1]^*$. 
	 We also used the fact that $\mathrm{N}$ commutes with $\smatrix{0&1\\-1&0}^*$,
	 which follows by Lemma \ref{pullbackdiagram}. 
\end{proof}

With the residues of the symbols attached to $n$ computed, the main theorem follows as in the case of $\gm^2$.

\begin{proof}[Proof of Theorem \ref{expformE2}]
	The existence and uniqueness of 
	${}_n\Theta$ in part b follows from Lemma \ref{actiononunits} as in the proof of 
	Proposition \ref{abstractcocyc}.  That it is parabolic follows as in Proposition \ref{parabolic}
	from the fact that $\gamma^*\langle 0, 1 \rangle_n = \langle 0, 1 \rangle_n$ for $\gamma = \smatrix{ 1&0 \\ c & \pm 1 }$, again by 
	Lemma \ref{actiononunits}.  Part c then follows as $\partial \sum_{i=1}^k \langle v_i, - v_{i-1} \rangle_n = 
	(\gamma^*-1)\langle 0,1 \rangle_n$ by Lemma \ref{symbolfacts}, as in the proof of Proposition \ref{expform}.
 \end{proof}

 \subsection{Hecke actions} \label{HeckeE}
 
We study Hecke operators on the complex $\varK$ arising from correspondences, and we compare their action on the class of the cocycle ${}_n\Theta$ with that of the previously-defined Hecke operators on group cohomology (see Lemma \ref{Heckelemma}). 
 
 \subsubsection{Hecke operators via correspondences}  \label{T'Hecke}
Let us define Hecke operators using correspondences on $\mc{E}$.  We restrict ourselves to $m$th Hecke operators $T'_m$
for $m \ge 1$ prime to the level $N$. 
We have a commutative diagram
\begin{equation} \label{corresp}
	\begin{tikzcd}
		\mc{E} \ar[d,"\pi"] & \mc{E} \times_Y Y_m \arrow{d}{(\pi,\id)} \ar[l,"\Phi"'] \ar[r,"\Psi"]  & \mc{E} \ar[d,"\pi"] \\
		Y & Y_m \ar[l,"\phi"'] \ar[r,"\psi"] & Y,
	\end{tikzcd}
\end{equation}
where the effect of $\phi$ and $\psi$ on points is given by
\begin{eqnarray} \label{phipsidef}
	\phi(E,P,K) = (E,P) &\mr{and}& \psi(E,P,K) = (E/K, P+K).
\end{eqnarray}
The morphisms $\Phi$ and $\Psi$ are then defined on fibers by 
the identity on $E$ for $\Phi$ and taking the image under $E \to E/K$ for $\Psi$.  We also then have morphisms $\Phi^2$ and $\Psi^2$ sending $\mc{E}^2 \times_Y Y_m$ to $\mc{E}^2$.   
All these maps are finite {\'e}tale. 

We define Hecke operators $T_m'$ on the motivic cohomology of $Y$ and $\mc{E}^2$  by the respective rules
\begin{eqnarray} \label{Tmpdef}
	 T'_m =  \phi_*\psi^* &\mr{and}& T'_m = \Phi^2_*(\Psi^2)^*.
\end{eqnarray}

We also have operators $[m]'$ acting on motivic cohomology of $Y$ and $\mc{E}^2$, given by pullback under multiplication by $m$: i.e., by the morphisms given by $m(E,P) = (E,mP)$ on points of $Y$ and given on $\mc{E}^2$ by taking a point $x$ in the fiber $E^2$ of $(E,P)$ to the point $mx$ in the fiber $E^2$ of $(E,mP)$.\footnote{In particular, while the operators $[m]^*$  arise from a fiber-preserving map
over $Y$, the operators $[m]'$ do not.}
Note that the operators $[m]'$ arise from diagrams of the same form as \eqref{corresp}, but replacing $Y_m$ by $Y$, taking $\phi$ to be the identity map, and defining $\psi$ by $\psi(E, P) = (E, mP)$.   In this way,  arguments given for $T_m'$ will usually adapt to $[m]'$ without change. 
 
The reader might ask why we use the notation $T'_m$, as opposed to $T_m$.  The point is this: when we deal with cocycles
\begin{equation} \label{testmoo} \mbox{congruence subgroup of $\GL_2(\Z)$} \longrightarrow \mbox{$K$-group of function field of $\mc{E}^2$,}
\end{equation} 
there are two reasonable definitions of Hecke operators (both of which preserve coboundaries).  
\begin{itemize}
\item The fiberwise $\GL_2(\Z)$-action on $K_2(k(\mc{E}^2))$ extends to an action of $M_2(\Z) \cap \GL_2(\Q)$. Therefore, we can define the $m$th Hecke operator $T_m$ on $1$-cocycles valued in $K_2(k(\mc{E}^2))$ as in Section \ref{actions}.
\item  The definition  \eqref{Tmpdef} also defines operators $T'_m$ on the $K$-group of the function field of $\mc{E}^2$. This induces
an operator on cocycles as in \eqref{testmoo}, also denoted $T'_m$.
\end{itemize}

We note that the action of $T_m$ would exist if we replaced $\mc{E} \rightarrow Y$ by any other family of elliptic curves, whereas $T'_m$ requires that we work with the universal elliptic curve.  
The primary result of this subsection, Theorem \ref{heckeE2} below, is that the these two operators {\em coincide} on the class of the cocycle ${}_n \Theta$.  

\subsubsection{Hecke equivariance of the cocycle}

We can also define Hecke operators $T'_m \colon \varK_i \to \varK_i$ on the terms of $\varK$ via $\Phi^2_*(\Psi^2)^*$.  These give $\Delta$-equivariant morphisms of complexes, where again $\Delta = M_2(\Z) \cap \GL_2(\Q)$.
  
\begin{lemma} \label{correspcplx}
	For each $m \ge 1$ prime to $N$, both $T_m' \colon \varK \rightarrow \varK$ and $[m]' \colon \varK \to \varK$
	are maps of complexes which are equivariant for the pullback action of $\Delta$ for
	its right action on $\mc{E}^2$.
\end{lemma}

\begin{proof}
	The maps $T_m'$ and $[m]'$ are compositions of \'etale pullbacks and  finite pushforwards (transfers) in the $K$-theory of fields, and  as such commute with residue maps (see Lemma \ref{restra}).   Thus $T_m'$ and $[m]'$ define maps of complexes. 
	
	The $\Delta$-action on $E^2$ for an elliptic curve $E$ is equivariant for the reduction $E^2 \to (E/H)^2$ for any finite 
	subgroup scheme $H$, so $(\Psi^2)^*$ is equivariant for the pullback action of $\Delta$.  
	The operators $[m]'$
	are $\Delta$-equivariant, in particular since multiplication by $m$ commutes with the $\Delta$-action on $E^2$.
	 For $\Phi$, we note that
	 	$$
 		\begin{tikzcd}
  		\mc{E}^2 \times_{Y} Y_m \arrow{d}{\delta} \arrow{r}{\Phi^2} & \mc{E}^2 \arrow{d}{\delta} \\
  		\mc{E}^2  \times_{Y} Y_m \arrow{r}{\Phi^2} & \mc{E}^2 
  		\end{tikzcd}
  	$$
	is Cartesian for any $\delta \in \Delta$ in that the morphism $\Phi^2$ is flat and the identity on fibers.  
	Therefore, $(\Phi^2)_*$ commutes with pullback by $\Delta$, again employing Lemma \ref{pullbackdiagram}.
	Thus, the Hecke actions and pullback $\Delta$-actions commute. 
\end{proof} 

For $m \ge 1$ prime to $N$, let us use $T_m^{\varK}$ to denote any sum of pullbacks by representatives of the
double coset of $\smatrix{m & \\ & 1}$, as in \S \ref{actions}.  (The choice is unimportant, but there is a standard one.)
 
 \begin{lemma}  \label{Hecke preserves fix}  Let $m \ge 1$ be prime to $N$.
	\begin{enumerate}
		\item[a.] The Hecke operators $T'_m$ and $[m]'$ on $\varK$ commute with $[\mu]_*$ 
		for all $\mu$ prime to $m$,  and in particular they preserve $\varK^{(0)}$.  
		\item[b.] 
		The Hecke operators $T_m^{\varK}$ and $[m]^*$ on $\varK$ commute with all $[\mu]_*$ for $\mu$ prime
		to $m$, and in particular they preserve   $\varK^{(0)}$.
	\end{enumerate}
\end{lemma}

\begin{proof}
	The commutativity of $[\mu]_*$ with the pushforward map $\Phi^2_*$ is automatic because $[\mu]$ and $\Phi^2$ commute.  
  	To see the commutativity with $(\Psi^2)^*$, we note that
 	$$
	\begin{tikzcd}
  	\mc{E}^2 \times_{Y} Y_m \arrow{d}{[\mu]} \arrow{r}{\Psi^2} & \mc{E}^2 \arrow{d}{[\mu]} \\
  	\mc{E}^2  \times_{Y} Y_m \arrow{r}{\Psi^2} & \mc{E}^2 
 	\end{tikzcd}
  	$$
  	is a cartesian diagram, which in turn amounts to the fact that the degree $m$ isogenies
  	$E \rightarrow E/K$  underlying $\Psi$ (see \eqref{phipsidef})  induce isomorphisms on $\mu$-torsion.
  	Then we apply Lemma \ref{pullbackdiagram}. A similar argument applies to both $[m]'$ and $[m]^*$, in that they are
	also isomorphisms on $\mu$-torsion. 
 	Finally, the argument for $T_m^{\varK}$ was already given in the course of Lemma \ref{Symbols are Fixed}. 
\end{proof}

\begin{remark} \label{T5}  
	Lemma \ref{Hecke preserves fix} for $m=5$ requires that $5$ be invertible in $\Z'$ and is the reason we invert $5$.  
	If we instead tried to work
	with $\Z[\frac{1}{6}]$-coefficients, then to maintain the left-exactness of $\varK^{(0)}$ in Lemma \ref{bigcplxleftex}, 
	we would need to add the additional requirement that its elements are fixed by $[5]_*$. However, the Hecke operators 
	with $m = 5$ do not preserve $[5]_*$-fixed parts.
\end{remark}

The action of $\Delta$ on the complex $\varK$ provides Hecke operators
$T_m$ on $H^1(\GL_2(\Z), \varK_2^{(0)})$ for $m \nmid N$, following the recipe of \S \ref{actions} for the double coset of 
$\smatrix{m\\&1}$. The various Hecke operators all commute with one another.

\begin{lemma} \label{formality}
	Every two Hecke operators in the collection of operators $T_m$, $T'_m$, $[m]^*$, and $[m]'$ for $m \ge 1$ prime to $N$ 
	commute with each other in their action on $H^1(\GL_2(\Z), \varK_2^{(0)})$.  
\end{lemma}

\begin{proof}  
	First, note that these operators all act on  $H^1(\GL_2(\Z), \varK_2^{(0)})$ since they (or, in the case of $T_m$, the operators
	$T_m^{\varK}$) preserve fixed parts by Lemma \ref{Hecke preserves fix}.
Commutativity between operators of the form $T_m$ or $[m]^*$ and commutativity between operators of the form $T'_m$ or $[m]'$
for various $m$ is standard. The operators $[m]^*$ already commute with the operators $T'_m$ or $[m]'$ on $\varK_2$ 
by Lemma \ref{correspcplx}.

Given a $1$-cocycle $\theta \colon \GL_2(\Z) \to \varK_2$, the cocycle
$T_m \theta$ is defined by the formula of \eqref{Heckecocyc} for $g = \smatrix{m\\&1}$.  It is a sum
of terms of the form $[\delta]^* \theta(\gamma')$ with $\gamma'\in \GL_2(\Z)$
and $\delta \in \Delta$.   Any $T'_{\mu}$ or $[\mu]'$ for $\mu \in \mb{N}_N$
commutes with each $[\delta]^*$ by Lemma \ref{correspcplx}, 
so also commutes with $T_m$ on $\theta$. 
\end{proof}

We now proceed to the main result of this section, which unlike the preceding lemmas is not a formality. 

\begin{theorem} \label{heckeE2}
	For each prime $\ell \nmid N$, the actions of $T_{\ell}$ and $T_{\ell}'$ coincide on the class of ${}_n \Theta$ in 
	$H^1(\GL_2(\Z), \varK_2^{(0)})$.
	The same is true for the actions of $[\ell]^*$ and $[\ell]'$.    
\end{theorem}

\begin{proof} 
	Let us say that a cocycle $\theta \colon \GL_2(\Z) \rightarrow \varK_2^{(0)}$
	is associated to $Z \in \varK_0^{(0)}$ if $\partial \theta(\gamma) = (\gamma^*-1) \eta$ with $\eta \in \varK_1^{(0)}$ and  $\partial \eta =Z$. 
	Recall that, by Proposition \ref{abstractcocyc}, any two cocycles associated to $Z$ are cohomologous.  To show
	that the classes of $T_{\ell}'({}_n\Theta)$ and $T_{\ell}({}_n\Theta)$ coincide, it therefore suffices to show
	that $T_{\ell}({}_n\Theta)$ and $T'_{\ell}({}_n\Theta)$ are associated to the same cycle.  Let us consider the two cases:
	\begin{itemize}
		\item  The cocycle $T_{\ell} ( {}_n \Theta)$  
		 is associated to $T_{\ell}^{\varK} \newE_n$.  Indeed, 		
		 $\partial(T_{\ell}^{\varK} \langle 0,1 \rangle_n) = T_{\ell}^{\varK} \newE_n$ as $T_{\ell}^{\varK}$ is a map of complexes, and 
		$T_{\ell}^{\varK} \langle 0,1 \rangle_n$ belongs to $\varK_1^{(0)}$
		by Lemma \ref{Hecke preserves fix}.  Moreover, for $\gamma \in \GL_2(\Z)$, we have
		$$
		 (\gamma^*-1)T_{\ell}^{\varK} \langle 0,1 \rangle_n = \partial (T_{\ell}\Theta)_{\gamma}
		$$
		exactly as in equation \eqref{TT} of the proof of Proposition \ref{ThetaEis}. 
				
 		\item The cocycle $T_{\ell}'  ({}_n \Theta)$ is associated to $T_{\ell}' \newE_n$, since 
		$T_{\ell}' \colon \varK \to \varK$ is a map of complexes that commutes with the $\GL_2(\Z)$-action by Lemma 
		\ref{correspcplx} and preserves fixed parts by Lemma \ref{Hecke preserves fix}.
	\end{itemize} 
	
 		We must therefore show that $T_{\ell}' \newE_n = T_{\ell}^{\varK} \newE_n$.  		
		We claim that it is enough to show the same assertion but replacing  $\newE_n$ by the $\GL_2(\Z)$-fixed class $(0) \in \varK_0$.
		Indeed,  $\newE_n = V_n^{\varK} (0)$, with notation as in \eqref{Vn0}, and 
		  $T_{\ell}^{\varK}$ and $V_n^{\varK}$ commute in their action on the $\GL_2(\Z)$-invariant subgroup 
	of $\varK_0$, whereas $T_{\ell}'$ and $V_n^{\varK}$ commute   by Lemma \ref{correspcplx}
	(noting $V_n^{\varK}$ is a sum of various pullback maps).

%
	It remains therefore only to show that $T_{\ell}'(0) = T_{\ell}^{\varK}(0)$.
	We will describe the fibers of $T_{\ell}'(0)$ and $T_{\ell}^{\varK}(0)$ over a geometric point $s$ of $Y$ and show that they coincide. 
	\begin{itemize}
		\item The fiber of $T_{\ell}'(0)$ is the union of the kernels of all $\varphi^2 \colon E_s^2 \to (E')^2$,
	where $E_s$ is the fiber of $\mc{E}$ over $s$ and $\varphi \colon E_s \rightarrow E'$ is an $\ell$-isogeny.
	In other words, it is the sum of all $K \times K$ where $K$ is a cyclic subgroup scheme of $E_s$ of order $\ell$:
	\begin{equation} \label{Tell1} T_{\ell}'(0) = \sum_{K} K \times K.\end{equation} 	
		\item The fiber of $T_{\ell}^{\varK}(0)$ above $s$ is given
	by those points of $E_s$ in the kernel of some matrix $g_j$ as in \eqref{Gjdef}. 
	Regarding $j$ as valued in $\mathbb{P}^1(\F_{\ell})$, the kernel of $g_j$
	is the set of pairs  $(P,Q) \in E_s[\ell]^2$ such that 
	and $Q/P =-j$
	(by which we mean that if we write $j =a/b$, then $aP+bQ=0$), and thus 
	\begin{equation} \label{Tell2} 
		T_{\ell}^{\varK}(0) = \sum_{j \in \mathbb{P}^1(\F_{\ell})} \{ (P, Q) \in E_s[\ell]^2 \mid  P/Q = j \}. 
	\end{equation} 
	\end{itemize}
	
	One easily checks that \eqref{Tell1} and \eqref{Tell2} coincide. For example, 
	if we choose a basis to identify $E_s[\ell]$ with $\F_{\ell}^2$ and use this to identify $E_s[\ell]^2$
	with $M_2(\F_{\ell})$ with the columns giving the coordinates, then \eqref{Tell1} and \eqref{Tell2}
	become identified with the formal sums of matrices with linearly dependent rows and linearly dependent columns, respectively, 
	in both cases counting the zero matrix with multiplicity $\ell+1$.   
	
 	Finally, to see that the classes of $[\ell]^*({}_n \Theta)$ and $[\ell]'({}_n \Theta)$ are equal, 
	it is similarly enough to show that $[\ell]^*(0) = [\ell]'(0)$. 
	This is immediate: 
	both amount to pullback of the zero section by the matrix $\smatrix{ \ell & 0 \\ 0 & \ell }$, so equal $\mc{E}[\ell]^2$.
\end{proof}

\section{Cocycles for modular curves} \label{modcocyc}

In this section, we pull back the cocycles of Section \ref{E2} to the modular curve via a torsion section.  

As before, we fix $N \ge 4$ and a prime $n \nmid N$, and we write $Y = Y_1(N)$
and $\mc{E}$ for the universal elliptic curve above $Y$.  The surjection $\pi \colon \mc{E} \to Y$ has a canonical $N$-torsion 
section $\iota_N \colon Y \to \mc{E}[N]$.  
In \S \ref{specializeE2}, we pull back the cocycle ${}_n\Theta$ of Theorem \ref{expformE2}
by the section of $\pi^2 \colon \mc{E}^2 \to Y$ given by
$$
 	s = (0,\iota_N) \colon Y \rightarrow \mc{E}^2.
$$
We denote the result as ${}_n \Theta_N$. 
In order to make sense of such a pullback, we must  -- 
as in the case of $\gm^2$ described in \S \ref{specialization} -- 
restrict our cocycle to the congruence subgroup $\tilde{\Gamma}_0(N)$ of $\GL_2(\Z)$ defined in \eqref{tildegammadef} to consist of matrices with lower-left entry divisible by $N$.

In \S \ref{universal}, we describe modifications that enable us to obtain a universal cocycle $\Theta_N$ that should be thought of as the ``$n=1$''
version of the construction: see Theorem \ref{canoncocyc}. Much as with theta-functions, we do not know how to make sense of this on the universal elliptic curve, but we can after pullback. 
The characterizing property 
of $\Theta_N$ is that it gives rise to each ${}_n \Theta_N$ upon application of the Hecke operator $V_n = n^4 - n^2 T_n + n[n]^*$ or its counterpart $V'_n = n^4 - n^2 T'_n + n[n]' $.  
In \S \ref{expformuniv}, we prove an explicit formula for this  universal cocycle $\Theta_N$ modulo a subgroup that vanishes under standard regulator maps. 

Finally, in \S \ref{homologymod}, we construct the zeta map $z_N$ of \eqref{zetamap} and compare with the prior work of Goncharov, Brunault, and Fukaya-Kato.  The map $z_N$ is constructed in Theorem \ref{zeta_map}, where we show that it is Hecke equivariant and takes values in the motivic cohomology of $X_1(N)$ (over $\Z[\frac{1}{N}]$), as opposed to $Y_1(N)$. We also describe an integral, ordinary $p$-adic analogue in Proposition \ref{FKzeta}.
 
We suppose that $\Z' = \Z[\frac{1}{30}]$ throughout this section.

\subsection{Specialization via an $N$-torsion section} \label{specializeE2}

In this subsection, we pull back our cocycles ${}_n \Theta$ for primes $n \nmid N$ via the $N$-torsion section $s \colon Y \to \mc{E}^2$ 
to obtain cocycles
${}_n \Theta_N \colon \bGamma_0(N) \to H^2(Y,\Z'(2))$.
 
\subsubsection{Comparison of Hecke operators upon restriction}

As in the case of $\gm^2$, the section $s$ is not defined on all of $\varK_2$ but 
at least on classes ``regular along $s$''. Writing  
$$\varKs = \varinjlim_{(\Z/N\Z)^{\times}s \subset U} H^2(U, \Z'(2)) \subset \varK_2,$$
where $U$ runs over the open $Y$-subschemes of $\mc{E}^2$ containing all prime-to-$N$ multiples of the image of $s$,
we have a specialization map
$$
	s^* \colon \varKs \to H^2(Y,\Z'(2)),
$$
and similarly we can pull back by any prime-to-$N$ multiple of $s$.\footnote{In the case of 
$\gm^2$, the point $(1,\zeta_N)$ was defined over $\Q(\mu_N)$, whereas the subschemes $U$ were defined over $\Q$, 
so the containment of $(\Z/N\Z)^{\times}s$ for $s \in U(\Q(\mu_N))$ was automatic.}

 Now, for $\gamma = \smatrix{a&b\\c&d} \in \tilde{\Gamma}_0(N)$, 
 we have in fact ${}_n \Theta_{\gamma} \in \varKs$. Indeed, 
 write  $
	U_{\gamma} = \mc{E}^2 - S_{0,n} \cup S_{nb,nd}$, 
that is to say, $U_{\gamma}$ is the complement 
in $\mc{E}^2$ of the kernels of the maps $\mc{E}^2 \rightarrow \mc{E}$ defined by $(P,Q) \mapsto n Q$ and $(P,Q) \mapsto n (bP +dQ)$.
Then $\partial({}_n\Theta_{\gamma}) = (\gamma^*-1)\langle 0,1 \rangle_n$ lies in $H^1(S_{0,n},1)^{(0)} \oplus H^1(S_{nb,nd},1)^{(0)}$ inside $\varK_1^{(0)}$. Since $s$ and its multiples do not lie on either $S_{0,n}$  or $S_{nb,nd}$, it follows that ${}_n\Theta_{\gamma} \in H^2(U_{\gamma},2)$.  Moreover, the image of  any prime-to-$N$ multiple of $s$ is contained in  $U_{\gamma}$ as $N \nmid d$. In this way, the cocycle ${}_n \Theta$ restricted to $\bGamma_0(N)$ takes values in $\varKs \subset \varK_2$.\footnote{At this point, we have no further need for trace fixed parts, but of course ${}_n\Theta$ takes values in $\varK_2(N) \cap \varK_2^{(0)}$.}
  
The operators $T_{\ell}$ and $T_{\ell}'$ act on $H^1(\bGamma_0(N),\varK_2)$ and lift naturally to $H^1(\bGamma_0(N), \varKs)$. 
 For $T_{\ell}$, this is simply because the action of $\Delta_0(N)$, as defined in \eqref{Delta0def},
 sends the section $s$ to a multiple of itself and therefore preserves $\varKs$.
 For $T_{\ell}' = \Phi^2_*(\Psi^2)^*$ as in \S \ref{T'Hecke}, consider the diagram
 $$
	\begin{tikzcd}
		\mc{E}^2 & \mc{E}^2 \times_Y Y_{\ell}  \ar[l,"\Phi^2"'] \ar[r,"\Psi^2"]  & \mc{E}^2   \\
		Y  \ar[u,"s"] & Y_{\ell} \ar[l,"\phi"'] \arrow{u}{(s,\id)} \ar[r,"\psi"] & Y \ar[u,"s"],
	\end{tikzcd}
$$
and note that $\Psi^2$ preserves $s$, whereas the preimage
of the image of $s$ under $\Phi^2$ is again  the image of $(s,\id)$.

\begin{lemma} \label{K2ss}
For primes $\ell \nmid N$, the classes of $T'_{\ell}({}_n \Theta)$ and
$T_{\ell}({}_n \Theta)$ coincide in $H^1(\bGamma_0(N),  \varKs)$, as do the classes of $[\ell]'({}_n\Theta)$ and $[\ell]^*({}_n\Theta)$.
\end{lemma} 	 

\begin{proof}  
Theorem \ref{heckeE2} implies that $T'_{\ell}( {}_n \Theta)$ and $T_{\ell}( {}_n \Theta)$, as well as $[\ell]'({}_n\Theta)$ and $[\ell]^*({}_n\Theta)$, 
are cohomologous when considered with target $\varK_2$. So it is enough to check the following claim regarding the inclusion $\varKs \hookrightarrow \varK_2$:
\begin{equation} \label{inclusion}
	\mbox{for any $H \leqslant \GL_2(\Z)$ of finite index, $H^1(H, \varKs) \to H^1(H,\varK_2)$ is injective. }
\end{equation}

This injectivity will follow from the Gysin sequence analogous to \eqref{GysinK2N} into which the above inclusion fits
if one proves the infinitude of all $\GL_2(\Z)$-orbits of irreducible divisors on $\mc{E}^2$ containing the image of $s$.  
Such a divisor induces a divisor on the fiber of $\mc{E}^2$ over the generic point of $Y_1(N)$. 
Restricting to this fiber, it is enough to prove
that, for a non-CM elliptic curve over a field $K$  (in our case, the function field of $Y_1(N)$)
 and an irreducible $K$-divisor $D$ on $E^2$, the $\SL_2(\Z)$-orbit of $D$
 is infinite.  In fact, this is even true at the level of the N{\'e}ron-Severi group:  by \cite[Theorem 4.2]{RosenShnidman}, the $\Q$-vector space
 $\mathrm{NS}(E^2) \otimes_{\Z} \Q$ realizes the representation 
 of $\SL_2(\Z)$ on binary quadratic forms. In this representation, all nonzero orbits are infinite,
 and the class of   $D$ in the N\'eron-Severi group is nonzero because its intersection with a suitable hyperplane section is nonzero.  
 \end{proof}

 \subsubsection{Specialization of the cocycles}
   
Through the right action of $\Delta = M_2(\Z) \cap \GL_2(\Q)$ on $\mc{E}^2$, which preserves fibers, any
$\delta = \smatrix{a&b\\c&d} \in \Delta_0(N)$ acts on the $N$-torsion sections of 
$\mc{E}^2 \to Y$.   This action of $\Delta_0(N)$ does not preserve the section $s$. 
Indeed, let us agree to write points of $\mc{E}^2$ 
as triples $((E, P), x, y)$, where $E$ is an elliptic curve and $P$ is
an $N$-torsion point on $E$ (so that $(E,P)$ defines a point of $Y$) and $x$, $y$
are points of $E$. With this notation, we compute $\delta \circ s$: 
$$  
	(E, P) \stackrel{s}{\longmapsto} ((E, P), 0, P) \stackrel{\delta}{\longmapsto}   ((E, P), 0, dP).
$$
This does not coincide with $s \circ [d]'$, where $[d]'$ for $d$ prime to $N$
is the diamond operator on $Y_1(N)$ that sends $(E,P)$ to $(E, dP)$. Rather,
\begin{equation}
	\label{section equality}
	\delta \circ s = \phi_d^{-1} \circ s \circ [d]',
\end{equation}
where $\phi_d \colon \mc{E}^2 \rightarrow \mc{E}^2$ 
  sends $((E, P), x, y)$ to  $((E, d  P), x, y).$

Consider the action of $\Delta_0(N)$ on $H^2(Y_1(N),\Z(2))$ whereby $\smatrix{a&b\\c&d} \in \Delta_0(N)$ acts as $[d]'$, i.e.,
$\Delta_0(N)$ acts through its lower right-hand map to $(\Z/N\Z)^{\times}$.

\begin{lemma}  \label{phifix} 
The restriction of the pullback map 
$$
	s^* \colon \varKs \rightarrow H^2(Y, \Z'(2))
$$
to the $\Z[\Delta_0(N)]$-span of the image of ${}_n \Theta$ on $\bGamma_0(N)$ is $\Delta_0(N)$-equivariant.
\end{lemma}

\begin{proof} 
In fact, if $x \in \varK_2(N)$ is fixed by $\phi_d^*$ with $\phi_d$ as defined above, then by \eqref{section equality} we have
$$
	s^* \circ \delta^* (x) = [d]' \circ s^* \circ (\phi_d^{-1})^* (x)
	= [d]' s^*(x).
$$
Thus, we need only show that
\begin{equation} \label{phifix2}  
	\phi_d^*({}_n \Theta_{\gamma}) = {}_n \Theta_{\gamma}
\end{equation} 
for all $\gamma \in \widetilde{\Gamma}_0(N)$.

By the characterization of ${}_n \Theta$ in part a of Theorem \ref{expformE2}, 
it is sufficient to verify that  $\phi_d^*$ preserves $\varK_2^{(0)}$ and fixes $\langle 0, 1 \rangle_n$.  
It preserves $\varK_2^{(0)}$ as the relevant diagram with $\phi_d$ and $[m]$ is evidently Cartesian,
and it fixes $\langle 0,1 \rangle_n$ since the latter is ``pulled back from level $1$''; in particular, it restricts to the same function on the fiber $E^2$ over $(E,P)$ and $(E,dP)$.
\end{proof}
 
Recall that $T_{\ell}$ acts on the group of cocycles $\tilde{\Gamma}_0(N) \to H^2(Y, \Z'(2))$ as in \eqref{Heckecocyc}
(where we view $H^2(Y,\Z'(2))$ as a $\Z[\Delta_0(N)]$-module as above), preserving coboundaries, whereas 
$T_{\ell}'$ acts on such cocycles through its action on the motivic cohomology of $Y$ defined in \eqref{Tmpdef}.  The foregoing lemmas, taken together, have established the following.
 
\begin{proposition} \label{modcocycprop}
		For $\gamma \in \tilde{\Gamma}_0(N)$, set
		\begin{equation} \label{nNThetadef}
			{}_n \Theta_{N,\gamma} = s^* ({}_n \Theta_{\gamma}) \in H^2(Y,\Z'(2)).
		\end{equation}
 		\begin{enumerate}
			\item[a.] The map 
			$$
			{}_n\Theta_N \colon \tilde{\Gamma}_0(N) \to H^2(Y,\Z'(2)), \quad \gamma \mapsto {}_n\Theta_{N,\gamma}
			$$ 
			is a parabolic cocycle.
			\item[b.] For each prime $\ell \nmid N$, the cocycles $T_{\ell}({}_n\Theta_N)$ and 
			$T'_{\ell}({}_n\Theta_N)$ are cohomologous. 
		\end{enumerate}	
\end{proposition}

\begin{proof}
	For part a, that ${}_n \Theta_N$ is a cocycle clear from Lemmas \ref{K2ss} and \ref{phifix}.
	That it is parabolic at all but the parabolic $Q = \{ \smatrix{\pm 1 & n \\ 0 & 1} \mid n \in \Z\}$
	follows from a nearly identical argument to that of Proposition \ref{cyclcocycle}, using
	the parabolicity of ${}_n\Theta$ in Theorem \ref{expformE2}b and the equivariance of
	$s^*$ of Lemma \ref{phifix}. However, to see that ${}_n \Theta_N$ is a coboundary on the exceptional parabolic
	$Q$, we argue differently. 
	Since $2$ is invertible in $\Z'$ and $Q \cong \Z \rtimes \Z/2\Z$, it suffices
	to see that ${}_n \Theta_N$ vanishes on the generator $\smatrix{1 & -1 \\ 0 & 1}$ of
	the unipotent subgroup of $Q$.   Using
	 \eqref{Thetachar2} for the connecting sequence $v_0=(0,1),  v_1=(-1,1)$ for $\smatrix{1 & 1 \\ 0 & 1}$,
	and then applying \eqref{1001def} and
	  \eqref{gammalrndef}, we obtain  		
	 $$
		{}_n \Theta_{N,\gamma_0} = s^* ({}_n \Theta_{\gamma_0})  =
		s^*( \langle (-1, 1), (0,-1) \rangle)  
		= s^* \langle \smatrix {-1 & 0 \\ 1 & -1}\rangle^*  ( \theta \boxtimes \theta - \Norm (\theta' \boxtimes \theta')).
	$$
	 Since $ \langle \smatrix {-1 & 0 \\ 1 & -1}\rangle$ applied to the section $s=(0,\iota_N)$
	 gives $(\iota_N, -\iota_N)$, this becomes
	$$ 
		\iota_N^*\theta \cup (-\iota_N)^*\theta - \Norm ((\iota'_N)^*\theta'
		\cup (-\iota'_N)^*\theta'),
	$$
	where $\iota'_N \colon Y' \to (\mc{E}')^2$ is the canonical $N$-torsion section and the norm is from 
	$Y'$ to $Y$. 
	The two terms in the above expression vanish by the evenness of $\theta$ and $\theta'$ 
	and the skew-symmetry of the cup product with $\Z'$-coefficients (cf. \cite[Theorem 15.9]{mvw}).
	
 	As for part b, let $\ell$ be a prime not dividing $N$.  
	By Lemma \ref{phifix}, we have $s^* T_{\ell} ({}_n \Theta) = T_{\ell} ({}_n \Theta_N)$.	
	Moreover, $\iota_N^* \Phi_*\Psi^* = \phi_*\psi^*\iota_N^*$ since in the diagram
	$$
	\begin{tikzcd}
		\mc{E} & \mc{E} \times_Y Y_m  \ar[l,"\Phi"'] \ar[r,"\Psi"]  & \mc{E}   \\
		Y  \ar[u,"\iota_N"] & Y_m \ar[l,"\phi"'] \arrow{u}{(\iota_N,\id)} \ar[r,"\psi"] & Y \ar[u,"\iota_N"],
	\end{tikzcd}
	$$
 	the right-hand square commutes and the left-hand square is cartesian. The analogue for $\mc{E}^2$ 
	then holds with $\iota_N$ replaced by $s$, and
	therefore we have $s^*T'_{\ell}({}_n \Theta) = T'_{\ell}({}_n \Theta_N)$ as well.   
	Again recalling Lemma \ref{K2ss}, we conclude that the cohomology classes of 
	$T_{\ell}({}_n\Theta_N)$ and $T'_{\ell}({}_n\Theta_N)$ are equal. 
\end{proof}

There is also a formula for ${}_n\Theta_{N,\gamma}$ as a sum of cup products of Siegel units that follows in the obvious way by specializing \eqref{Thetachar2}; we do not write it down here, but we will discuss its ``$n=1$'' analogue in the next section. 

\subsection{The universal ``$n=1$'' cocycle} \label{universal}
The cocycle $${}_n \Theta_N \colon  \bGamma_0(N) \rightarrow H^2(Y_1(N), \Z'(2))$$
constructed above depends on the choice of an auxiliary prime $n$
in addition to the level $N$. 
As we shall detail, it satisfies a simple distribution relation in $n$
that permits us to construct an ``$n=1$'' version rationally. 

\subsubsection{Relation between cocycles and statement of the result}

Suppose that $\ell$ is a prime with $\ell \nmid N$.  Set 
\begin{eqnarray*}
	V'_{\ell} = \ell^4 - \ell^2 T'_\ell + \ell [\ell]' &\mr{and}& V_{\ell} = \ell^4 - \ell^2 T_\ell + \ell[\ell]^*.
\end{eqnarray*}
As in \S \ref{HeckeE}, the operator $V_{\ell}'$ acts on $\varK$, and the operator $V_{\ell}$ acts on cocycles valued in $\varK_2$. 
 (Strictly speaking, the action of $V_{\ell}$ depends on choice of representatives for the double coset of $\smatrix{ \ell \\ & 1}$, but recall
that we made a particular choice in defining $T_{\ell}$ in \S \ref{actions}.)
For example, one has by \eqref{Vn0} the equality $\newE_n = V_n(0)$ in $H^0(\GL_2(\Z),\varK_0)$. 
Beyond this, these operators act on several closely related groups; for convenience, we summarize some of these actions and their relationships: 

\begin{itemize}
\item[(i)]
As in the discussion prior to Lemma \ref{K2ss}, the operators $V'_{\ell}$ acts directly on $\varKs$, and $V_{\ell}$ acts on  $H^*( \GL_2(\Z), \varKs)$. \item[(ii)] Using double cosets of $\bGamma_0(N)$ inside $\Delta_0(N)$, the operator $V_{\ell}$ acts on 
$H^1(\bGamma_0(N), \varK_2)$ compatibly with the restriction map $H^1(\GL_2(\Z), \varK_2) \rightarrow H^1(\bGamma_0(N), \varK_2)$, 
and similarly with $\varK_2$ replaced by $\varKs$.  
\item[(iii)] As in Proposition 
\ref{modcocycprop}, the operators $V_{\ell}$ and $V'_{\ell}$ both act on  $H^1(\tilde{\Gamma}_0(N), H^2(Y_1(N), \Z'(2))$.
The specialization map $\varKs \rightarrow H^2(Y_1(N), \Z'(2))$
is equivariant for both operators; this is argued just as in the proof of said proposition.
\end{itemize}

The ``distribution relation'' between our cocycles is then as follows. 
\begin{lemma} \label{comparetheta}
	For any prime $\ell \nmid N$, the classes of $ V_{\ell}({}_n\Theta_{N})$ and $ V_n({}_{\ell}\Theta_{N})$ are equal. 
\end{lemma}

\begin{proof}
	 The lemma follows by specialization via $s^*$, noting point (iii) above, from the claim that
$$
 \mbox{the classes of $V_{\ell}({}_n\Theta)$ and $V_n({}_{\ell}\Theta)$ coincide, considered 
 with target $\varKs$. }
$$
	By \eqref{inclusion}, it suffices to prove this instead with target $\varK_2$.  
	Noting point (ii) above, it is moreover sufficient to prove the equality inside $H^1(\GL_2(\Z), \varK_2)$
	rather than $H^1(\bGamma_0(N), \varK_2)$, and then, 
	 by Theorem \ref{heckeE2} it is sufficient to prove  it with the $V$-operators replaced by the $V'$-operators.
	 But just as  in the proof of Theorem \ref{heckeE2}, this is a consequence of the fact that $V'_{\ell} \newE_n = V'_n \newE_{\ell} =
	V'_{\ell} V'_n (0)$.
\end{proof}

Let us state the main theorem of this section.  

\begin{theorem} \label{canoncocyc} 
	There exists a parabolic cocycle 
	$$
 		\Theta_N \colon \bGamma_0(N) \to H^2(Y, \Z'[\tfrac{1}{N}](2))
	$$
	with class uniquely specified by the property that the classes of $V_{\ell}(\Theta_N)$ and ${}_{\ell} \Theta_N$
	are equal for each $\ell$ not dividing $N$. 
	Moreover $T_n$ and $T_n'$ coincide on the class of $\Theta_N$ for all primes $n \nmid N$.
\end{theorem}

The proof requires the following statement about the ring-theoretic structure of the Hecke algebra.

\begin{proposition} \label{Vgenerate2}
 	 Write  $M = H^2(Y, \Z'[\frac{1}{N}](2))$,
	and let $\mb{T}_M$ be the subalgebra of the $\Z'[\frac{1}{N}]$-endomorphism ring of $H^1(\bGamma_0(N),M)$ 
	generated by all $T_{\ell}$ for primes $\ell \nmid N$ and $[d]'$ for 
	$d$ prime to $N$.  Then the operators $V_n$ for primes $n \nmid N$ generate $\mb{T}_M$. 
\end{proposition}

We will prove Proposition \ref{Vgenerate2} in the remainder of this section. Namely, 
we show 	 in  Lemma \ref{Heckefactor}, the algebra $\mb{T}_M$ is a 
	quotient of the $\Z'[\frac{1}{N}]$-Hecke algebra $\mb{T}$ of weight $2$ modular forms for $\Gamma_1(N)$, with
	$T_{\ell}$ mapping to $T_{\ell}$ and $\langle d \rangle$ mapping to $[d]'$, and we will show in 
  Proposition \ref{Vgenerate} that the $V_n$-operators generate $\mb{T}$.  This last statement
  uses the structure of Galois representations attached to level $N$ eigenforms. 
  
We now prove Theorem \ref{canoncocyc} assuming Proposition \ref{Vgenerate2}.  
  
\begin{proof}[Proof of Theorem \ref{canoncocyc}]
	To construct $\Theta_N$, let us choose
	operators $r_i \in \mb{T}_M$ and primes $n_i \nmid N$ for $1 \le i \le t$ for some $t$ such that $\sum_{i=1}^t r_i V_{n_i} = 1$. 
	We then set $\Theta_N = \sum_{i=1}^t r_i ({}_{n_i} \Theta_N )$, a parabolic cocycle.  
	By Lemma \ref{comparetheta}, we see immediately that, as
	cohomology classes, we have
	$$
		V_{\ell}(\Theta_N) = \sum_{i=1}^t r_i V_{n_i} ({}_{\ell} \Theta_N) = {}_{\ell} \Theta_N.
	$$
	Uniqueness follows as, if $\theta$ is a cocycle with 
	$V_{\ell}\theta = {}_{\ell} \Theta_N$ as cohomology classes for all $\ell \nmid N$, then 
	$\theta = \sum_{i=1}^t r_i V_{n_i} \theta = \Theta_N$.
	To show that $T_n$ and $T_n'$ coincide, it is enough by Proposition \ref{Vgenerate2} to show the same for $V_{\ell} T_n$ and $V_{\ell} T_n'$. This follows by 
	Lemma \ref{K2ss} and the commutativity of the two types of Hecke operators on 
	$H^1(\bGamma_0(N),M)$ which is proved just as in Lemma \ref{formality}. 
\end{proof}

\subsubsection{Normalizations of Hecke operators} 
 \label{eichshim}

Our conventions regarding Hecke operators on cocycles differ slightly from standard conventions in the literature
due to issues of left versus right actions. 
We briefly describe the precise relationship, which will be useful in using results about Galois representations. 

 Let $M_2(\Gamma_1(N))$ denote the complex vector space of weight $2$ modular forms for $\Gamma_1(N)$.  
Elements of $M_2(\Gamma_1(N))$ are $\Gamma_1(N)$-invariant
functions on $\mathbb{H}$ for a natural {\em right} action on functions, namely, 
$$
	f|_{\gamma}(z) = (cz+d)^{-2} f(\gamma z) 
$$
for $\gamma = \smatrix{a&b\\c&d} \in \SL_2(\Z)$. Similarly, elements of $H^1(\Gamma_1(N) \backslash \mathbb{H},\C)$ are represented by
$\Gamma_1(N)$-invariant cochains on $\mathbb{H}$,  where we regard $\SL_2(\Z)$ acting 
on the right on cochains in a fashion dual to its obvious left action on chains. 
 Correspondingly, it is natural
to consider {\em right} Hecke operators on these two groups, as is usually done in the literature (cf. Remark \ref{leftright});
when extending the actions above to $M_2(\Z) \cap \GL_2(\Q)$ we introduce an extra factor of $\det(\gamma)$.

Let us consider the Hecke equivariance of the following two maps:
\begin{equation} \label{Basic_seq} M_2(\Gamma_1(N)) \rightarrow H^1(\Gamma_1(N) \backslash \mathbb{H},\C) \rightarrow H^1(\Gamma_1(N), \C),\end{equation}
 where the first (Eichler-Shimura) sends $f$ to $f(z) dz$ and the second sends a cohomology class to the cocycle that, given $\gamma \in
 \Gamma_1(N)$, evaluates the cohomology class on the homology class of an arbitrary path from $z$ to $\gamma z$ (for an arbitrarily chosen $z \in \mathbb{H}$). 
These maps intertwine the right Hecke $T(h)^R$-actions on all three groups defined by a decomposition $\Gamma_1(N) h \Gamma_1(N) = \coprod_{j=1}^t \Gamma_1(N) h_j$.  For instance, $T(h)^R$ is defined on differential forms as $\sum_{j=1}^t h_j^*$, which is clearly compatible with the sum of the actions of the representatives $h_j$ on a modular form.

Now take  $h = \smatrix{ 1 \\ & \ell}$ for $\ell \nmid N$.  
The action of $T(h)^R$
on $M_2(\Gamma_1(N))$ is readily verified to coincide with the Hecke operator denoted $T_{\ell}^*$ by Edixhoven in \cite{edixhoven}.
By Remark \ref{leftright}, the corresponding operator on $H^1(\Gamma_1(N),\C)$  can also be described as $T(h^*)$
(now defined with left cosets),
and this operator $T(h^*) = T(  \smatrix{ \ell \\ & 1}$)
is exactly  our definition of $T_{\ell}$.  

For $\delta \in \Gamma_0(N)$ with lower right-hand entry $d$, the $T(\delta)^R$-action on cusp forms is the diamond operator $\langle d \rangle$ (denoted $\langle d \rangle^*$ in \cite{edixhoven}.) The $T(\delta^{-1}) = T(\delta)^R$-action on a cocycle becomes precomposition with the conjugation $\gamma \mapsto \delta \gamma \delta^{-1}$ (since in this case $t = 1$ and $\gamma_1 = \delta \gamma \delta^{-1}$).

\subsubsection{Comparison of Hecke algebras for $\Gamma_0$ and $\Gamma_1$}

Let us view modules for $(\Z/N\Z)^{\times}$ as having an action of $\Delta_0(N)$ through the quotient map $\Delta_0(N) \to (\Z/N\Z)^{\times}$ under which a matrix is sent to its lower right-hand corner modulo $N$.  Recall that $\bGamma_1(N)$ is the analogue
of $\Gamma_1(N)$ for $\GL_2(\Z)$ defined in \eqref{tildegammadef2}.

\begin{lemma} \label{Shapiro}
	Let $R$ be a commutative ring, and let $\mc{O} = R[(\Z/N\Z)^{\times}]$. Then Shapiro's lemma defines an isomorphism   
	$$
		H^1(\bGamma_0(N),\mc{O}) \xrightarrow{\sim} H^1(\bGamma_1(N),R)
	$$
	that is compatible with the action  of Hecke operators $T(g)$ as in \eqref{Heckecocyc} with $g \in \Delta_1(N)$.  Moreover, this map is $(\Z/N\Z)^{\times}$-equivariant
	for the action of $d \in (\Z/N\Z)^{\times}$ on the right by precomposition by $\gamma \mapsto \delta\gamma\delta^{-1}$ for
	any $\delta \in \bGamma_0(N)$ with image $d$ in $(\Z/N\Z)^{\times}$.
\end{lemma}

\begin{proof}
	Let us denote the image of a cocycle $\theta \colon \bGamma_0(N) \rightarrow R[(\Z/N\Z)^{\times}]$ under the Shapiro 
	isomorphism by $\overline{\theta}$: it is obtained by restriction of cocycles together with the map 
	$\phi \colon R[(\Z/N\Z)^{\times}] \rightarrow R$ that takes the coefficient of the identity element.    
	For $g \in \Delta_0(N)$ and $\gamma \in \bGamma_0(N)$, equation \eqref{Heckecocyc} states that
 	$T(g)\theta(\gamma) =  \sum_{j=1}^t g_{\sigma(j)}\theta(\gamma_j)$, recalling the notation of \S \ref{actions}.
	If in fact $g \in \Delta_1(N)$ (as defined in \eqref{Delta1def}) and $\gamma \in \bGamma_1(N)$, then we may choose
	the representatives $g_j$ to also belong to $\Delta_1(N)$, in which case the $\gamma_j$ belong to $\bGamma_1(N)$.  
	We then have
	$$
		\overline{T(g) \theta}(\gamma) =  \sum_{j=1}^t \phi (g_{\sigma(j)} \theta(\gamma_j)) =  \sum_{j=1}^t g_{\sigma(j)} \phi(\theta(\gamma_j) ) 
		= \sum_{j=1}^t g_{\sigma(j)} \bar{\theta}(\gamma_j) = (T(g) \overline{\theta})(\gamma).
	$$ 
	In particular, taking $g=\smatrix{\ell & 0 \\0&1}$ exhibits  equivariance for $T_{\ell}$.
	
	Finally, take $d \in (\Z/N\Z)^{\times}$
	and a representative $\delta \in \bGamma_0(N)$. Then for $\gamma \in \bGamma_1(N)$,	
	we get 
	$$
		(\delta \cdot \bar{\theta})(\gamma) = \phi(\theta(\delta \gamma \delta^{-1}))  = \phi(\delta  \theta(\gamma))
	$$ 
	so the Shapiro isomorphism is $(\Z/N\Z)^{\times}$-equivariant.
\end{proof}

Now suppose $6 \in R^{\times}$, and
let $\mb{T}$ denote the $R$-Hecke algebra of weight $2$ modular forms for $\Gamma_1(N)$ generated
by prime-to-level Hecke operators $T_{\ell}$ and diamond operators $\langle \ell \rangle$.   
 
\begin{lemma} \label{Heckefactor}	
	Let $M$ be a $R[(\Z/N\Z)^{\times}]$-module, where $6 \in R^{\times}$. We equip $M$ with the action of  
	$\Delta_0(N)$  via the surjection $\Delta_0(N) \twoheadrightarrow (\Z/N\Z)^{\times}$.
	 Let $\mathbb{T}_M$ be the $R$-algebra of endomorphisms of the cohomology group $H^1(\bGamma_0(N),M)$ generated by $T_n$ for primes 
	$n \nmid N$ and by the elements of $(\Z/N\Z)^{\times}$. Then there is a surjection
	$$ \mathbb{T} \rightarrow \mathbb{T}_M$$
	 carrying $T_n$ to $T_n$ and the diamond operator
	$\langle d \rangle$ to the action of $d \in (\Z/N\Z)^{\times}$.
\end{lemma}
 
\begin{proof}
	 Write $\mathcal{O}$
	for the group algebra of $(\Z/N\Z)^{\times}$ over $R$, which we view as a quotient of the $R$-monoid 
	algebra of $\Delta_0(N)$ through 
	 the lower right-hand corner map.   
	Then $M$ is isomorphic to a quotient of $\mathcal{O}^{\oplus J}$ for some indexing set $J$.   
	Since $2$ is invertible in $R$, we have   	$$
		H^1(\bGamma_1(N),M) \cong H^1(\Gamma_1(N),M)_+
	$$
	(see the proof of Proposition \ref{maphomology}), so $\mb{T}$ acts on $H^1(\bGamma_1(N),M)$.
	Consider the composition
	\begin{equation} \label{mc_ren} 
		H^1(\bGamma_1(N), R)^{\oplus J} \xrightarrow{\sim} 
		H^1(\bGamma_0(N), \mathcal{O})^{\bigoplus J} \to H^1(\bGamma_0(N), M),
	\end{equation}
	where the first map comes from Shapiro's lemma 
	as in Lemma \ref{Shapiro}, and the cokernel of the last map injects into
	$H^2(\bGamma_0(N),A)$, for $A = \ker(\mc{O}^{\bigoplus J} \to M)$. We claim that $H^2(\bGamma_0(N),A)$ is zero, so that
	the composition in \eqref{mc_ren} is surjective. Since this composition is compatible with the
	action of Hecke and diamond operators as in Lemma \ref{Shapiro}, 
	we will then have the lemma.
	
	To see the claim, note that the restriction map 
	$$
		H^2(\bGamma_0(N),A) \to H^2(\Gamma_0(N) \cap \Gamma(4),A)
	$$ 
	is injective since the index $h = [\bGamma_0(N):\Gamma_0(N) \cap \Gamma(4)]$ is
	invertible in $R$, and the composition of restriction and corestriction is multiplication by $h$.
	The target of restriction is a second cohomology group of an open $2$-manifold, hence trivial.
\end{proof}

\subsubsection{The $V_n$-operators generate}

They following result implies Proposition \ref{Vgenerate2}, in view of the results of the prior subsection. 

\begin{proposition} \label{Vgenerate}
	 Let $\mathbb{T}$ be the Hecke ring for $\Gamma_1(N)$ with $\Z'[\frac{1}{N}]$ coefficients. 
	 The operators $V_n$ for primes $n \nmid N$ generate the unit ideal of $\mb{T}$.
\end{proposition}

\begin{proof}
	Let $\mathfrak{v}$ be the ideal of $\mathbb{T}$ that the operators $V_n$ for $n \nmid N$ generate. 
	Suppose by way of contradiction that $\mf{v} \neq \mb{T}$.  Then $\mathbb{T}/\mathfrak{v}$
	is a ring that is finite over $\Z'[\frac{1}{N}]$ and admits a nontrivial homomorphism to a field 
	$F$ that is algebraically closed of finite characteristic $p \nmid N$. 
 
 	In this situation, there exists an associated continuous, semisimple Galois representation
     	$$ 
    		\rho \colon G_{\Q} \to \GL_2(F)
    	$$
     	such that the trace of a 
	Frobenius element $\varphi_{\ell}$ at any prime $\ell \nmid Np$ 
    	coincides with the image of $T_{\ell}$ in $F$, and the determinant of $\varphi_{\ell}$
	is given by the image of the diamond operator $\langle \ell \rangle$ in $F^{\times}$, which we'll denote by the same symbol.
	Since $V_{\ell}$ maps to zero in $F$, this implies that
     	$$  
		 \Tr \rho(\varphi_{\ell}) =\ell^2  +  \ell^{-1} \langle \ell \rangle \in F
	$$
	for $\ell \nmid Np$.
	
       	By \v{C}ebotarev density, $\rho$ is isomorphic to the direct sum $\omega^2 \oplus \omega^{-1}\nu$,
    	where $\omega$ denotes the mod $p$ cyclotomic character
    	and $\nu \colon G_{\Q} \rightarrow F^{\times}$ 
     	is the composition of the cyclotomic character $G_{\Q} \rightarrow (\Z/N\Z)^{\times}$
    	and the diamond operator map $(\Z/N\Z)^{\times} \rightarrow F^{\times}$.
	Restricted to the inertia group above $p$, we get
	$$ \rho|_{I_p} \simeq \omega^2 \oplus \omega^{-1},$$
	which contradicts well-known properties of the Galois representations attached to weight $2$ eigenforms of level $N$:
	that is, the restriction of $\rho$ to any inertia group $I_p$ at $p$ 
	must be $\omega \oplus 1$  or the sum of two tame characters. It is enough to verify this
	separately for Eisenstein series and cusp forms; in the Eisenstein only the former case occurs,
	and in the cuspidal case   the two possibilities are distinguished by whether the image of $T_p$ in $F$
	is zero or nonzero (see ~\cite[Theorems 2.5-2.6]{edixhoven}).	 
\end{proof}

  \begin{remark} \label{Qcoeff}
	If we are willing to work with $\Q$-coefficients in place of coefficients in  $\Z'[\frac{1}{N}]$, then  Proposition \ref{Vgenerate}
	has a much simpler proof. Indeed each $V_{\ell}$ with $\ell \nmid N$ 
	is itself a unit in the Hecke algebra acting on group cohomology.  The key point is that
	the $T_{\ell}$-eigenvalues of any weight $2$ eigenform for $\Gamma_1(N)$ have complex absolute value at most 
	$\ell+1$, and the eigenvalues of diamond operators are roots of unity, so $V_{\ell}$ has eigenvalues of complex 
	absolute value at least $\ell^3-1-\ell(\ell+1) > 0$.  
\end{remark}

\subsection{Explicit formula for the universal cocycle} \label{expformuniv}

We cannot quite write down an explicit formula for a cocycle in the universal class of the previous section
because of our lack of understanding of the motivic cohomology group $H^2(Y, \Q(2))$. 
However, we can at least do it modulo a subgroup $\mc{V}$ which can be seen to vanish
under any standard regulator map.

\subsubsection{The explicit formula, in brief}

For a prime $n \nmid N$, let $\mc{V}_n$ denote the kernel of $V'_n$ on $H^2(Y,\Q(2))$, and let $\mc{V} = \bigcap_{n \nmid N} \mc{V}_n$.
Remark \ref{Qcoeff} implies that
the group $\mc{V}$ maps to 
zero in any quotient of $H^2(Y, \Q(2))$
that factors through the action of the Hecke algebra on $H^1(\Gamma_1(N), \Q)$.

\begin{proposition} \label{expformmod}
	The class of $\Theta_N$ modulo $\mc{V}$ equals the class of the  cocycle
	\begin{equation} \label{WD}
		\tilde{\Gamma}_0(N) \to H^2(Y,\Q(2))/\mc{V}, \quad \gamma \mapsto 
		\sum_{i=1}^k g_{\frac{d_i}{N}}  \cup g_{\frac{-d_{i-1}}{N}} \bmod \mc{V} 
	\end{equation}
 	for $(b_i,d_i)_{i=0}^k$ any $N$-connecting sequence for $\gamma$, where 
	$g_{\frac{a}{N}}$ for $a$ prime to $N$ is the standard Siegel unit on $Y$ (see \S \ref{reviewSiegel}).
\end{proposition}
	Implicit in the statement is the assertion that the right-hand side of \eqref{WD} is independent of the choice of connecting sequence and
	defines a cocycle.\footnote{It is very likely that the proposition remains true without taking the quotient by $\mc{V}$, but we do not know how to prove it.} 
We explain the proof modulo certain explicit computations with Siegel units that are carried out in the rest of the section. 

\begin{proof} 
	We will prove in Lemma \ref{explicit pullback} that 
 	for $\gamma = \smatrix{a & b \\ c & d} \in \SL_2(\Z)$ 
	with both $c$ and $d$ relatively prime to $N$, we have 
	\begin{equation} \label{key}
		s^* \langle \gamma \rangle_n = V_n' (g_{\frac{c}{N}} \cup  g_{\frac{d}{N}}).
	\end{equation}
 	for each $n \nmid N$.  For a given $N$-connecting sequence $(b_i,d_i)_{i=0}^k$ for $\gamma$, let us set
	$$
		f_{\gamma} = \sum_{i=1}^k   g_{\frac{d_i}{N}} \cup g_{\frac{-d_{i-1}}{N}} \in H^2(Y,\Q(2))
	$$
	with the understanding that this depends on the connecting sequence. 
	From Proposition \ref{expformE2}(c) and \eqref{nNThetadef}, we know that
	$$
 		{}_n \Theta_{N, \gamma} = s^* \sum_{i=1}^k  \left\langle  \smatrix{b_i & -b_{i-1} \\ d_i & -d_{i-1}  }\right\rangle_n \stackrel{\eqref{key}}{=} \sum_{i=1}^k V_n' (g_{\frac{d_i}{N}} \cup  g_{\frac{-d_{i-1}}{N}}).
	$$ 
 	So, by \eqref{key}, 
	we have that
	$${}_n\Theta_{N,\gamma} =  V_n' f_{\gamma}.$$
	This equality uniquely determines $f_{\gamma}$ as an element of $H^2(Y,\Q(2))/\mc{V}_n$.
	From this and the fact that ${}_n \Theta_N$
	is a cocycle, we see that the quantity $f_{\gamma} \bmod \mc{V}_n$
	is independent of choice of connecting sequence, and $\gamma \mapsto f_{\gamma} \bmod \mc{V}_n$ is a cocycle. 
	But then, the latter two facts are true modulo $\mc{V} = \bigcap_n \mc{V}_n$ as well.

	By Theorem \ref{canoncocyc}, the cocycle $V'_n(\Theta_N)$ is cohomologous to ${}_n\Theta_N$.
	In particular, the class of $\gamma \mapsto \Theta_{N,\gamma} - f_{\gamma}$ 
	lies in the kernel of all $V_n'$ acting on $H^1(\bGamma_0(N),M)$ with $M = H^2(Y,\Q(2))/\mc{V}$.
	To see that this common kernel is zero, consider the injection $\iota \colon M \hookrightarrow \bigoplus_{n \nmid N} M$
	induced by the collection of operators $V'_n$. We must show that the map
	$$
		H^1(\bGamma_0(N),M) \to \bigoplus_{n \nmid N} H^1(\bGamma_0(N),M)
	$$
	induced by $\iota$ is injective. 
	This follows from the surjectivity of the map
	\begin{equation} \label{H0surj}
		\bigoplus_{n \nmid N} H^0(\bGamma_0(N),M) \to H^0(\bGamma_0(N),\coker \iota)
	\end{equation}
	which in turn is a consequence of the fact that $\bGamma_0(N)$-action on the $\Q$-vector space $M$ factors through 
	the finite group $(\Z/N\Z)^{\times}$. 
\end{proof}

Aside from the change of modular curve, the following is a corollary of Proposition \ref{expformmod} and its proof.
 
\begin{proposition} \label{Gamma1X}
 	The cocycle $\Theta_N$ restricts to a cocycle
	$$
		\Theta_N \colon \bGamma_1(N) \to H^2(X_1(N),\Z'[\tfrac{1}{N}](2)),
	$$
	satisfying
	$$
 		\Theta_{N, \gamma} \equiv \sum_{i=1}^k g_{\frac{d_i}{N}} \cup g_{\frac{-d_{i-1}}{N}} \bmod \mc{V}
	$$
	for $\gamma \in \Gamma_1(N)$.
 \end{proposition}
 
 \begin{proof}
 	To see that we can work with $\Z'[\frac{1}{N}]$-coefficients, note that the only place where we may need to 
	invert further primes (i.e., those dividing $\varphi(N)$) in the proof of Proposition \ref{expformmod} 
	is for the surjectivity of the map in \eqref{H0surj}, but if we replace $\bGamma_0(N)$ by $\bGamma_1(N)$,
	then this need is alleviated as $\bGamma_1(N)$ acts trivially on $H^2(Y,2)$.
 
 	The second claim is immediate from Proposition \ref{expformmod}, since both cocycles restrict to homomorphisms
	on $\Gamma_1(N)$, so are equal (modulo $\mc{V}$).
 	Set $X = X_1(N)$ (over $\Q$) and $C = X-Y$. We have an exact Gysin sequence
	$$
		0 \to H^2(X,\Z'[\tfrac{1}{N}](2)) \to H^2(Y,\Z'[\tfrac{1}{N}](2)) \to H^1(C,\Z'[\tfrac{1}{N}](1)) \to 0
	$$
	Each $V_n$ for $n$ not dividing $N$ acts on    	$
	H^1(C,\Z'[\tfrac{1}{N}](1)) \cong \mc{O}_C^{\times} \otimes_{\Z} \Z'[\tfrac{1}{N}]$.  
	Since this is torsion-free as a $\Z'[\tfrac{1}{N}]$-module, the argument of Remark \ref{Qcoeff} can be applied
	to show that no nonzero element of $H^1(C,\Z'[\tfrac{1}{N}](1)) $ is killed by all $V_n$ with $n \nmid N$. 
	From this, we see that $\mc{V} \subseteq H^2(X,\Z'[\frac{1}{N}](2))$.

	It therefore suffices
	to show that $\sum_{i=1}^k g_{\frac{d_i}{N}} \cup g_{\frac{-d_{i-1}}{N}}$ has trivial residue in 
	$\mc{O}_C^{\times} \otimes_{\Z} \Z'[\frac{1}{N}]$. 
	It follows from \cite[Lemma 3.3.12]{fk} that the tame symbol of this sum at the cusp 
	$\infty \colon \Spec \Q(\mu_N) \to X_1(N)$ has image in $\Q(\mu_N)^{\times} \otimes_{\Z} \Z'[\frac{1}{N}]$ equal to
	$$
		\prod_{i=1}^k \left(\frac{1-\zeta_N^{d_i}}{1-\zeta_N^{d_{i-1}}}\right)^{1/12} = 1,
	$$
	and similarly for the other cusps over the infinity cusp of $X_0(N)$.
	At the other, non-infinity cusps, the same lemma tells us that the residues of the individual terms 
	$g_{\frac{d_i}{N}} \cup g_{\frac{-d_{i-1}}{N}}$ are trivial.
\end{proof}

\subsubsection{Review of Siegel units} \label{reviewSiegel}

Let us consider units on the modular curve $Y(M)$ for $M \ge 3$, which is the moduli space of triples $(E,P,Q)$ with $E$
an elliptic curve and $(P,Q)$ an ordered basis of $E[M]$.  The universal
elliptic curve $\mc{E}(M)$ has two canonical order $M$ sections $\iota_{M,1}, \iota_{M,2} \colon Y(M) \to \mc{E}(M)$
corresponding to $P$ and $Q$.  For $(c,d) \in \Z^2-M\Z^2$ and $m$ prime to $\frac{M}{(c,M)} \cdot \frac{M}{(d,M)}$, let 
$$ 
	{}_m g_{\frac{c}{M},\frac{d}{M}} = (c\iota_{M,1}+d\iota_{M,2})^*({}_m\theta) \in \mc{O}_{Y(M)}^{\times} \otimes \Z[\tfrac{1}{6}],
$$
where ${}_m \theta \in \mc{E}(M)^{\times} \otimes \Z[\frac{1}{6}]$ is the theta-function defined analogously to \eqref{thetasdef}: it has zeroes of multiplicity $1$ along nonzero $m$-torsion points, and a pole of order $(m^2-1)$ at the identity section. 

Next let  
$$
	g_{\frac{c}{M},\frac{d}{M}} = {}_m g_{\frac{c}{M},\frac{d}{M}} \otimes (m^2-1)^{-1} \in \mc{O}_Y^{\times} \otimes_{\Z} \Q
$$
for any $m \equiv 1 \bmod M$ and prime to $6$, independent of the choice.  
(In fact, we may define $g_{\frac{c}{M},\frac{d}{M}}$ as an element of $\mc{O}_Y^{\times} \otimes_{\Z} \Z[\frac{1}{6M}]$.)
Then 
$$
	{}_m g_{\frac{c}{M},\frac{d}{M}} = g_{\frac{c}{M},\frac{d}{M}}^{m^2} \cdot g_{\frac{mc}{M},\frac{md}{M}}^{-1}.
$$ 
For any $m \ge 1$ and $(c,d) \in \Z^2 - M\Z^2$, the Siegel units satisfy the distribution relation 
$$
	\prod_{i = 0}^{m-1} \prod_{j=0}^{m-1} g_{\frac{c}{Mm} + \frac{i}{m}, \frac{d}{Mm} + \frac{j}{m}} =
	g_{\frac{c}{M},\frac{d}{M}}.
$$

The Siegel units ${}_m g_{0,\frac{d}{M}}$ and $g_{0,\frac{d}{M}}$ are units rationally on $Y_1(M)$.  We denote them more simply
by ${}_m g_{\frac{d}{M}}$ and $g_{\frac{d}{M}}$, respectively.

\subsubsection{Some computations with Siegel units}

Our goal here is to prove \eqref{key}.
\begin{lemma} \label{thetapullback}
	Let us consider ${}_n\theta'$ defined analogously to \eqref{thetasdef} as a rational function on the universal elliptic curve 
	$\mc{E}_n$ over the modular curve $Y_1(Nn)$.  
 	For $d \in (\Z/N\Z)^{\times}$, we have
	$$
		(d\iota_N)^*({}_n \theta') 
		=  g_{\frac{nd}{N}}^{-1} \cdot \prod_{j=0}^{n-1} g_{\frac{d}{N}+\frac{j}{n}}^n 
		\in 
		\mc{O}_{Y'}^{\times} \otimes_{\Z} \Q.
	$$ 
\end{lemma}

\begin{proof}
	For $\mc{E}_n \rightarrow Y_1(Nn)$, we have canonical order $N$ and order $n$ sections $\iota_N$ and $\iota_n$, respectively.
	Let $\phi \colon \mc{E}_n\to \mc{E}_n$ be translation by $\iota_n$, i.e., 
	given on the fiber over $(E,P,Q)$ with $P$ of order $N$ and $Q$ of order $n$ by $\phi(x) = x + Q$.  Then
	$$
		n^2\mc{K} - nE[n] = \sum_{j=0}^{n-1} (j\phi)^* (n^2(0)- E[n]),
	$$
	so 
	$$
		({}_n\theta')^n = \sum_{j=0}^{n-1} (j\phi)^* {}_n\theta,
	$$
	where we view ${}_n \theta$ as the theta-function with divisor $n^2(0)-E[n]$ on $\mc{E}_n$.
	Since, for $d \in (\Z/N\Z)^{\times}$, the section $d\iota_N+j\iota_n$ has order divisible by $N$, the pullback $(d\iota_N+j\iota_n)^*{}_n\theta$ is well-defined.
	We then see that
	$$
		(d\iota_N)^* (j\phi)^* {}_n \theta = (d\iota_N+j\iota_n)^* {}_n\theta
		= {}_n g_{\frac{d}{N}+\frac{j}{n}} = g_{\frac{nd}{N}}^{-1} \cdot g_{\frac{d}{N}+\frac{j}{n}}^{n^2}.
	$$
	Taking the product over $0 \le j \le n-1$ and the $n$th root gives the result.
\end{proof}

\begin{lemma} \label{explicit pullback}
	For $\gamma = \smatrix{a & b \\ c & d} \in \SL_2(\Z)$ 
	with both $c$ and $d$ relatively prime to $N$, we have 
	$$
		s^* \langle \gamma \rangle_n = V_n' (g_{\frac{c}{N}} \cup  g_{\frac{d}{N}}) \in H^2(Y,\Q(2)).
	$$
\end{lemma}
\begin{proof}
 	Note that $s^* \langle \gamma \rangle_n$ is the pullback
	of $\langle \smatrix{1&0\\0&1} \rangle_n$ by the section $(c\iota_N,d\iota_N)$. 
	We have
	\begin{equation} \label{thetatheta} 
		(c\iota_N,d\iota_N)^*({}_n\theta \boxtimes {}_n \theta) =  \frac{g_{\frac{c}{N}}^{n^2}}{g_{\frac{nc}{N}}} \cup
		\frac{g_{\frac{d}{N}}^{n^2}}{g_{\frac{nd}{N}}}.
	\end{equation}
	For the $N$-torsion section of $\iota_N \colon Y' \to \mc{E}'$ 
	(with $Y' = Y_n$ as in \S \ref{symboldef}), 
	Lemma \ref{thetapullback} tells us that
	\begin{equation} \label{theta'theta'}
		(c\iota_N,d\iota_N)^*({}_n \theta' \boxtimes {}_n \theta')
		= \frac{\prod_{i=0}^{n-1}g_{\frac{c}{N}+\frac{i}{n}}^n}{g_{\frac{nc}{N}}} \cup  
		\frac{\prod_{j=0}^{n-1} g_{\frac{d}{N}+\frac{j}{n}}^n}{g_{\frac{nd}{N}}}.
	\end{equation}
	The individual functions here are defined on $Y_1(Nn)$, but the product is defined on $Y'$. 

	Given that we have a cartesian diagram
	$$
		\begin{tikzcd}[column sep = large]
			Y' \arrow{r}{(c\iota_N,d\iota_N)} \arrow{d} & (\mc{E}')^2 \arrow{d} \\
			Y \arrow{r}{(c\iota_N,d\iota_N)} & \mc{E}^2, 
		\end{tikzcd}
	$$
	Lemma \ref{pullbackdiagram} implies that
	the norms for $Y' \to Y$ and $(\mc{E}')^2 \to \mc{E}^2$ commute with pullback by the 
	$N$-torsion section $(c\iota_N,d\iota_N)$.
	Recalling now that   
	$\langle \smatrix{1&0\\0&1} \rangle_n =  {}_n\theta \boxtimes {}_n\theta - \Norm({}_n\theta' \boxtimes {}_n\theta')$,
	we then obtain
	from \eqref{thetatheta} and \eqref{theta'theta'} that  	
	\begin{equation} \label{combined}
		(c\iota_N,d\iota_N)^* \langle \smatrix{1&0\\0&1} \rangle_n
		= \frac{g_{\frac{c}{N}}^{n^2}}{g_{\frac{nc}{N}}} \cup \frac{g_{\frac{d}{N}}^{n^2}}{g_{\frac{nd}{N}}}
		- \left( \sum_{\langle(\alpha,\beta)\rangle} 
		\frac{\prod_{i=0}^{n-1} g_{\frac{i\alpha}{n},\frac{c}{N}+\frac{i\beta}{n}}^n}
	 	{g_{\frac{nc}{N}}} \cup \frac{\prod_{j=0}^{n-1} g_{\frac{j\alpha}{n},\frac{d}{N}+\frac{j\beta}{n}}^n}
	 	{g_{\frac{nd}{N}}}  \right),
	\end{equation}
	where the first sum runs over chosen generators of the $n+1$ cyclic subgroups of order $n$ in $(\Z/n\Z)^2$.
	Note that in \eqref{combined}, we work with the cover $Y''$ of $Y$ obtained
	by the additional data of a full level $n$ structure. The group $H^2(Y,\Q(2))$ injects into $H^2(Y'',\Q(2))$ 
	under pullback, since
	$Y'' \to Y$ is finite and these groups are $\Q$-vector spaces. We can then compute the norm 
	 $\Norm({}_n\theta' \boxtimes {}_n\theta')$ by taking a sum over the actions of 
	coset representatives for $\GL_2(\Z/n\Z)$ modulo the upper triangular subgroup.
	
	We now analyze the terms of \eqref{combined}. In the second term, we have the following. 
	\begin{itemize}
	\item The numerators give $n^2 T'_n(g_{\frac{c}{N}} \cup g_{\frac{d}{N}})$. Indeed, 
	by definition, $T_n' (g_{\frac{c}{N}} \cup g_{\frac{d}{N}})$
 	is obtained by pulling back $g_{\frac{c}{N}} \cup g_{\frac{d}{N}}$ to $Y'$ along $\psi \colon (E,P,K) \mapsto (E/K,P+K)$, 
	and then taking the norm $\phi_*$   along $Y'/Y$. The pullback
	$\psi^*g_{\frac{c}{N}}$ is given by $\prod_{i=0}^{n-1} g_{\frac{c}{N} + \frac{i}{n}}$, and
	the norm is as before.
	\item 
	The cross terms are 
	\begin{eqnarray*}
		-n^2g_{\frac{nc}{N}} \cup g_{\frac{d}{N}} - ng_{\frac{nc}{N}} \cup g_{\frac{nd}{N}} &\mr{and}&
		-n^2g_{\frac{c}{N}} \cup g_{\frac{nd}{N}} - ng_{\frac{nc}{N}} \cup g_{\frac{nd}{N}}
	\end{eqnarray*}
	by the distribution relation. 
	\item
	The denominators contribute $(n+1) g_{\frac{nc}{N}} \cup g_{\frac{nd}{N}}$.  	\end{itemize}
	Subtracting this from the first term 
	and noting that $[n]'(g_{\frac{c}{N}}\cup g_{\frac{d}{N}}) = g_{\frac{nc}{N}} \cup g_{\frac{nd}{N}}$,
	we obtain
	$$
		s^*\langle \gamma \rangle_n = (n^4-n^2T'_n+n[n]') (g_{\frac{c}{N}} \cup g_{\frac{d}{N}})
		= V'_n (g_{\frac{c}{N}} \cup g_{\frac{d}{N}}).
	$$
\end{proof}

\subsection{Maps on the homology of $X_1(N)$} \label{homologymod}

We conclude by comparing our cocycle $\Theta_N$ to related ``zeta maps'' 
on the homology of modular curves.

\subsubsection{Zeta maps with $\Z'[\frac{1}{N}]$-coefficients}
 
Since $\Theta_N$ restricts to a homomorphism on $\tilde{\Gamma}_1(N)$ which is trivial on parabolic subgroups, we have the 
following analogue of Proposition \ref{maphomology}.  We note that $H^2(X_1(N),2)$ is preserved by the Hecke and diamond operators on $H^2(Y_1(N),2)$.

\begin{theorem} \label{zeta_map} 
	The map
	$$
		z_N \colon H_1(X_1(N),\Z')_+ \to H^2(X_1(N),\Z'[\tfrac{1}{N}](2)) 
	$$
	sending the image of $\vec\gamma = \{ 0 \to \gamma \cdot 0 \}$ to $\Theta_{N,\gamma}$ for all $\gamma \in \Gamma_1(N)$
	is a Hecke-equivariant homomorphism in the sense that $z_N(T_{\ell} \vec\gamma) = T'_{\ell} \cdot z_N(\vec\gamma)$ for primes 
	$\ell \nmid N$ and $z_N(\langle d \rangle \vec\gamma) = [d]' \cdot z_N(\vec\gamma)$ for $d \in (\Z/N\Z)^{\times}$.
\end{theorem}

\begin{proof}
	The existence of a map to $H^2(X_1(N),\Z'[\frac{1}{N}](2))$ follows from Proposition \ref{Gamma1X} just as in Proposition
	\ref{maphomology}, since the induced $\tilde{\Gamma}_1(N)$-action on the latter cohomology group is trivial. 
	The  Hecke equivariance follows as in the argument of Theorem \ref{varpi}.  
\end{proof}

In \cite[Proposition 2.16]{goncharov}, Goncharov outlined a construction of an analogue of $z_N$ for $Y(N)$ via a map from a complex computing the cohomology of the modular curve $Y(N)(\C)$ to a certain ``Euler complex'' involving a Bloch group. 
In a recent preprint, Brunault \cite[Theorem 4.3]{brunault-K4} gives what amounts to an explicit construction of a well-defined
homomorphism
$$
	z_N^{\circ} \colon 
	H_1(X_1(N),C^{\circ}_1(N),\Z) \to H^2(Y_1(N),\Z[\tfrac{1}{6N}](2)), \quad [u:v]_N \mapsto g_{\frac{u}{N}} \cup g_{\frac{v}{N}},
$$
directly verifying that Steinberg symbols of Siegel units satisfy the Manin relations, improving earlier work in \cite{brunault}.

The map $z_N^{\circ}$ agrees on $H_1(X_1(N),\Z)$ with the restriction to $\Gamma_1(N)$ of the explicit map $f$ of
the proof of Proposition \ref{expformmod} (after taking its image in $K$-theory), showing it to be a homomorphism without the need to reduce modulo $\mc{V}$.  (Note that we invert $5$ in order to obtain Hecke equivariance for the $5$th Hecke operator, not to have a well-defined cocycle.) 
However, that still leaves a needed argument to show $f$ agrees with $\Theta_N$ on $\Gamma_1(N)$ to deduce the Hecke equivariance of $f$.

\subsubsection{Ordinary zeta maps with $\zp$-coefficients}

Fix a prime $p \ge 5$ dividing $N$.
Let $\mb{T}^*_N$ denote the full adjoint weight $2$ Hecke algebra for $\Gamma_1(N)$ over $\zp$ (see \eqref{adjhecke}). 
We also view it as acting via adjoint operators on 
$H^2_{\et}(Y_1(N),\qp(2))$, and let us use a superscript $\ord$ to denote the $U_p^*$-ordinary part for this action.
This $U_p^*$-ordinary part is canonically a direct summand via application of Hida's idempotent in $\mb{T}^*_N$.

In \cite[Theorem 3.3.9]{fk} (see also Lemma 5.2.5 therein), Fukaya and Kato construct the following Hecke-equivariant zeta map to the $U_p^*$-ordinary part of cohomology (or more precisely, the negative of this map precomposed with an Atkin-Lehner involution).

\begin{theorem}[Fukaya-Kato] \label{FKzeta}
	There is a $\mb{T}_N^*$-equivariant homomorphism
	$$
		z_{N,\et}^{\ord} \colon H_1(X_1(N),C_1^{\circ}(N),\zp) \to H^2_{\et}(Y_1(N),\qp(2))^{\ord}, \quad
		[u:v]_N \mapsto g_{\frac{u}{N}} \cup g_{\frac{v}{N}},
	$$
	where we identify the cup product of Siegel units with its $U_p^*$-ordinary projection.
\end{theorem}
	
The proof of Theorem \ref{FKzeta} is quite involved but in particular uses a $p$-adic regulator computation of the values of a related map taken up the cyclotomic tower, which are norm-compatible sequences of Beilinson-Kato elements in Iwasawa cohomology.\footnote{They in fact obtain a map to the subgroup given by the cohomology of the integral model $Y_1(N)_{/\Z[\frac{1}{N}]}$. It is also possible to see our zeta maps are similarly valued in the motivic cohomology of $X_1(N)_{/\Z[\frac{1}{N}]}$, for instance using explicit formulas for ${}_n \Theta_N$.}

The restriction of the ordinary zeta map $z_{N,\et}^{\ord}$ 
to $H_1(X_1(N),\zp)$ is the \'etale realization of the zeta map $z_N$ of Theorem \ref{zeta_map}. That is, the explicit formula for $z_N(\vec\gamma) = \Theta_{N,\gamma}$ given in Theorem \ref{expformE2} agrees in its \'etale realization with that of $z_{N,\et}^{\ord}$.
To see this, note that the group $\mc{V}$ providing the ambiguity in the explicit formula for $\Theta_N$ of Theorem \ref{expformE2} vanishes in the \'etale realization, since the prime-to-level Hecke operators on $H^2_{\et}(Y_1(N),\qp(2))$ factor through the $\zp$-Hecke algebra of weight $2$ modular forms, where each $V'_{\ell}$ has trivial kernel (see Remark \ref{Qcoeff}).

\begin{remark}
	The operators $T'_{\ell}$ on $H^2(Y,2)$ defined in \S \ref{T'Hecke} arise from the composition of the operators $[\ell]'$ and the
	dual (or adjoint) Hecke operators $T_{\ell}^*$ (or $T(\ell)^*$) in \cite[1.2.3]{fk}. On \'etale cohomology, where we know that their
	actions factor through the usual weight $2$ Hecke algebra, we have that $T'_{\ell}$ acts as $T_{\ell} = \langle \ell \rangle T_{\ell}^*$.
	So, the Hecke-equivariance at prime-to-level operators in Theorem \ref{zeta_map} matches that of Theorem \ref{FKzeta}.	
\end{remark}

\begin{remark} 
	Jun Wang \cite[\S 5.1]{wang} (see also \cite[Theorem 3.7]{lw}) proved the analogue of Theorem \ref{FKzeta} 
	for $p \nmid N$, in which case one 
	need not take ordinary parts. His map is shown to take values in the quotient of 
	$H^2_{\et}(Y_1(N),\zp(2))$ by the finite subgroup $H^2_{\et}(\Z[\frac{1}{Np}],\zp(2))$.  The $p$-adic \'etale realization of our map
	$z_N$ takes image in $H^2_{\et}(Y_1(N),\zp(2))$ and induces Wang's map in the quotient.
\end{remark}

In \cite{fks}, it is shown that if $p \nmid \varphi(N)$, then there exists an integral version of $z_N$ to the primitive part of $H^2_{\et}(Y_1(N),\zp(2))^{\ord}$ for the action of $(\Z/N\Z)^{\times}$
by diamond operators, after excluding the $\omega^{-2}$-eigenspace
for $(\Z/p\Z)^{\times}$.  Let us describe a motivic version of this, without some of these assumptions.

Let  $\iota \colon (\Z/p\Z)^{\times} \to (\Z/N\Z)^{\times}$ denote the canonical map splitting reduction modulo $p$.  We have an idempotent
$$
	\varepsilon = 1- \frac{1}{p-1} \sum_{a = 1}^{p-1} \omega^2(a)\iota(a) \in \zp[(\Z/N\Z)^{\times}].
$$
This idempotent applied to $H^2(X_1(N),\zp(2))$ serves to remove the 
$\omega^{-2}$-eigenspace of $(\Z/p\Z)^{\times}$, where $a \in (\Z/p\Z)^{\times}$ acts as $[\iota(a)]'$.

\begin{proposition} \label{p-adic_zeta}
	There exists a unique homomorphism
	$$
		z_N^{\ord} \colon H_1(X_1(N),\zp)_+ \to \varepsilon \cdot H^2(X_1(N),\zp(2))
	$$
	which 
	\begin{itemize}
		\item factors through the $U_p$-ordinary projection $H_1(X_1(N),\zp)_+ \to H_1(X_1(N),\zp)_+^{\ord}$ and
		\item satisfies $z_N^{\ord}(V_n\vec\gamma) = \varepsilon \cdot {}_n\Theta_{N,\gamma}$ for all primes $n \nmid N$
		and $\gamma \in \Gamma_1(N)$ with $\vec\gamma \in H_1(X_1(N),\zp)^{\ord}$.
	\end{itemize}
	It is Hecke equivariant for the prime-to-level Hecke operators in the sense of Theorem \ref{zeta_map}.
\end{proposition}

\begin{proof}
	The proof mirrors that of Proposition \ref{Vgenerate}. 
	Consider the $\zp$-algebra of endomorphisms $\mb{T}_M$ generated 
	by the Hecke operators of $T_{\ell}$ for $\ell \nmid N$, $U_{\ell}^*$ for $\ell \mid N$, and $[d]'$ for $d \in (\Z/N\Z)^{\times}$ acting on 
	$H^1(\bGamma_1(N),M) = H^1(\Gamma_1(N),M)_+$. The $U_p^*$-ordinary part $\mb{T}_M^{\ord}$ of this Hecke algebra 
	acts on $H^1(\bGamma_1(N),M)^{\ord} \cong \Hom(H_1(Y_1(N),\Z)_+^{\ord},M)$.
	
	First, we note this Hecke algebra $\mb{T}_M^{\ord}$  
	is a quotient of the Hecke algebra $\mb{T}^{\ord}$ for $U_p$-ordinary modular forms of weight $2$ for 
	$\Gamma_1(N)$ that is generated by these operators,
	by a map taking an operator to its adjoint, i.e., via the map that sends $T_{\ell}^*$  
		to $T_{\ell}$, $U_{\ell}$ to $U_{\ell}^*$, and $\langle d \rangle^{-1}$ to $[d]'$ (see \S \ref{eichshim}).
	For the direct summand 
	$\varepsilon \cdot H^1(\bGamma_1(N), M)^{\ord} = H^1(\bGamma_1(N),\varepsilon \cdot M)^{\ord}$, 
	the corresponding Hecke algebra is a quotient of $\varepsilon \cdot \mb{T}^{\ord}$.

	We claim that the operators $V_{\ell}^* = 
	\ell( \ell^3 - \ell\langle \ell \rangle^{-1}T_{\ell} + \langle \ell \rangle^{-1})$ generate $\varepsilon \cdot \mb{T}^{\ord}$,
	which will tell us that the operators $V_{\ell}$ generate $\mb{T}_M^{\ord}$.
	Suppose they do not. We then have a nonzero homomorphism $\phi \colon \varepsilon \cdot \mb{T}^{\ord} \to F$ 
	to an algebraically closed field $F$ of characteristic $p$ such that $V_{\ell}^* \in \ker \phi$ for all $\ell \nmid N$. 
	
	Let $N'$ be the prime-to-$p$ part of $N$.  Hida theory (see \cite[Theorem 1.2]{hida})
	provides a $U_p$-ordinary eigenform $f$ in $M_2(\Gamma_1(N'p),F)^{\ord}$ such that $\phi(T_{\ell})$ for
	$\ell \nmid N$ and $\phi(U_{\ell})$ for $\ell \mid N$ is its $\ell$th Fourier coefficient $a_{\ell}(f) \in F$.
	Let $\omega^j$ for $1 \le j \le p-1$ be the restriction of the 
	Nebentypus of $f$ to $(\Z/p\Z)^{\times}$,
	where $\omega$ denotes the mod $p$ cyclotomic character.
	A result of Ohta \cite[Proposition 1.3.5]{ohta} implies that
	$f$ arises from an eigenform $f'$ in the $T_p$-ordinary part of $M_{j+2}(\Gamma_1(N'),F)$
	with $a_{\ell}(f) = a_{\ell}(f')$ for $\ell \neq p$.
	
	As in the proof of Proposition \ref{Vgenerate}, to $f'$ we may associate a semisimple Galois representation 
	$\rho \colon G_{\Q} \to \GL_2(F)$ satisfying 
	$\rho|_{I_p} \simeq \omega^{j+1} \oplus 1$ (again by \cite[Theorem 2.5]{edixhoven}).
	On the other hand, since
	$\ell T_{\ell} - 1 - \ell^3 \langle \ell \rangle \in \ker \phi$
	for all $\ell \nmid Np$ by assumption, we must have
	$$  
		\ell \Tr \rho(\varphi_{\ell}) = 1 + \ell^{j+3}\chi(\ell) \in F
	$$
	for some $F$-valued Dirichlet character $\chi$ of modulus $N'$,
	where $\varphi_{\ell}$ denotes the Frobenius at $\ell$. By \v{C}ebotarev density, we then have
	$\rho \simeq \omega^{-1} \oplus \omega^{j+2}\chi$.  
	This in turn forces $j = -2$, but the $\omega^{-2}|_{I_p}$-eigenspace
	 of $\varepsilon \cdot \mb{T}^{\ord}$ is trivial. Thus, we have the necessary 
	contradiction.
	
	Since the operators $V_{\ell}$ generate $\mb{T}_M^{\ord}$, as in the proof of Theorem \ref{canoncocyc} we may 
	construct a $U_p^*$-ordinary parabolic cocycle $\Theta_N^{\ord} \colon \bGamma_1(N) \to \varepsilon \cdot M$
	as a $\mb{T}_M$-linear combination of the restrictions of the cocycles ${}_n \Theta_N$ to $\bGamma_1(N)$,
	where the coefficients sum to Hida's ordinary idempotent in $\mb{T}_M$.
	The class of $V_n \Theta_N^{\ord}$ for a prime $n \nmid N$ is the ordinary projection of the class of ${}_n \Theta_N$.
	This in turn gives rise to the homomorphism $z_N^{\ord}$ in the statement of the proposition. In particular, note that 
	its image lands in $H^2(X_1(N),\zp(2))$ via the argument of Proposition \ref{Gamma1X}.
\end{proof}

We remark that we do not show that $z_N^{\ord}$ is equivariant for $p$th Hecke operators, as prior to this point we only considered 
prime-to-level operators on our cocycles.  It would be interesting to prove this.
Passing to \'etale cohomology, the explicit formula for $\Theta_{N,\gamma}$ 
of Theorem  \ref{expformE2} holds in $H^2_{\et}(Y,\qp(2))$ without ambiguity, since $\mc{V}$ vanishes there.  From this, we see that the $\qp$-linear extension of the $p$-adic \'etale realization of our ordinary zeta map $z_N^{\ord}$ induces $\varepsilon$ applied to the restriction of the zeta map $z_{N,\et}^{\ord}$ of Fukaya-Kato to $H_1(X_1(N),\zp)$. This is Hecke-equivariant for the full Hecke algebra by Theorem \ref{FKzeta}.

\end{document}